\documentclass{amsart}
\usepackage{graphicx}
\usepackage{amsmath}
\usepackage{amscd}
\usepackage{mathtools}
\usepackage{orcidlink}
\usepackage{enumerate}
\usepackage{tikz-cd}
\usepackage[all,cmtip]{xy}
\usepackage{tikz}
\usepackage{amssymb}
\usepackage{ascmac} 
\usepackage{color}
\usepackage{autobreak}
\usepackage{breqn}
\usepackage{array}
\usepackage{makecell}
\usepackage{longtable}
\usepackage{docmute}

\newtheorem{dfn}{Definition}[section]
\newtheorem{thm}[dfn]{Theorem}
\newtheorem{prop}[dfn]{Proposition}
\newtheorem{lem}[dfn]{Lemma}

\theoremstyle{definition}

\newcommand{\ol}{\overline}
\newcommand{\Sing}{\mathrm{Sing}}
\newcommand{\Spec}{\mathrm{Spec}\,}
\newcommand{\length}{\mathrm{length}}

\newcommand{\Div}{\mathrm{Div}}

\title[Singular fibers of elliptic fibrations]{Singular fibers of elliptic fibrations on normal $K3$ surfaces}

\author{Taro Hayashi}
\author{Kazuki Utsumi}
\address{(Taro Hayashi)
Department of Mathematical Sciences,
Ritsumeikan University,
1$-$1$-$1 Nojihigashi, Kusatsu, Shiga, 525$-$8577, Japan}
\email{haya4taro@gmail.com}
\address{(Kazuki Utsumi)
College of Science and Engineering, Ritsumeikan University,
1$-$1$-$1 Nojihigashi, Kusatsu, Shiga, 525$-$8577, Japan}
\email{kutsumi@fc.ritsumei.ac.jp}

\date{\today}

\begin{document}
\maketitle
\begin{abstract}
We study singular fibers of elliptic fibrations on normal \(K3\) surfaces. 
For a normal \(K3\) surface \(Y\) with minimal resolution \(\nu \colon X \to Y\), we describe singular fibers on \(Y\) in terms of contractions of suitable ADE configurations of \((-2)\)-curves in singular fibers on \(X\). 
We determine ADE configurations occurring in singular fibers and describe fibers obtained after contraction. 
As a consequence, we obtain a description of singular fibers of elliptic fibrations on normal \(K3\) surfaces.
\end{abstract}

\section{Introduction}

Throughout this paper, we work over an algebraically closed field of
characteristic \(0\).
The theory of elliptic fibrations on smooth surfaces is governed, in a
fundamental way, by Kodaira's classification of singular fibers
\cite{k63}; see also \cite{schs10} for a modern account of elliptic
surfaces.  In the case of a smooth \(K3\) surface, every singular fiber
of a relatively minimal elliptic fibration is one of the Kodaira fibers
\[
I_n\ (1 \leq n \leq 19),\quad I_m^*\ (0 \leq m\leq 14),\quad
II,\ III,\ IV,\ II^*,\ III^*,\ IV^* .
\]
Singular fibers of elliptic fibrations on smooth \(K3\) surfaces have
also been studied systematically.  For example, Miranda and Persson
studied configurations of \(I_n\)-fibers on elliptic \(K3\) surfaces
\cite{mp89}, Shimada studied the possible ADE-types of singular fibers
of complex Jacobian elliptic \(K3\) surfaces \cite{s00}, and Shimada and
Zhang classified extremal elliptic \(K3\) surfaces \cite{sz01}.
Elliptic \(K3\) surfaces with large or maximal singular fibers were
studied in \cite{s03,sch07,schs13}.
There is also a closely related line of work in which one fixes a special
\(K3\) surface, or a special class of \(K3\) surfaces, and determines
the elliptic fibrations and the singular fibers that occur on them.  This
includes Jacobian fibrations on Kummer surfaces and related \(K3\)
surfaces \cite{o89,n96,ku14,ku18,u23}, and elliptic fibrations on
singular or special \(K3\) surfaces \cite{k06,u12,be15,u16,ba18,gs19,gs20,c25}.

In this paper, we study elliptic fibrations on normal \(K3\) surfaces.
Normal \(K3\) surfaces have been studied from several viewpoints; see,
for example, Shimada \cite{s07}.  They also appear naturally in the
study of smooth loci and universal covers of open \(K3\) surfaces
\cite{cko03}.  The first author studied normal \(K3\) surfaces in
connection with automorphisms of \(K3\) surfaces, where their rational
double points were also analyzed
\cite[Theorems~1.1 and~1.2]{h24}, \cite[Theorem~1.3]{h26}.
These results indicate that, on a normal \(K3\) surface, the rational
double points should be treated as part of the geometry, rather than as
auxiliary data.  This motivates us to describe singular fibers not only
by their surviving components, but also by the positions and ADE types
of the rational double points lying on them.
However, singular fibers of elliptic fibrations on normal \(K3\)
surfaces do not seem to have been studied systematically.  The purpose
of this paper is to initiate such a study.

Let \(Y\) be a normal \(K3\) surface and let
\[
\nu\colon X\rightarrow Y
\]
be its minimal resolution.  Then \(X\) is a smooth \(K3\) surface.  The
singularities of \(Y\) are rational double points; equivalently, the
exceptional curves on the minimal resolution form configurations whose
dual graphs are Dynkin diagrams of type \(A_k\), \(D_l\), or \(E_m\)
\cite{a66} where $1\leq k$, $4\leq l$, and $m=6,7,8$.  Conversely, by Artin's contractibility criterion
\cite{a62}, negative definite configurations of smooth rational curves
can be contracted under suitable hypotheses.

Suppose that
\[
f\colon Y\rightarrow \mathbb P^1
\]
is an elliptic fibration.  Then the composition
\[
g:=f\circ \nu\colon X\rightarrow \mathbb P^1
\]
is an elliptic fibration on the smooth \(K3\) surface \(X\).  Moreover,
every exceptional curve of \(\nu\) is contained in a singular fiber of
\(g\).  Thus the study of singular fibers of \(f\) is reduced to the
following problem: given a Kodaira fiber \(F\) of an elliptic fibration
on a smooth \(K3\) surface and a reduced subdivisor \(D\subset F\) whose
dual graph is of ADE type, describe the fiber obtained after contracting
\(D\).
We call such a divisor \(D\) an ADE reduced divisor, abbreviated as an ADER. If
\[
\nu\colon X\rightarrow Y
\]
is the contraction of \(D\), then each connected component of \(D\)
gives a rational double point on \(Y\) of the corresponding ADE type.
The remaining irreducible components of \(F\) survive on \(Y\), and their
multiplicities in the induced fiber are the same as their multiplicities
in \(F\).  Hence the main task is to determine which ADE configurations
can occur inside each Kodaira fiber and to describe how the surviving
components meet at the newly created rational double points.

The first basic observation is that singular fibers on \(Y\) are
controlled by singular fibers on the minimal resolution.  More precisely,
if \(f\colon Y\to \mathbb P^1\) is an elliptic fibration and
\(g=f\circ\nu\), then the exceptional curves of \(\nu\) are contained in
singular fibers of \(g\); see Proposition~\ref{2.1}.  Moreover,
\(f^*(s)\) is a singular fiber if and only if \(g^*(s)\) is a singular
fiber; see Lemma~\ref{2.3}.  Conversely, contracting an ADE configuration
contained in a singular fiber of \(g\) produces a normal \(K3\) surface
together with an induced elliptic fibration.  We also prove that, after
such a contraction, the irreducible components and multiplicities of the
fiber on \(Y\) are obtained simply by deleting the contracted components
from the fiber on \(X\); see Theorem~\ref{2.4}.

We then analyze each Kodaira type.  Kodaira fibers of type \(I_1\) and \(II\)
contain no ADE configuration of smooth rational components, and hence no
contraction occurs.  For Kodaira fibers of type \(I_2\), the only possible
connected ADER is of type \(A_1\), and its contraction transforms the
fiber into an irreducible nodal rational curve, which has underlying
Kodaira type \(I_1\); see
Propositions~\ref{prop:I2-ADER} and~\ref{prop:I2-contraction}.  For
Kodaira fibers of type \(III\), the only possible connected ADER is again of
type \(A_1\), and its contraction transforms the fiber into an
irreducible cuspidal rational curve, which has underlying Kodaira type
\(II\); see
Propositions~\ref{prop:III-ADER} and~\ref{prop:III-contraction}.
Similarly, for a Kodaira fiber of type \(IV\), the possible connected ADERs are
of types \(A_1\) and \(A_2\), and the corresponding contractions yield
the two types described in Propositions~\ref{prop:ADE-IV}
and~\ref{prop:IV-contraction}.

For fibers of type \(I_n\) with \(n\geq 3\), we show that every connected
component of an ADER is of type \(A\).  More precisely, the possible
ADERs are exactly
\[
A_{\lambda_1}\oplus\cdots\oplus A_{\lambda_r}
\]
with
\[
\lambda_1+\cdots+\lambda_r+r\leq n.
\]
Here the symbol \(\oplus\) means the disjoint union of Dynkin diagrams.
After contracting such a divisor, the resulting fiber has
\(n-\sum_i\lambda_i\) surviving components.  Put
\[
m:=n-\sum_i\lambda_i.
\]
If \(m\geq 2\), the surviving components form a cycle, and the resulting
fiber has underlying Kodaira type \(I_m\).  If \(m=1\), the resulting fiber
is an irreducible nodal rational curve, and hence has underlying Kodaira
type \(I_1\); see
Propositions~\ref{prop:ADE-In-fiber} and~\ref{prop:In-fiber-contraction}.

The cases of type \(I_0^*\) and \(I_n^*\) require a more detailed
analysis.  For \(I_0^*\), the possible ADE configurations are
\[
D_4,\qquad A_3,\qquad A_2,\qquad A_1^{\oplus s}
\quad (1\leq s\leq 4),
\]
and the resulting fibers are described in
Propositions~\ref{prop:ADE-I0star} and~\ref{prop:I0star-contraction}.
Some of the new fiber types appearing in this case are introduced in
Definitions~\ref{dfn:IV-contracted-fibers} and~\ref{dfn:new-I0star-fibers}.

For \(I_n^*\) with \(n\geq 1\), we divide the classification according
to whether the two end components of the central chain are contained in
the ADER.  This leads to several families of decorated fibers, denoted
by symbols such as
\[
\mathfrak D_m^{\mathrm{mark},1}(\Gamma),\quad
\mathfrak A_m^{\mathrm{mark},1,1}(\Gamma_L,\Gamma_R),\quad
\mathfrak I_m^{*,\mathrm{fork},\varepsilon}(\Gamma),\quad
\mathfrak A_m^{\tau_L,\tau_R}(\Gamma_L,\Gamma_R),
\]
and
\[
\mathfrak S_t(\Gamma).
\]
The precise definitions of these fiber types are given in
Definitions~\ref{dfn:new-I0star-fibers},
\ref{dfn:marked-fibers}, \ref{dfn:left-end-marked-fibers}, and
\ref{dfn:both-end-marked-fibers}.  The three cases for \(I_n^*\) are
treated in Propositions~\ref{prop:Instar-no-end},
\ref{prop:Instar-left-end}, and~\ref{prop:Instar-both-ends}, and the
corresponding contracted fibers are described in
Propositions~\ref{prop:Instar-no-end-contraction},
\ref{prop:Instar-left-end-contraction}, and
\ref{prop:Instar-both-ends-contraction}.  These symbols record not only
the weighted configuration of surviving components but also the
positions and ADE types of the rational double points created by the
contraction.

Finally, we treat the additive fibers
\[
II^*,\qquad III^*,\qquad IV^*.
\]
For each of these types, we give a complete classification of ADERs
contained in the fiber.  The case \(II^*\) is treated in
Propositions~\ref{prop:ADE-IIstar-no-C}, \ref{prop:ADE-IIstar-C},
and~\ref{prop:IIstar-contraction}; the case \(III^*\) is treated in
Propositions~\ref{prop:ADE-IIIstar-no-C}, \ref{prop:ADE-IIIstar-C},
and~\ref{prop:IIIstar-contraction}; and the case \(IV^*\) is treated in
Propositions~\ref{prop:ADE-IVstar-no-branch},
\ref{prop:ADE-IVstar-with-branch}, and~\ref{prop:IVstar-contraction}.
The corresponding contracted fibers are denoted by the new symbols
introduced in Definitions~\ref{dfn:IIstar-contracted-fibers},
\ref{dfn:IIIstar-contracted-fibers}, and
\ref{dfn:IVstar-contracted-fibers}.  These definitions describe the
resulting fibers by deleting the contracted components from the original
Kodaira fiber, preserving the multiplicities of the remaining components,
and marking the rational double points produced by the connected
components of the ADER.  The explicit lists are given in
Appendices~\ref{app:IIstar-no-branch-list},
\ref{app:IIstar-with-branch-list},
\ref{app:IIIstar-no-branch-list},
\ref{app:IIIstar-with-branch-list},
\ref{app:IVstar-no-branch-list}, and
\ref{app:IVstar-with-branch-list}.

The paper is organized as follows.  In Section~2 we collect basic facts
on normal \(K3\) surfaces, elliptic fibrations, contractions of ADE
configurations, and local criteria for the smoothness, nodality, and
cuspidality of curves on normal surfaces.  In Section~5 we classify the
ADE configurations contained in each Kodaira fiber and describe the
singular fibers obtained after their contractions.  This gives a
systematic description of singular fibers of elliptic fibrations on
normal \(K3\) surfaces in terms of Kodaira fibers on their minimal
resolutions.

\section{Preliminaries}

In this section, we collect some basic facts that will be used throughout the
paper.  We first recall the relation between an elliptic fibration on a normal
\(K3\) surface and the induced elliptic fibration on its minimal resolution.
We then review the \(\delta\)-invariant of curves and a simple criterion for
deciding whether an irreducible curve obtained by contraction is smooth, nodal,
or cuspidal at a rational double point.
\subsection{Fibers on the minimal resolution}
Let $Y$ be a projective surface, possibly with singularities. An \emph{elliptic fibration} on $Y$ is a surjective morphism $f \colon Y \to C$
to a smooth projective curve $C$ such that the general fiber of $f$ is a smooth curve of genus one.
For any \(t\in C\), the fiber
\[
F_t:=f^*(t)
\]
is a Cartier divisor on \(Y\). 
Since any two fibers are linearly equivalent and
distinct fibers are disjoint, one has
\[
F_t^2=0.
\]
A fiber $F_t$ over a point $t \in C$ is called a \emph{singular fiber} if $F_t$ is not a smooth curve of genus one.
If moreover there exists a section
$\sigma \colon C \to Y$
satisfying $f \circ \sigma = \mathrm{id}_C$, then $f$ is called a \emph{Jacobian fibration} (or a \emph{Jacobian elliptic fibration}).

For an effective divisor
\[
D=\sum_i m_iD_i
\]
on a surface, where the \(D_i\) are distinct irreducible curves and
\(m_i>0\), we denote by
\[
\operatorname{Supp}D:=\bigcup_i D_i
\]
the support of \(D\), regarded as a reduced closed subset.
\begin{prop}\label{2.1}
Let $Y$ be a normal $K3$ surface admitting an elliptic fibration $f \colon Y \to \mathbb{P}^1$, and let
$\nu \colon X \to Y$
be its minimal resolution. Denote by
$g := f \circ \nu \colon X \to \mathbb{P}^1$
the induced elliptic fibration on $X$.
Then $Y$ is obtained from $X$ by contracting certain ADE configurations of $(-2)$-curves contained in singular fibers of $g$. 
\end{prop}
\begin{proof}
Let $y\in Y$ be a singular point of $Y$, and put  $t:=f(y)\in\mathbb P^1$.
Since $g=f\circ \nu$, we have $\nu^{-1}(y)\subset g^{-1}(t)$.
Since $y$ is a rational double point, $\nu^{-1}(y)$ is an
ADE configuration of smooth rational $(-2)$-curves, and hence $g^{*}(t)$ is a singular fiber of $g$.
\end{proof}
\begin{lem}\label{2.2}
Let \(f \colon Y \to \mathbb{P}^1\) be an elliptic fibration on a normal
\(K3\) surface \(Y\), and let \(\nu \colon X \to Y\) be the minimal resolution.
We put $g:=f\circ\nu \colon X\to \mathbb P^1$.
Let \(s\in\mathbb P^1\).  If
\[
\operatorname{Supp} f^*(s)\cap \Sing(Y)=\emptyset,
\]
then the pullback by \(\nu\) identifies the Cartier divisor \(f^*(s)\) with
the Cartier divisor \(g^*(s)\).
\end{lem}
\begin{proof}
Since \(\operatorname{Supp} f^*(s)\) does not meet \(\Sing(Y)\), the Cartier
divisor \(f^*(s)\) is contained in the smooth locus of \(Y\).  The morphism
\(\nu\) is an isomorphism over the smooth locus of \(Y\).  Therefore the
pullback of the Cartier divisor \(f^*(s)\) by \(\nu\) is identified with
\(f^*(s)\).  Since
\[
\nu^*f^*(s)=(f\circ\nu)^*(s)=g^*(s),
\]
the pullback by \(\nu\) identifies \(f^*(s)\) with \(g^*(s)\) as Cartier
divisors.
\end{proof}

\begin{lem}\label{2.3}
Let \(f \colon Y \to \mathbb{P}^1\) be an elliptic fibration on a normal
\(K3\) surface \(Y\), and let \(\nu \colon X \to Y\) be the minimal resolution.
We put $g:=f\circ\nu \colon X\to \mathbb P^1$.
Let \(s\in\mathbb P^1\).  Then \(f^*(s)\) is a singular fiber if and only if
\(g^*(s)\) is a singular fiber.  In particular, if
\[
\operatorname{Supp} f^*(s)\cap \Sing(Y)\neq\emptyset,
\]
then \(f^*(s)\) is a singular fiber.
\end{lem}

\begin{proof}
First, we assume that $\operatorname{Supp} f^*(s)\cap \,\Sing(Y)=\emptyset$.
By Lemma~\ref{2.2}, the pullback by \(\nu\) identifies \(f^*(s)\) with
\(g^*(s)\) as Cartier divisors.  Hence \(f^*(s)\) is a singular fiber if and
only if \(g^*(s)\) is a singular fiber.

Next, we assume that $\operatorname{Supp} f^*(s)\cap \,\Sing(Y)\neq\emptyset$.
Then the Cartier divisor $g^*(s)=\nu^*f^*(s)$
contains exceptional curves over the singular points of \(Y\) lying on
\(\operatorname{Supp} f^*(s)\).  Hence \(g^*(s)\) is a singular fiber.
We assume that \(f^*(s)\) is a smooth curve of genus one.  Then \(g^*(s)\)
contains the strict transform of this smooth genus-one curve, together with
exceptional \((-2)\)-curves over the singular points of \(Y\) lying on
\(\operatorname{Supp} f^*(s)\).  In particular, \(g^*(s)\) has an irreducible
component of genus one. Since \(X\) is a smooth \(K3\) surface, the
elliptic fibration \(g\) is relatively minimal and has no multiple smooth
fibers.  Hence, by Kodaira's classification, every irreducible component of a
singular fiber of \(g\) is a rational curve.  Thus \(f^*(s)\) is a singular
fiber.
\end{proof}

\begin{thm}\label{2.4}
Let \(f \colon Y \to \mathbb P^1\) be an elliptic fibration on a normal
\(K3\) surface \(Y\), and let $\nu \colon X \to Y$
be the minimal resolution.  We put $g:=f\circ \nu \colon X \to \mathbb P^1$.
Let \(s\in\mathbb P^1\), and assume that
\[
\operatorname{Supp} f^*(s)\cap \Sing(Y)\neq\emptyset.
\]
We write
\[
g^*(s)=\sum_{i=1}^m a_iG_i+\sum_{j=1}^n b_jE_j,
\]
where \(G_1,\dots,G_m\) are the irreducible components of \(g^*(s)\) which are
not contracted by \(\nu\), and \(E_1,\dots,E_n\) are the irreducible components
of \(g^*(s)\) contracted by \(\nu\).
For each \(i=1,\dots,m\), we put $\overline{G_i}:=\nu(G_i)$,
regarded as an irreducible curve on \(Y\) with the reduced structure.  Then
\[
f^*(s)=\sum_{i=1}^m a_i\,\overline{G_i}.
\]
In other words, the fiber \(f^*(s)\) is obtained from \(g^*(s)\) by deleting
the components contracted by \(\nu\), while preserving the multiplicities of
the remaining components.
\end{thm}

\begin{proof}
We write
\[
f^*(s)=\sum_{k=1}^r c_kF_k
\]
as the decomposition of \(f^*(s)\) into irreducible components with
multiplicities.  Let \(\widetilde F_k\) be the strict transform of \(F_k\) on
\(X\).

By the functoriality of pullbacks of Cartier divisors, we have
\[
g^*(s)=(f\circ\nu)^*(s)=\nu^*f^*(s).
\]
Since \(\nu\) is an isomorphism at the generic point of \(F_k\), the coefficient
of \(\widetilde F_k\) in \(\nu^*f^*(s)\) is equal to the coefficient of \(F_k\)
in \(f^*(s)\), namely \(c_k\).  Hence the coefficient of \(\widetilde F_k\) in
\(g^*(s)\) is equal to \(c_k\).

The irreducible components of \(g^*(s)\) which are not contracted by \(\nu\)
are precisely the strict transforms
\[
\widetilde F_1,\dots,\widetilde F_r.
\]
Thus \(m=r\), and after reindexing we may write
\[
G_i=\widetilde F_i
\qquad (i=1,\dots,m).
\]
Then
\[
a_i=c_i
\qquad (i=1,\dots,m).
\]
Since $\overline{G_i}=\nu(G_i)=F_i$,
we obtain $f^*(s)=\sum_{i=1}^m a_i\,\overline{G_i}$.
\end{proof}
\subsection{Curves and normalization}
Let \(C\) be a projective integral curve, and let $\rho\colon \widetilde C\rightarrow C$
be its normalization.  We set
\[
Q:=\operatorname{coker}
\left(
\rho^\#\colon\mathcal O_C\to \rho_*\mathcal O_{\widetilde C}
\right).
\]
For a point \(x\in C\), the \(\delta\)-invariant of \(C\) at \(x\) is defined by
\[
\delta_x(C)
:=
\length_{\mathcal O_{C,x}} Q_x.
\]
By the exact sequence $0\to \mathcal O_C \to \rho_*\mathcal O_{\widetilde C}\to Q\to 0$,
we have
\[
p_a(C)-g(\widetilde C)
=
\sum_{x\in \Sing(C)}\delta_x(C).
\]

Let \(Y\) be a normal \(K3\) surface, and let \(C\subset Y\) be an integral Cartier curve.  Since \(Y\) is Gorenstein and \(K_Y\sim 0\),
adjunction gives
\[
2p_a(C)-2=C^2.
\]
Hence
\[
\sum_{x\in \Sing(C)}\delta_x(C)
=
1+\frac{C^2}{2}-g(\widetilde C).
\]
In particular, if \(\widetilde C\cong \mathbb P^1\), then
\[
\sum_{x\in \Sing(C)}\delta_x(C)
=
1+\frac{C^2}{2}.
\]

We also use the following elementary observation.  Let $\pi\colon X\rightarrow Y$
be a resolution of a normal surface \(Y\).  Let \(C\subset Y\) be an irreducible
curve, and let \(D\subset X\) be its strict transform.
We assume that the
restriction
\[
\rho:=\pi|_D\colon D\rightarrow C
\]
is the normalization of \(C\).
Let \(y\in C\).  Since \(\rho=\pi|_D\), we have a natural identification of
schemes
\[
\rho^{-1}(y)
\cong
D\times_X \pi^{-1}(y).
\]
Let \(\mathcal I_D\subset \mathcal O_X\) and
\(\mathcal I_{\pi^{-1}(y)}\subset \mathcal O_X\) be the ideal sheaves defining
\(D\) and the scheme-theoretic fiber \(\pi^{-1}(y)\), respectively.  Then, for
a point
\[
x\in D\cap \pi^{-1}(y),
\]
we have
\[
\mathcal O_{\rho^{-1}(y),x}
\cong
\mathcal O_{X,x}/
\bigl(
\mathcal I_{D,x}+\mathcal I_{\pi^{-1}(y),x}
\bigr).
\]
Consequently,
\[
\length_{\mathcal O_{X,x}}
\mathcal O_{\rho^{-1}(y),x}
=
\length_{\mathcal O_{X,x}}
\frac{\mathcal O_{X,x}}
{\mathcal I_{D,x}+\mathcal I_{\pi^{-1}(y),x}}
=
(\pi^{-1}(y)\cdot D)_x.
\]
Here the right-hand side denotes the local intersection multiplicity on
the smooth surface \(X\).
\begin{prop}\label{prop:smoothness-criterion}
Let \(Y\) be a normal \(K3\) surface, and let $\pi \colon X \to Y$
be its minimal resolution. Then \(X\) is a smooth \(K3\) surface.
Let \(C\subset Y\) be an irreducible curve, and let \(D\subset X\) be its strict transform.
We assume that the restriction
\[
\rho:=\pi|_D \colon D \to C
\]
is the normalization of \(C\), and that
there is a point \(y\in C\cap\Sing(Y)\) such that
\[
D\cap \pi^{-1}(y)=\{x\}
\qquad \mathrm{and}
\qquad
(\pi^{-1}(y)\cdot D)_x=1.
\]
Then \(C\) is smooth at \(y\).
\end{prop}
\begin{proof}
Let $y:=\pi(x)\in \Sing(Y)$.
We set
\[
A:=\mathcal O_{C,y},\qquad B:=\mathcal O_{D,x}.
\]
Since \(D\) is a smooth curve, \(B\) is a regular local ring of dimension one.

We choose an affine open neighbourhood \(U=\operatorname{Spec} R\) of \(y\) in \(C\).
Since \(\rho\) is finite, we may write
\[
\rho^{-1}(U)=\operatorname{Spec} S
\]
for some finite \(R\)-algebra \(S\).
Let \(\mathfrak p\subset R\) and \(\mathfrak q\subset S\) be the prime ideals corresponding
to \(y\) and \(x\), respectively. Then
\[
A=R_{\mathfrak p},\qquad B=S_{\mathfrak q}.
\]
The scheme-theoretic fiber of \(\rho\) over \(y\) is
\[
\rho^{-1}(y)\cong \Spec(S\otimes_R \kappa(y)).
\]
Since \(\rho^{-1}(y)=\{x\}\),
\(\mathfrak q\) is the unique
prime ideal of \(S\) lying over \(\mathfrak p\). It follows that \(S_{\mathfrak p}\) is a local
ring with maximal ideal \(\mathfrak qS_{\mathfrak p}\), and therefore
\[
S_{\mathfrak p}=S_{\mathfrak q}=B.
\]
Consequently,
\[
\rho^{-1}(y)\cong \Spec(S_{\mathfrak p}/\mathfrak pS_{\mathfrak p})= \Spec(B/\mathfrak m_A B),
\]
where \(\mathfrak m_A=\mathfrak pA\) is the maximal ideal of \(A\).

By the observation above, $(\pi^{-1}(y)\cdot D)_x=1$ implies that
\[
\length_{\mathcal O_{X,x}}
\rho^{-1}(y)=1.
\]
Therefore \(\rho^{-1}(y)\) is a zero-dimensional local \(k\)-scheme of length one.
Hence it is reduced and isomorphic to \(\Spec k\). Thus
\[
\rho^{-1}(y)\cong \Spec k.
\]
Therefore
\[
B/\mathfrak m_A B\cong k.
\]
Since \(B/\mathfrak m_A B\cong k\), Nakayama's lemma implies that \(B\) is generated
by one element as an \(A\)-module. Since \(1\in B\), this forces \(A=B\). Therefore
\[
\mathcal O_{C,y}\cong \mathcal O_{D,x}.
\]
Since \(B\) is a regular local ring of dimension one, the same is true for \(A\).
Hence \(C\) is smooth at \(y\).
\end{proof}
\begin{prop}\label{prop:delta-one-node-cusp}
Let \(Y\) be a normal \(K3\) surface, and let $\pi \colon X \to Y$
be its minimal resolution. Then \(X\) is a smooth \(K3\) surface.
Let \(C\subset Y\) be an irreducible curve, and let \(D\subset X\) be its strict transform.
We assume that the restriction
\[
\rho:=\pi|_D \colon D \to C
\]
is the normalization of \(C\), and that there exists a point
\[
y\in C\cap\Sing(Y)
\]
such that
\[
\delta_y(C)=1.
\]
Then the following hold:
\begin{enumerate}[$(1)$]
\item If $D\cap \pi^{-1}(y)=\{x_1,x_2\}$ with $(\pi^{-1}(y)\cdot D)_{x_i}=1$ for $i=1,2$,
then \(C\) has a node at \(y\).
\item If $D\cap \pi^{-1}(y)=\{x\}$
with $(\pi^{-1}(y)\cdot D)_x=2$,
then \(C\) has a cusp at \(y\).
\end{enumerate}
\end{prop}
\begin{proof}
We set
\[
A:=\mathcal O_{C,y},
\qquad
B:=(\rho_*\mathcal O_D)_y.
\]
Since \(\delta_y(C)=1\), we have 
\[\length_A(B/A)=1.\]

We show \textup{(1)}.  We assume that $D\cap \pi^{-1}(y)=\{x_1,x_2\}$
and $(\pi^{-1}(y)\cdot D)_{x_i}=1$ for $i=1,2$.
By the observation above, the scheme-theoretic fiber \(\rho^{-1}(y)\) is
identified with $D\times_X \pi^{-1}(y)$.
As in the proof of Proposition~\ref{prop:smoothness-criterion}, the assumptions imply that
\[
\rho^{-1}(y)\cong \Spec k\sqcup \Spec k.
\]

Let $\widehat A$
be the \(\mathfrak m_A\)-adic completion of \(A\), and set $\widehat B:=B\otimes_A \widehat A$. Since \(B\) is finite over \(A\), completion is exact, and
\[
\length_{\widehat A}(\widehat B/\widehat A)
=
\length_A(B/A)
=
1.
\]
Since \(B\) is the semilocal normalization of \(A\) and
$\rho^{-1}(y)=\{x_1,x_2\}$,
the completion of \(B\) decomposes as
\[
\widehat B
\cong
\widehat{\mathcal O}_{D,x_1}
\times
\widehat{\mathcal O}_{D,x_2}.
\]
Since \(D\) is smooth at \(x_1\) and \(x_2\), we may choose local parameters
\(u\) and \(v\) such that
\[
\widehat{\mathcal O}_{D,x_1}\cong k[[u]],
\qquad
\widehat{\mathcal O}_{D,x_2}\cong k[[v]].
\]
Hence
\[
\widehat B
\cong
k[[u]]\times k[[v]].
\]
The homomorphism $\rho^\#\colon A\rightarrow B$
induces $\widehat{\rho^\#}\colon \widehat A\rightarrow \widehat B
\cong k[[u]]\times k[[v]]$.
Since the two residue maps to \(k\) induced by the two branches both factor
through \(A/\mathfrak m_A\),
the image of \(\widehat A\) is contained in
\[
R:=
\{(f(u),g(v))\in k[[u]]\times k[[v]]
\mid f(0)=g(0)\}.
\]
On the other hand,
\[
\widehat B/R\cong k.
\]
Since the action of \(\widehat A\) on this quotient factors through
\(\widehat A/\mathfrak m_{\widehat A}\cong k\), we have
\[
\length_{\widehat A}(\widehat B/R)=1.
\]
Since $\widehat A\subset R\subset \widehat B$
and $\length_{\widehat A}(\widehat B/\widehat A)=1$,
we obtain
\[
\widehat A=R.
\]
Therefore
\[
\widehat{\mathcal O}_{C,y}
\cong
\{(f(u),g(v))\in k[[u]]\times k[[v]]
\mid f(0)=g(0)\}.
\]
Thus the completed local ring of \(C\) at \(y\) is isomorphic to
\[
k[[U,V]]/(UV).
\]
Hence \(C\) has an ordinary node at \(y\).

We show \textup{(2)}.  We assume that
$D\cap \pi^{-1}(y)=\{x\}$ and $(\pi^{-1}(y)\cdot D)_x=2$.
In this case, $B\cong \mathcal O_{D,x}$.
Since \(D\) is smooth at \(x\), we may choose a local parameter \(t\) such that
\[
\widehat B\cong k[[t]].
\]
By the observation above, the scheme-theoretic fiber \(\rho^{-1}(y)\) is
identified with $D\times_X \pi^{-1}(y)$.
As in the proof of Proposition~\ref{prop:smoothness-criterion}, the assumption $(\pi^{-1}(y)\cdot D)_x=2$ gives
\[
\dim_k \widehat B/\mathfrak m_{\widehat A}\widehat B=2.
\]
Since \(\widehat B\cong k[[t]]\), this means
\[
\mathfrak m_{\widehat A}\widehat B=(t^2).
\]
Moreover,
\[
\length_{\widehat A}(\widehat B/\widehat A)
=
\length_A(B/A)
=
1.
\]
Since $\mathfrak m_{\widehat A}\widehat B=(t^2)$,
the image of \(\mathfrak m_{\widehat A}\) in
\(\widehat B\cong k[[t]]\) is contained in \((t^2)\).  Hence the image of
\(\widehat A\) is contained in
\[
k+t^2k[[t]].
\]
On the other hand,
\[
\length_{\widehat A}(\widehat B/\widehat A)=1
\]
and
\[
\widehat B/(k+t^2k[[t]])\cong k
\]
has length \(1\) as an \(\widehat A\)-module.  Therefore
\[
\widehat A=k+t^2k[[t]]=k[[t^2,t^3]].
\]
Thus the completed local ring of \(C\) at \(y\) is isomorphic to $k[[U,V]]/(V^2-U^3)$.
Hence \(C\) has a cusp at \(y\).
\end{proof}

\section{General setup for singular fibers}

Let $g\colon X\rightarrow \mathbb P^1$
be an elliptic fibration on a \(K3\) surface \(X\).  Let $F:=g^*(s)$
be a singular fiber of \(g\).  By Kodaira's classification \cite{k63}, \(F\) is
one of the following types:
\[
I_n\ (1 \leq n\leq 19),\quad I_m^*\ ( 0 \leq m\leq 14),\quad
II,\ III,\ IV,\ 
II^*,\ III^*,\ IV^*.
\]

An \emph{ADE reduced divisor}, abbreviated as an \emph{ADER}, in \(F\) is a
reduced subdivisor of \(F\) whose dual graph is a disjoint union of Dynkin
diagrams of ADE type.
For a chosen ADER \(R\subset F\),
let
\[
\nu\colon X\rightarrow Y
\]
be the contraction of \(R\).  By the standard contraction theory for configurations of rational curves on
surfaces \cite{a62,a66}, the contraction of \(R\) produces rational double
points. Then \(Y\) is a normal \(K3\) surface, and each
connected component of \(R\) gives a rational double point on \(Y\).  Since
\(R\) is contained in the fiber \(F\), the fibration \(g\) induces an elliptic
fibration
\[
f\colon Y\rightarrow \mathbb P^1
\]
such that
\[
g=f\circ\nu.
\]
We denote the corresponding fiber of \(f\) by
\[
\overline F:=f^*(s).
\]
By Lemma~\ref{2.3}, \(\overline F\) is a singular fiber of \(f\).  Moreover, by
Theorem~\ref{2.4}, the fiber \(\overline F\) is obtained from \(F\) by deleting
the components contained in \(R\), while preserving the multiplicities of the
remaining components.

In the following sections, we classify the possible fibers \(\overline F\) by
dividing the argument according to the Kodaira type of \(F\).  More precisely,
for each Kodaira fiber \(F\), we determine the possible ADERs contained in \(F\)
and describe the singular fiber obtained after contracting them.

This construction also gives all singular fibers of elliptic fibrations on
normal \(K3\) surfaces. 
Indeed, conversely, let $f'\colon Y'\rightarrow \mathbb P^1$ be an elliptic fibration on a normal \(K3\) surface $Y'$, and let
$\nu'\colon X'\rightarrow Y'$
be the minimal resolution.  We put $
g':=f'\circ\nu'$.
Then \(g'\) is an elliptic fibration on the \(K3\) surface \(X'\). 
By Lemma~\ref{2.3}, the
exceptional curves of \(\nu'\) are contained in singular fibers of \(g'\), and in each singular fiber they form an ADER.  
Therefore the classification below gives the classification of the singular
fibers of elliptic fibrations on normal \(K3\) surfaces. 
If a singular fiber of \(f'\) does not meet \(\Sing(Y')\), then by Lemma~\ref{2.2} it is identified with the corresponding Kodaira fiber of \(g'\).

\begin{dfn}\label{dfn:underlying-kodaira-type}
Let \(f\colon Y\to \mathbb P^1\) be an elliptic fibration on a normal
\(K3\) surface, and let \(\overline F\) be a singular fiber of \(f\).
We say that \(\overline F\) has underlying Kodaira type \(T\) if, as a
one-dimensional scheme with multiplicities, \(\overline F\) has the same
curve-theoretic configuration as a Kodaira fiber of type \(T\).
\end{dfn}

This terminology only records the curve-theoretic type of the fiber and does
not record the singularities of the ambient surface \(Y\).  When necessary, we
will specify separately the rational double points of \(Y\) lying on
\(\operatorname{Supp}(\overline F)\).  For example, an irreducible nodal
rational curve has underlying Kodaira type \(I_1\), even if its node lies at a
rational double point of the ambient surface.

\section{An example }

Before we proceed to the detailed classification, we present a concrete computational example.


Let $g : X \longrightarrow \mathbb{P}^1$ be an elliptic fibration on a
$K3$ surface $X$, and $F=g^{*}(s)$ be its singular fiber of Kodaira type $II^{*}$. We write
\[
  F = 6C + 3S_1 + 4M_1 + 2M_2 + 5L_1 + 4L_2 +3L_3 + 2L_4 + L_5
\]
whose dual graph is the following.
\[
  \begin{tikzpicture}[baseline=(bs)]
    \coordinate (bs) at (0, 0.4);

    \node (L5) at (0,0) {$L_5$};
    \node (L4) at (1,0) {$L_4$};
    \node (L3) at (2,0) {$L_3$};
    \node (L2) at (3,0) {$L_2$};
    \node (L1) at (4,0) {$L_1$};
    \node (C) at (5,0) {$C$};
    \node (M1) at (6,0) {$M_1$};
    \node (M2) at (7,0) {$M_2$};
    \node (S1) at (5,0.8) {$S_1$};

    \draw (L5) -- (L4) -- (L3) -- (L2) -- (L1) -- (C) -- (M1) -- (M2) (C) -- (S1);
  \end{tikzpicture}
\]
We choose an \emph{ADER}
\[
  R = L_5+L_3+L_1+M_1+M_2+S_1 \subset F
\]
of type $A_1^{\oplus 4} \oplus A_2$. More precisely, the connected
components of $R$ are $L_5, L_3, L_1$ and $S_1$, all of which are of
type $A_1$, and $M_1+M_2$ of type $A_2$. Let
$\nu : X \longrightarrow Y$ be the contraction of $R$, where $Y$ is a
normal $K3$ surface. Then we describe the singular fiber
$\ol{F}:=f^{*}(s)$ of the induced elliptic fibration
$f : Y \longrightarrow \mathbb{P}^1$ which is obtained by the
contraction $\nu$ such that $g = f \circ \nu$. The corresponding singular
fiber of $f$ is
\[
  \ol{F} = 2G_1 + 4G_2 + 6G_3,
\]
where $G_1 := \nu(L_4), \; G_2 := \nu(L_2)$ and $G_3:= \nu(C)$ are
smooth rational curves on $Y$ (as we shall see in
Proposition\ref{prop:IIstar-contraction}). The dual graph of $\ol{F}$
is of type $A_3$:
\[
  \begin{tikzpicture}[
    baseline=(bs),
    vertex/.style={circle, draw, inner sep=0pt, minimum size=5pt, thick},
    blackV/.style={vertex, fill=black},
    whiteV/.style={vertex, fill=white},
    every label/.style={font=\small, inner sep=4pt}
    ]

    \coordinate (bs) at (0,0.4);
    
    \node[blackV, label=below:{$L_5$}] (L5) at (0,0) {};
    \node[whiteV, label=below:{$L_4$}] (L4) at (1,0) {};
    \node[blackV, label=below:{$L_3$}] (L3) at (2,0) {};
    \node[whiteV, label=below:{$L_2$}] (L2) at (3,0) {};
    \node[blackV, label=below:{$L_1$}] (L1) at (4,0) {};
    \node[whiteV, label=below:{$C$}] (C) at (5,0) {};
    \node[blackV, label=below:{$M_1$}] (M1) at (6,0) {};
    \node[blackV, label=below:{$M_2$}] (M2) at (7,0) {};
    \node[blackV, label=right:{$S_1$}] (S1) at (5,0.8) {};
    
    \draw[thick] (L5) -- (L4) -- (L3) -- (L2) -- (L1) -- (C) -- (M1) -- (M2) (C) -- (S1);

    \node[whiteV, label=below:{$G_1$}] (G1) at (9,0) {};
    \node[whiteV, label=below:{$G_2$}] (G2) at (10,0) {};
    \node[whiteV, label=below:{$G_3$}] (G3) at (11,0) {};

    \draw[->] (7.5,0) -- node[above]{$\nu$} (8.5,0);
    
    \draw[thick] (G1) -- (G2) -- (G3);

  \end{tikzpicture}
\]
The normal $K3$ surface $Y$ has the following five singular points. 
\begin{itemize}
\item The point $p_1:=\nu(L_1)$ is of type $A_1$ and lies on the curve $G_1$.

\item The point $p_{2}:=\nu(L_3)$ is of type $A_1$ and $\{p_2\} = G_1 \cap G_2$.

\item The point $p_{3}:=\nu(L_1)$ is of type $A_1$ and $\{p_3\} = G_2 \cap G_3$.
  
\item The point $p_4:=\nu(S_1)$ is of type $A_1$ and lies on the curve $G_3$.

\item The point $p_5:=\nu(M_1 \cup M_2)$ is of type $A_2$ and lies on the curve $G_3$.
\end{itemize}
We will later define the type of this fiber $\ol{F}$ in
Definition\ref{dfn:IIstar-contracted-fibers} as
\[
  \mathfrak{II}^{*,\mathrm{br}}
(C ; M_1 + M_2 ,L_2+L_4).
\]

We calculate the rational pullbacks of $G_1, G_2$ and $G_3$ for
determining the intersection numbers among them. Let $F_1 := L_4, F_2:=L_2, F_3:=C$ and
\[
  \begin{aligned}
    \nu^{*} G_i = F_i + a_{i1} L_5 + a_{i2}L_3 + a_{i3}L_1 + a_{i4} M_1 + a_{i5}M_2 + a_{i6}S_1 \in \Div(X) \otimes_{\mathbb{Z}} \mathbb{Q}
  \end{aligned}
\]
be the rational pullback of each $G_i$. The
coefficients $a_{ij} \in \mathbb{Q}$ are uniquely determined by the
conditions
\[
  \left( \nu^{*}G_i \cdot D \right) =0
\]
for $(i,D) \in \{1,2,3\} \times \{ L_5, L_3, L_1, M_1, M_2,
S_1\}$. Thus, we have
\[
  \begin{aligned}
  \nu^{*}G_1 &= L_4 + \frac{1}{2}L_5 + \frac{1}{2} L_3, \quad
  \nu^{*}G_2 = L_2 + \frac{1}{2}L_3 + \frac{1}{2}L_1,\\
    \nu^{*}G_3& = C + \frac{1}{2}L_1 + \frac{2}{3}M_1 + \frac{1}{3}M_2 + \frac{1}{2}S_1.
  \end{aligned}
\]
Since the intersection number $(G_i \cdot G_j)$ on
$Y$ is equal to $(\nu^{*}G_i \cdot \nu^{*}G_j)$ on $X$, we obtain
the following intersection matrix of $\ol{F}$.
\[
  \left( \left( G_{i} \cdot G_{j}\right) \right)  =
  \begin{pmatrix*}[r]
    -1 & \frac{1}{2} & 0\\[1ex]
    \frac{1}{2} & -1 & \frac{1}{2}\\[1ex]
    0 & \frac{1}{2} & -\frac{1}{3}
  \end{pmatrix*}
\]
Moreover, it is easy to check that $(\ol{F} \cdot \ol{F}) =0$ from this matrix.

We assume that the fibration $g$ has a zero section $O$. Then
$\ol{O}=\nu(O)$ is a zero section of $f$.  Note that although
$\ol{F} = 2 (G_1 + 2 G_2 + 3G_3)$ is a multiple fiber of $f$, it
intersects the zero section $\ol{O}$ at $p_1$ and
$(\ol{F} \cdot \ol{O}) =1$. Indeed, since
$\nu^{*}\ol{O} = O + \frac{1}{2}L_5$ , we have
\[
  \begin{aligned}
    \left( \ol{F} \cdot \ol{O}\right)
    &= (2G_1 \cdot \ol{O})
      = \left(\left( 2L_4 + L_5 + L_3 \right) \cdot \left( O + \frac{1}{2}L_5\right) \right)\\
    &= (L_4 \cdot L_5) + (L_5\cdot O) +\frac{1}{2}  (L_5 \cdot L_5) =  1 + 1 -1 = 1.
  \end{aligned}
\]
This is an effect unique to the singular case. A smooth elliptic
surface with a section has no multiple fibers, since the multiplicity
of an irreducible component of a fiber intersecting a section must be
one. Furthermore, an elliptic fibration on a smooth $K3$ surface has
no multiple fibers even without a section.

\section{Fibers of type \(I_1\) and \(II\)}

Fibers of type \(I_1\) and \(II\) do not contain any ADER.  Indeed,
they are irreducible singular rational curves and hence contain no smooth rational component.
Therefore no contraction occurs in these cases.

\section{Fibers of type \(I_2\)}

We next consider the first case in which a non-empty ADER can occur.
An \(I_2\)-fiber consists of two smooth rational \((-2)\)-curves meeting
transversely at two distinct points.  This simple configuration gives only
one possible ADER.

\begin{prop}\label{prop:I2-ADER}
We assume that \(F\) is of Kodaira type \(I_2\).  
Let $R$ be an ADER in \(F\).
Then $R$ is of type \(A_1\).
\end{prop}
\begin{proof}
We write \(F=C_1+C_2\), where \(C_1\) and \(C_2\) are the two irreducible
components of the fiber.  Since \(F\) is of Kodaira type \(I_2\), the curves
\(C_1\) and \(C_2\) are smooth rational \((-2)\)-curves and they meet
transversely at two distinct points. 
If \(R\) contains both \(C_1\) and \(C_2\), then the corresponding dual graph
has two vertices joined by two edges, because \(C_1\) and \(C_2\) meet at two
points.  This is not the Dynkin diagram of type \(A_2\), where the two vertices
are joined by only one edge. 
Therefore $R$ contain only one irreducible
component of \(F\).  Each \(C_i\) is a smooth rational \((-2)\)-curve, hence
gives a Dynkin diagram of type \(A_1\).  Thus \(R\) is of type \(A_1\).
\end{proof}

\begin{prop}\label{prop:I2-contraction}
We assume that \(F\) is of Kodaira type \(I_2\).  Let \(R\) be an ADER in
\(F\), and let $\nu\colon X\rightarrow Y$
be the contraction of \(R\).  Let \(G\) be the irreducible component of \(F\)
different from \(R\). We put $C:=\nu(G)$.
Then 
\[\overline F=C.\]
Moreover, \(C\) is an irreducible rational curve with one node at the point
\(\nu(R)\), and \(\nu(R)\) is a rational double point of type \(A_1\).  In
particular, \(\overline F\) has underlying Kodaira type \(I_1\).
\end{prop}

\begin{proof}
By Proposition~\ref{prop:I2-ADER}, the ADER \(R\) is of type \(A_1\).  Hence
the contraction of \(R\) produces a rational double point of type \(A_1\) at
the point $y:=\nu(R)$.
Since \(F\) is of Kodaira type \(I_2\), we may write
\[
F=R+G,
\]
where \(R\) and \(G\) are smooth rational \((-2)\)-curves meeting transversely
at two distinct points.  By Theorem~\ref{2.4}, the fiber \(\overline F\) is
obtained from \(F\) by deleting the contracted component \(R\), while preserving
the multiplicity of the remaining component.  Since both components of an
\(I_2\)-fiber have multiplicity \(1\), we obtain
\[
\overline F=\nu(G)=C.
\]

The curve \(G\) is smooth and rational, and the morphism
\[
\nu|_G\colon G\rightarrow C
\]
is the normalization of \(C\).  Away from the point \(y=\nu(R)\), the morphism \(\nu\) is an isomorphism.  Therefore \(C\) is smooth away from \(y\).
Let $G\cap R=\{x_1,x_2\}$.
Since \(R\) and \(G\) meet transversely at two distinct points, we have 
\[(R\cdot G)_{x_i}=1\qquad
(i=1,2).\]
Since \(\nu^{-1}(y)=R\), this gives
\[
G\cap \nu^{-1}(y)=\{x_1,x_2\}
\quad
\mathrm{and}
\quad
(\nu^{-1}(y)\cdot G)_{x_i}=1
\qquad
(i=1,2).
\]
Since \(\overline F=C\) is a fiber of the elliptic fibration \(f\), we have
\[
C^2=0.
\]
By adjunction on the normal \(K3\) surface \(Y\),
\[
2p_a(C)-2=C^2.
\]
Hence
\[
p_a(C)=1.
\]
Since the normalization of \(C\) is \(G\cong \mathbb P^1\), we have
\[
g(G)=0.
\]
Moreover, \(y\) is the only singular point of \(C\).  Therefore
\[
\delta_y(C)=p_a(C)-g(G)=1.
\]
Thus the assumptions of Proposition~\ref{prop:delta-one-node-cusp} are
satisfied.  Hence \(C\) has a node at \(y=\nu(R)\).  
Therefore
\(\overline F=C\) is an irreducible nodal rational curve.  In particular,
\(\overline F\) has underlying Kodaira type \(I_1\) in the sense of
Definition~\ref{dfn:underlying-kodaira-type}.  The node is the rational double
point \(y=\nu(R)\) of type \(A_1\) on \(Y\).
\end{proof}

\section{Fibers of type \(III\)}

The case of a fiber of type \(III\) is similar to the case of type \(I_2\),
but the two components meet at one point with intersection multiplicity \(2\).
This difference is reflected after contraction: the surviving component becomes
cuspidal rather than nodal.

\begin{prop}\label{prop:III-ADER}
We assume that \(F\) is of Kodaira type \(III\).  Let \(R\) be an
ADER in \(F\).  Then \(R\) is of type \(A_1\).
\end{prop}
\begin{proof}
We write $F=C_1+C_2$,
where \(C_1\) and \(C_2\) are smooth rational curves.
Since \(F\) is of Kodaira type \(III\), \(C_1\) and \(C_2\) meet at one point with intersection multiplicity \(2\).
If \(R\) contained both \(C_1\) and \(C_2\), then the dual graph of \(R\) would
have two vertices joined by two edges.  This is not a Dynkin diagram of ADE
type.  Hence \(R\) contains only one irreducible component of \(F\).  
Thus \(R\) is of type \(A_1\).
\end{proof}

\begin{prop}\label{prop:III-contraction}
We assume that \(F\) is of Kodaira type \(III\).  Let \(R\) be an
ADER in \(F\), and let $\nu\colon X\rightarrow Y$ be the contraction of \(R\).  Let \(G\) be the irreducible component of \(F\)
different from \(R\). We put $C:=\nu(G)$.
Then 
\[\overline F=C.\]
Moreover, \(C\) is an irreducible rational curve with one cusp at the point
\(\nu(R)\), and \(\nu(R)\) is a rational double point of type \(A_1\).  
In particular, \(\overline F\) has underlying Kodaira type \(II\).
\end{prop}
\begin{proof}
By Proposition~\ref{prop:III-ADER}, the ADER \(R\) is of type \(A_1\).  Hence
the contraction of \(R\) produces a rational double point of type \(A_1\) at
the point $y:=\nu(R)$.
Since \(F\) is of Kodaira type \(III\), we may write
$F=R+G$
where \(R\) and \(G\) are smooth rational \((-2)\)-curves meeting at one point
with intersection multiplicity \(2\).  By Theorem~\ref{2.4}, the fiber
\(\overline F\) is obtained from \(F\) by deleting the contracted component
\(R\), while preserving the multiplicity of the remaining component.  Since
both components of a \(III\)-fiber have multiplicity \(1\), we obtain
\[
\overline F=\nu(G)=C.
\]
The curve \(G\) is smooth and rational, and the morphism $\nu|_G\colon G\rightarrow C$
is the normalization of \(C\).  Away from the point \(y=\nu(R)\), the morphism
\(\nu\) is an isomorphism.  Therefore \(C\) is smooth away from \(y\).
Let $G\cap R=\{x\}$.
Since \(G\) and \(R\) meet at \(x\) with intersection multiplicity \(2\), we have
\[
(R\cdot G)_x=2.
\]
Since \(\nu^{-1}(y)=R\), this gives
\[
G\cap \nu^{-1}(y)=\{x\}
\qquad\text{and}\qquad
(\nu^{-1}(y)\cdot G)_x=2.
\]
As in the proof of Proposition~\ref{prop:I2-contraction},
\[
\delta_y(C)=1.
\]
Thus the assumptions of Proposition~\ref{prop:delta-one-node-cusp} are
satisfied.  Hence \(C\) has a cusp at \(y=\nu(R)\).  Therefore
\(\overline F=C\) is an irreducible cuspidal rational curve.  In particular,
\(\overline F\) has underlying Kodaira type \(II\) in the sense of
Definition~\ref{dfn:underlying-kodaira-type}.  The cusp is located at the
rational double point \(y=\nu(R)\) of type \(A_1\) on \(Y\).
\end{proof}

\section{Fibers of type \(IV\)}

Throughout this section, we use the following notation.  Let \(F\) be a singular
fiber of Kodaira type \(IV\).  We write
\[
F=C_1+C_2+C_3,
\]
where \(C_1,C_2,C_3\) are smooth rational curves such that \(C_1,C_2,C_3\) meet at one common point with distinct tangent directions.
We represent the support of \(F\) by the following graph:
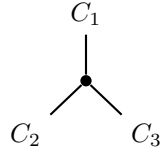
\[
\begin{tikzpicture}[
  baseline=(P),
  scale=0.8,
  every label/.style={font=\scriptsize, inner sep=2pt}
]
  \coordinate (P) at (0,0);

  \node (C1) at (0,1.1) {$C_1$};
  \node (C2) at (-1.0,-0.9) {$C_2$};
  \node (C3) at (1.0,-0.9) {$C_3$};

  \node[circle, fill, inner sep=1.5pt] (Q) at (0,0) {};

  \draw[thick] (Q) -- (C1.south);
  \draw[thick] (Q) -- (C2.north east);
  \draw[thick] (Q) -- (C3.north west);
\end{tikzpicture}
\]
Here and in the sequel, a solid line segment represents an irreducible
component of a fiber, and a black circle represents an intersection point of
irreducible components.  

\begin{prop}\label{prop:ADE-IV}
We assume that \(F\) is of Kodaira type \(IV\), with notation as above.  Let
\(R\) be an ADER in \(F\).  Then \(R\) is either of type \(A_1\) or of
type \(A_2\).  More precisely, one of the following holds.
\begin{enumerate}[\rm(1)]
\item If \(R\) consists of one irreducible component of \(F\), then \(R\) is of
type \(A_1\).

\item If \(R\) consists of two irreducible components of \(F\), then \(R\) is of
type \(A_2\).
\end{enumerate}

Conversely, both types occur.
\end{prop}
\begin{proof}
Since \(F\) has precisely three
irreducible components \(C_1,C_2,C_3\), the number of irreducible components of
\(R\) is \(1\), \(2\), or \(3\).
If \(R\) consists of one irreducible component, then its support graph has one
vertex, and hence \(R\) is of type \(A_1\).  If \(R\) consists of two
irreducible components, then its support graph is a chain with two vertices,
and hence \(R\) is of type \(A_2\).
It remains to exclude the case where \(R\) contains all three components
\(C_1,C_2,C_3\).  In this case, the associated intersection graph is the affine
Dynkin diagram \(\widetilde A_2\), not a finite ADE diagram.  Hence this case
cannot occur for an ADER.
Therefore the only possibilities are \(A_1\) and \(A_2\).  

Conversely, any one
component of \(F\) gives an \(A_1\)-configuration, and any two components of
\(F\) give an \(A_2\)-configuration.
\end{proof}

Before describing the contraction, we introduce notation for the cuspidal fiber
which occurs after contracting an \(A_2\)-configuration in a fiber of type
\(IV\).

\begin{dfn}\label{dfn:IV-contracted-fibers}
We define the following two types of singular fibers.

\begin{enumerate}[\rm(1)]
\item We say that a singular fiber is of type
\[
\mathfrak S_2(A_1)
\]
if it is of the form
\[
C_1+C_2,
\]
where \(C_1\) and \(C_2\) are distinct smooth rational curves, both with
coefficient \(1\), passing through one rational double point of type \(A_1\),
and they have no other intersection points.

\item We say that a singular fiber is of type
\[
\mathfrak C(A_2)
\]
if it is of the form
\[
C,
\]
where \(C\) is an irreducible rational curve with one cusp at a rational double
point of type \(A_2\), and \(C\) has no other singularities.
\end{enumerate}
\end{dfn}

\begin{prop}\label{prop:IV-contraction}
We assume that \(F\) is of Kodaira type \(IV\), with notation as above.  Let
\(R\) be an ADER in \(F\), and let \(\nu\colon X\rightarrow Y\)
be the contraction of \(R\).  Then the resulting singular fiber \(\overline F\)
is described as follows.

\begin{enumerate}[\rm(1)]
\item If \(R\) is of type \(A_1\), then \(\overline F\) is of type
\[
\mathfrak S_2(A_1)
\]
in the sense of Definition~\ref{dfn:IV-contracted-fibers}.

\item If \(R\) is of type \(A_2\), then \(\overline F\) is of type
\[
\mathfrak C(A_2)
\]
in the sense of Definition~\ref{dfn:IV-contracted-fibers}.
\end{enumerate}
\end{prop}

\begin{proof}
By Proposition~\ref{prop:ADE-IV}, the ADER \(R\) is either of type \(A_1\) or of
type \(A_2\).

First, we assume that \(R\) is of type \(A_1\).  After relabeling the
components, we may write
\[
R=C_1.
\]
Then the contraction of \(R\) produces one rational double point
\[
y:=\nu(C_1)
\]
of type \(A_1\).  The surviving components are \(C_2\) and \(C_3\).  We put
\[
\overline C_2:=\nu(C_2),
\qquad
\overline C_3:=\nu(C_3).
\]
By Theorem~\ref{2.4}, we have
\[
\overline F=\overline C_2+\overline C_3.
\]
For \(i=2,3\), the component \(C_i\) meets the exceptional curve \(C_1\)
transversely at one point.  Hence, by
Proposition~\ref{prop:smoothness-criterion}, the curve \(\overline C_i\) is
smooth at \(y\).  Away from \(y\), the morphism \(\nu\) is an isomorphism.
Therefore \(\overline C_2\) and \(\overline C_3\) are smooth rational curves.

Moreover, both \(\overline C_2\) and \(\overline C_3\) pass through \(y\), and
they have no other intersection points.  Thus \(\overline F\) is of type
\[
\mathfrak S_2(A_1).
\]

Next, we assume that \(R\) is of type \(A_2\).  After relabeling the
components, we may write
\[
R=C_1+C_2.
\]
Then the contraction of \(R\) produces one rational double point
\[
y:=\nu(C_1\cup C_2)
\]
of type \(A_2\).  The only surviving component is \(C_3\).  We put
\[
\overline C:=\nu(C_3).
\]
By Theorem~\ref{2.4}, we have
\[
\overline F=\overline C.
\]
The strict transform of \(\overline C\) on \(X\) is \(C_3\cong\mathbb P^1\),
and the restriction
\[
\nu|_{C_3}\colon C_3\rightarrow \overline C
\]
is the normalization of \(\overline C\).
As in the proof of Proposition~\ref{prop:I2-contraction},
\[
\delta_y(\overline C)=1.
\]
The component \(C_3\) meets the exceptional configuration
\(C_1+C_2\) at the common point of the three components of the \(IV\)-fiber, and
\[
((C_1+C_2)\cdot C_3)=2.
\]
Thus
\[
C_3\cap \nu^{-1}(y)=\{x\},
\qquad
(\nu^{-1}(y)\cdot C_3)_x=2.
\]
By Proposition~\ref{prop:delta-one-node-cusp}, the curve \(\overline C\) has a
cusp at \(y\).  Hence \(\overline F\) is of type
\[
\mathfrak C(A_2).
\]
\end{proof}

\section{Fibers of type $I_n$}
Configurations of \(I_n\)-fibers on elliptic \(K3\) surfaces were studied
by Miranda and Persson \cite{mp89}, and related questions for extremal
semistable elliptic \(K3\) surfaces were studied in \cite{a02}.

We study singular fibers of type \(I_n\) for $n\geq3$. Recall that an \(I_n\)-fiber is a cycle of \(n\) smooth rational curves. 
\begin{prop}\label{prop:ADE-In-fiber}
We assume that \(F\) is of Kodaira type \(I_n\) with \(n\geq 3\).  Let \(R\)
be an ADER in \(F\).  Then \(R\) is of type
\[
A_{\lambda_1}\oplus\cdots\oplus A_{\lambda_r}
\]
for some integers \(\lambda_1,\dots,\lambda_r\geq 1\) satisfying
\[
\lambda_1+\cdots+\lambda_r+r\leq n.
\]

Conversely, every such type occurs as an ADER in $F$.
\end{prop}
\begin{proof}
The dual graph of an \(I_n\)-fiber is a cycle with \(n\) vertices, and every
vertex has valency \(2\).  Hence every subgraph of this cycle has all vertices
of valency at most \(2\).  On the other hand, every Dynkin diagram of type
\(D_m\) \((m\geq 4)\) and of type \(E_6,E_7,E_8\) contains a vertex of valency
\(3\).  Therefore no connected component of the dual graph of \(R\) can be of
type \(D\) or \(E\).  Thus every connected component of the dual graph of \(R\)
is of type \(A_k\) for some \(k\geq 1\).  Hence \(R\) is of type
\[
A_{\lambda_1}\oplus\cdots\oplus A_{\lambda_r}
\]
for some integers \(\lambda_1,\dots,\lambda_r\geq 1\).

Each connected component of \(R\) is a chain of consecutive vertices in the
cycle.  If the connected components have types \(A_{\lambda_1},\dots,A_{\lambda_r}\), then they contain altogether
\[
\lambda_1+\cdots+\lambda_r
\]
vertices.  Since distinct connected components must be separated by at least
one vertex not contained in \(R\), at least \(r\) vertices of the cycle are not
contained in \(R\).  Therefore
\[
\lambda_1+\cdots+\lambda_r+r\leq n.
\]

Conversely, let \(\lambda_1,\dots,\lambda_r\geq 1\) be integers satisfying
\[
\lambda_1+\cdots+\lambda_r+r\leq n.
\]
Choose \(\lambda_1\) consecutive vertices on the cycle, leave one vertex unused,
then choose \(\lambda_2\) consecutive vertices, leave one vertex unused, and
continue in this way.  After choosing the \(r\)-th chain, the inequality above
ensures that there is still at least one unused vertex separating the last chain
from the first one on the cycle.  Therefore the chosen vertices define a reduced
subdivisor of \(F\) whose connected components are of types
\[
A_{\lambda_1},\dots,A_{\lambda_r}.
\]
Thus this subdivisor is an ADER in \(F\) of type
\[
A_{\lambda_1}\oplus\cdots\oplus A_{\lambda_r}.
\]
This proves the converse.
\end{proof}

\begin{prop}\label{prop:In-fiber-contraction}
We assume that \(F\) is of Kodaira type \(I_n\) with \(n\geq 3\).  Let \(R\)
be an ADER in \(F\) of type
\[
A_{\lambda_1}\oplus\cdots\oplus A_{\lambda_r}.
\]
We set
\[
m:=n-\sum_{i=1}^r\lambda_i.
\]
Let $\nu\colon X\rightarrow Y$ be the contraction of \(R\).  Then the singular points of \(Y\) lying on
\(\overline F\) are precisely \(r\) rational double points of types $A_{\lambda_1},\dots,A_{\lambda_r}$.
Moreover, the fiber \(\overline F\) is described as follows.
\begin{enumerate}[\rm(1)]
\item If \(m\geq 2\), then \(\overline F\) has underlying Kodaira type
\[
I_m.
\]
More precisely, \(\overline F\) has exactly \(m\) irreducible components.
All of them are smooth rational curves, and they form a cycle.  The rational
double points obtained by contracting \(R\) lie at some of the intersection
points of adjacent components in this cycle.

\item If \(m=1\), equivalently \(R\) is of type \(A_{n-1}\), then
\(\overline F\) has underlying Kodaira type
\[
I_1.
\]
More precisely, \(\overline F\) has exactly one irreducible component.
It is an irreducible rational curve with one node.  This node is supported
at the rational double point of type \(A_{n-1}\) obtained by contracting
\(R\).
\end{enumerate}
\end{prop}

\begin{proof}
Each connected component of \(R\) of type \(A_{\lambda_i}\) is contracted to a rational double point of type \(A_{\lambda_i}\).  Since the connected components of \(R\) are disjoint,
the singular points of \(Y\) lying on \(\overline F\) are precisely the \(r\) rational double points of types $A_{\lambda_1},\dots,A_{\lambda_r}$.

The fiber \(F\) has \(n\) irreducible components, and \(R\) contains $\sum_{i=1}^r\lambda_i$ of them.  Hence, by Theorem~\ref{2.4}, the fiber \(\overline F\) has $m=n-\sum_{i=1}^r\lambda_i$
irreducible components.

We first assume that \(m\geq 2\).  Let \(E\) be an irreducible component of
\(F\) not contained in \(R\).
We put $C:=\nu(E)$.
We show that \(C\) is smooth.  Let
\[
y\in C\cap\Sing(Y).
\]
Then \(\nu^{-1}(y)\) is a connected component of \(R\), say \(R_y\), of type
\(A_{\lambda_i}\) for some \(i\).  Since \(R_y\) is a chain of consecutive
components in the \(I_n\)-cycle and \(m\geq 2\), the component \(E\) meets
\(R_y\) at exactly one point.  Moreover, this intersection is transverse.  Thus
\[
E\cap \nu^{-1}(y)=\{x\}
\qquad\text{and}\qquad
(\nu^{-1}(y)\cdot E)_x=1.
\]
The restriction $\nu|_E\colon E\rightarrow C$ is the normalization of \(C\).  Therefore, by
Proposition~\ref{prop:smoothness-criterion}, the curve \(C\) is smooth at
\(y\).  Since \(\nu\) is an isomorphism away from the exceptional locus, \(C\)
is smooth away from the singular points of \(Y\).  Hence \(C\) is a smooth
rational curve.
It follows that all irreducible components of \(\overline F\) are smooth
rational curves.  Contracting each connected component of \(R\) replaces a chain
of consecutive components in the \(I_n\)-cycle by the corresponding rational
double point, and it preserves the cyclic order of the remaining components.
Thus the \(m\) remaining smooth rational curves form a cycle.  The rational
double points obtained by contracting \(R\) lie at some of the intersection
points of adjacent components in this cycle.  Therefore \(\overline F\) has
underlying Kodaira type \(I_m\).

We consider the case \(m=1\).  Then $\sum_{i=1}^r\lambda_i=n-1$.
By Proposition~\ref{prop:ADE-In-fiber}, we have $\sum_{i=1}^r\lambda_i+r\leq n$.
Hence \(r=1\), and therefore \(R\) is of type \(A_{n-1}\).
Let \(E\) be the unique irreducible component of \(F\) not contained in \(R\).
We put $C:=\nu(E)$.
By Theorem~\ref{2.4}, we have
\[
\overline F=C.
\]
The curve \(E\) is smooth and rational, and the restriction $\nu|_E\colon E\rightarrow C$
is the normalization of \(C\).
We set $y:=\nu(R)$.
Since \(R\) is the chain consisting of all components of the \(I_n\)-fiber except
\(E\), the component \(E\) meets the two end components of \(R\) transversely at
two distinct points.  Thus
\[
E\cap R=\{x_1,x_2\}
\qquad\text{and}\qquad
(R\cdot E)_{x_i}=1
\quad
(i=1,2).
\]
Since \(\nu^{-1}(y)=R\), we obtain
\[
E\cap \nu^{-1}(y)=\{x_1,x_2\}
\qquad\text{and}\qquad
(\nu^{-1}(y)\cdot E)_{x_i}=1
\quad
(i=1,2).
\]
As in the proof of Proposition~\ref{prop:I2-contraction},
\[
\delta_y(\overline C)=1.
\]
Thus the assumptions of Proposition~\ref{prop:delta-one-node-cusp} are
satisfied.  
Thus \(C\) has a node at \(y=\nu(R)\).  This node is supported at the rational
double point of type \(A_{n-1}\) obtained by contracting \(R\).  Therefore
\(\overline F\) is an irreducible rational curve with one node.  In particular,
\(\overline F\) has underlying Kodaira type \(I_1\).
\end{proof}

\section{Fibers of type $I_0^*$}
We first treat the fiber of type \(I_0^*\) separately. Throughout this
subsection, we use the following notation. Let \(F\) be a singular fiber of
Kodaira type \(I_0^*\). We write
\[
F=2C_0+C_1+C_2+C_3+C_4,
\]
where \(C_0\) is the unique component meeting four other components. The dual
graph of \(F\) is the affine Dynkin diagram \(\widetilde D_4\), and it is given
as follows:
\[
\begin{tikzpicture}[baseline=(current bounding box.center)]
\node (C1) at (-1,0) {$C_1$};
\node (C0) at (0,0) {$C_0$};
\node (C2) at (1,0) {$C_2$};

\node (C3) at (0,1) {$C_3$};
\node (C4) at (0,-1) {$C_4$};

\draw[thick] (C1.east) -- (C0.west);
\draw[thick] (C0.east) -- (C2.west);

\draw[thick] (C0.north) -- (C3.south);
\draw[thick] (C0.south) -- (C4.north);
\end{tikzpicture}
\]
\begin{prop}\label{prop:ADE-I0star}
We assume that \(F\) is of Kodaira type \(I_0^*\).  Let \(R\) be an ADER in
\(F\).  Then \(R\) is one of the following types:
\[
D_4,\qquad A_3,\qquad A_2,\qquad A_1^{\oplus s}\quad (1\leq s\leq 4).
\]

More precisely:
\begin{enumerate}[\rm(1)]
\item If \(C_0\subset R\), then \(R\) is connected and is of type
\[
D_4,\qquad A_t\quad (1\le t\le 3).
\]

\item If \(C_0\not\subset R\), then \(R\) is a disjoint union of leaves, hence is of type
\[
A_1^{\oplus s}\quad (1\le s\le 4).
\]
\end{enumerate}

Conversely, every configuration listed above occurs.
\end{prop}
\begin{proof}
We first assume that \(C_0\subset R\).  Since \(C_0\) is connected to each of
the leaves \(C_1,C_2,C_3,C_4\), the divisor \(R\) is connected.  Let \(q\) be
the number of leaves contained in \(R\).  Then the dual graph of \(R\) has one
central vertex and \(q\) leaves.  If \(q=4\), then the central vertex has
valency \(4\), and hence the graph is not a Dynkin diagram of ADE type.
Therefore, since \(R\) is an ADER, we have \(0\leq q\leq 3\).  Hence:
\begin{itemize}
\item if \(q=0\), then \(R\) is of type \(A_1\);
\item if \(q=1\), then \(R\) is of type \(A_2\);
\item if \(q=2\), then \(R\) is of type \(A_3\);
\item if \(q=3\), then \(R\) is of type \(D_4\).
\end{itemize}
Thus the possible ADERs containing \(C_0\) are exactly of types
\[
A_1,\quad A_2,\quad A_3,\quad D_4.
\]

Next, we assume that \(C_0\not\subset R\). Then \(R\) is supported on some of the leaves
$C_1,C_2,C_3$, and $C_4$.
Since the leaves are mutually disjoint, every connected component of \(R\) is an isolated vertex,
hence is of type \(A_1\). Therefore
\[
R\sim A_1^{\oplus s}\qquad (1\le s\le 4).
\]

Conversely, every configuration listed above occurs. 
Indeed, if \(C_0\subset R\), the types \(A_1, A_2, A_3, D_4\) are obtained by 
choosing \(C_0\) together with \(q=0,1,2,3\) leaves, respectively; if 
\(C_0\not\subset R\), the type \(A_1^{\oplus s}\ (1\le s\le 4)\) is obtained by choosing any \(s\) disjoint leaves.
\end{proof}

\begin{dfn}\label{dfn:new-I0star-fibers}
We define the following singular fibers on a normal \(K3\) surface.

\begin{enumerate}[\rm(1)] 
\item Let \(1\leq s\leq 4\), and let \(\Gamma\) be an ADE type.  We say that a
singular fiber is of type
\[
\mathfrak S_s(\Gamma)
\]
if, as a divisor, it is of the form
\[
R_1+\cdots+R_s,
\]
where \(R_1,\dots,R_s\) are smooth rational curves, all components have
multiplicity \(1\), and all of them pass through one rational double point of
type \(\Gamma\).  Moreover, these components have no other intersection points.

When \(s=1\), this means that the fiber is a single smooth rational curve
containing one rational double point of type \(\Gamma\).

For \(s=1,2,3,4\), the support graphs are as follows:
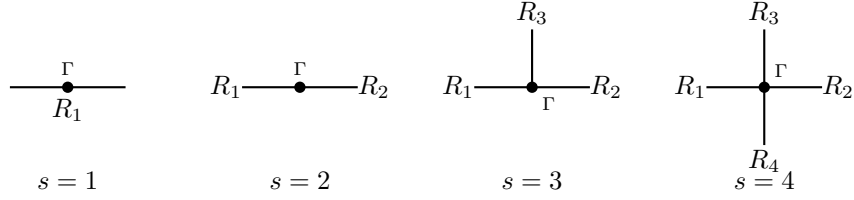
\[
\begin{array}{cccc}
\begin{tikzpicture}[
  baseline=(P),
  scale=0.85,
  rdp/.style={circle, fill, inner sep=1.5pt},
  every label/.style={font=\scriptsize, inner sep=2pt}
]
  \useasboundingbox (-1.6,-1.8) rectangle (1.6,1.5);
  \coordinate (P) at (0,0);

  \draw[thick] (-0.9,0) -- (0.9,0);
  \node[rdp, label=above:{$\Gamma$}] at (P) {};
  \node at (0,-0.35) {$R_1$};
  \node at (0,-1.45) {$s=1$};
\end{tikzpicture}
&
\begin{tikzpicture}[
  baseline=(P),
  scale=0.85,
  rdp/.style={circle, fill, inner sep=1.5pt},
  every label/.style={font=\scriptsize, inner sep=2pt}
]
  \useasboundingbox (-1.6,-1.8) rectangle (1.6,1.5);
  \coordinate (P) at (0,0);

  \draw[thick] (-0.9,0) -- (0,0);
  \draw[thick] (0,0) -- (0.9,0);
  \node[rdp, label=above:{$\Gamma$}] at (P) {};
  \node at (-1.15,0) {$R_1$};
  \node at (1.15,0) {$R_2$};
  \node at (0,-1.45) {$s=2$};
\end{tikzpicture}
&
\begin{tikzpicture}[
  baseline=(P),
  scale=0.85,
  rdp/.style={circle, fill, inner sep=1.5pt},
  every label/.style={font=\scriptsize, inner sep=2pt}
]
  \useasboundingbox (-1.6,-1.8) rectangle (1.6,1.5);
  \coordinate (P) at (0,0);

  \draw[thick] (-0.9,0) -- (0,0);
  \draw[thick] (0,0) -- (0.9,0);
  \draw[thick] (0,0) -- (0,0.9);
  \node[rdp, label=below right:{$\Gamma$}] at (P) {};
  \node at (-1.15,0) {$R_1$};
  \node at (1.15,0) {$R_2$};
  \node at (0,1.15) {$R_3$};
  \node at (0,-1.45) {$s=3$};
\end{tikzpicture}
&
\begin{tikzpicture}[
  baseline=(P),
  scale=0.85,
  rdp/.style={circle, fill, inner sep=1.5pt},
  every label/.style={font=\scriptsize, inner sep=2pt}
]
  \useasboundingbox (-1.6,-1.8) rectangle (1.6,1.5);
  \coordinate (P) at (0,0);

  \draw[thick] (-0.9,0) -- (0,0);
  \draw[thick] (0,0) -- (0.9,0);
  \draw[thick] (0,0) -- (0,0.9);
  \draw[thick] (0,0) -- (0,-0.9);
  \node[rdp, label=above right:{$\Gamma$}] at (P) {};
  \node at (-1.15,0) {$R_1$};
  \node at (1.15,0) {$R_2$};
  \node at (0,1.15) {$R_3$};
  \node at (0,-1.15) {$R_4$};
  \node at (0,-1.45) {$s=4$};
\end{tikzpicture}
\end{array}
\]
The case \(s=2\) and \(\Gamma=A_1\) agrees with
Definition~\ref{dfn:IV-contracted-fibers}\textup{(1)}.

\item We say that a singular fiber is of type
\[
\mathfrak D_4^{\mathrm{mark}}(A_1)
\]
if, as a divisor, it is of the form
\[
R_0+2R_1+R_2+R_3,
\]
where \(R_0,R_1,R_2,R_3\) are smooth rational curves whose support graph is as
follows:
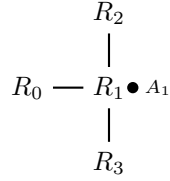
\[
\begin{tikzpicture}[
  baseline=(P),
  scale=0.9,
  rdp/.style={circle, fill, inner sep=1.5pt},
  every label/.style={font=\scriptsize, inner sep=2pt}
]
  \useasboundingbox (-1.8,-1.5) rectangle (1.8,1.5);
  \coordinate (P) at (0,0);

  \node (C1) at (0,0) {$R_1$};
  \node (C0) at (-1.2,0) {$R_0$};
  \node (C2) at (0,1.1) {$R_2$};
  \node (C3) at (0,-1.1) {$R_3$};

  \draw[thick] (C0.east) -- (C1.west);
  \draw[thick] (C1.north) -- (C2.south);
  \draw[thick] (C1.south) -- (C3.north);

  \node[rdp, label=right:{$A_1$}] at (0.35,0) {};
\end{tikzpicture}
\]
Moreover, \(R_1\) contains one rational double point of type \(A_1\), away from the intersection points of the irreducible components.

\item We say that a singular fiber is of type
\[
\mathfrak A_3^{\mathrm{mark}}(A_1^{\oplus 2})
\]
if, as a divisor, it is of the form
\[
R_0+2R_1+R_2,
\]
where \(R_0,R_1,R_2\) are smooth rational curves whose support graph is as
follows:
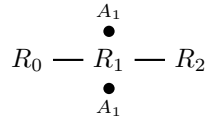
\[
\begin{tikzpicture}[
  baseline=(P),
  scale=0.9,
  rdp/.style={circle, fill, inner sep=1.5pt},
  every label/.style={font=\scriptsize, inner sep=2pt}
]
  \useasboundingbox (-1.8,-1.4) rectangle (1.8,1.4);
  \coordinate (P) at (0,0);

  \node (C0) at (-1.2,0) {$R_0$};
  \node (C1) at (0,0) {$R_1$};
  \node (C2) at (1.2,0) {$R_2$};

  \draw[thick] (C0.east) -- (C1.west);
  \draw[thick] (C1.east) -- (C2.west);

  \node[rdp, label=above:{$A_1$}] at (0,0.42) {};
  \node[rdp, label=below:{$A_1$}] at (0,-0.42) {};
\end{tikzpicture}
\]
Moreover, \(R_1\) contains two rational double points of type \(A_1\), neither of which is an intersection point of two irreducible components of the fiber.

\item We say that a singular fiber is of type
\[
\mathfrak A_2^{\mathrm{mark}}(A_1^{\oplus 3})
\]
if, as a divisor, it is of the form
\[
R_0+2R_1,
\]
where \(R_0\) and \(R_1\) are smooth rational curves whose support graph is as
follows:
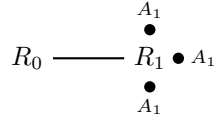
\[
\begin{tikzpicture}[
  baseline=(P),
  scale=0.9,
  rdp/.style={circle, fill, inner sep=1.5pt},
  every label/.style={font=\scriptsize, inner sep=2pt}
]
  \useasboundingbox (-1.8,-1.4) rectangle (1.8,1.4);
  \coordinate (P) at (0,0);

  \node (C0) at (-1.2,0) {$R_0$};
  \node (C1) at (0.6,0) {$R_1$};

  \draw[thick] (C0.east) -- (C1.west);

  \node[rdp, label=above:{$A_1$}] at (0.6,0.42) {};
  \node[rdp, label=right:{$A_1$}] at (1.02,0) {};
  \node[rdp, label=below:{$A_1$}] at (0.6,-0.42) {};
\end{tikzpicture}
\]
Moreover, the component \(R_1\) contains three rational double points of type
\(A_1\), and none of them is the intersection point \(R_0\cap R_1\).

\item We say that a singular fiber is of type
\[
\mathfrak A_1^{\mathrm{mark}}(A_1^{\oplus 4})
\]
if, as a divisor, it is of the form
\[
2R,
\]
where \(R\) is a smooth rational curve containing four rational double points of type \(A_1\).
\end{enumerate}
\end{dfn}

\begin{prop}\label{prop:I0star-contraction}
We assume that \(F\) is of Kodaira type \(I_0^*\), with notation as above.
Let \(R\) be an ADER in \(F\), and let $\nu\colon X\rightarrow Y$
be the contraction of \(R\). Then all irreducible components of \(\overline F\) are smooth rational curves.  
Moreover, \(\overline F\) is described as follows.

\begin{enumerate}[\rm(1)]
\item Suppose that \(C_0\subset R\).  Then one of the following holds:
\begin{enumerate}[\rm(a)]
\item If \(R\) is of type \(A_1\), equivalently \(R=C_0\), then
\(\overline F\) is of type 
\[\mathfrak S_4(A_1)\]
in the sense of
Definition~\ref{dfn:new-I0star-fibers}. 

\item If \(R\) is of type \(A_2\), equivalently \(R=C_0+C_i\) for some \(i\),
then \(\overline F\) is a singular fiber of type 
\[\mathfrak S_3(A_2).\]  

\item If \(R\) is of type \(A_3\), equivalently $R=C_0+C_i+C_j$
for distinct \(i,j\), then \(\overline F\) is of type 
\[\mathfrak S_2(A_3)\] in
the sense of Definition~\ref{dfn:new-I0star-fibers}. 

\item If \(R\) is of type \(D_4\), equivalently $R=C_0+C_i+C_j+C_k$
for distinct \(i,j,k\), then \(\overline F\) is of type 
\[\mathfrak S_1(D_4)\] in
the sense of Definition~\ref{dfn:new-I0star-fibers}. 
\end{enumerate}

\item Suppose that \(C_0\not\subset R\).  Then \(R\) is of type
\(A_1^{\oplus s}\), where \(1\leq s\leq 4\).  In this case \(\overline F\) is
of type
\[
\mathfrak D_4^{\mathrm{mark}}(A_1),\quad
\mathfrak A_3^{\mathrm{mark},}(A_1^{\oplus 2}),\quad
\mathfrak A_2^{\mathrm{mark}}(A_1^{\oplus 3}),\quad
\mathfrak A_1^{\mathrm{mark}}(A_1^{\oplus 4})
\]
according as \(s=1,2,3,4\).
\end{enumerate}
\end{prop}

\begin{proof}
We first record the following observation.  Let \(C\) be an irreducible
component of \(F\) not contained in \(R\), and put \(E:=\nu(C)\).  Let
\(R'\) be a connected component of \(R\), and put \(y:=\nu(R')\).  If \(C\)
meets \(R'\) transversely at one point \(x\), then
\[
C\cap \nu^{-1}(y)=\{x\}
\qquad\text{and}\qquad
(\nu^{-1}(y)\cdot C)_x=1.
\]
The restriction \(\nu|_C\colon C\rightarrow E\) is the normalization of \(E\).
Therefore, by Proposition~\ref{prop:smoothness-criterion}, the curve \(E\) is
smooth at \(y\).  If \(C\) is disjoint from \(R'\), then \(E\) does not pass
through \(y\).  Since \(\nu\) is an isomorphism away from the exceptional
locus, it follows that \(E\) is a smooth rational curve whenever \(C\) meets
each connected component of \(R\) transversely in at most one point.

We assume that \(C_0\subset R\).  Then
\[
R = C_0 + C_{i_1} + \cdots + C_{i_q}
\]
for some \(0 \leq q \leq 3\).  By Proposition~\ref{prop:ADE-I0star}, the ADER
\(R\) is of type \(A_1, A_2, A_3\), or \(D_4\), according as
\(q=0,1,2,3\).  Hence the contraction of \(R\) produces a rational double
point \(y:=\nu(R)\) of the corresponding type.

The components of \(F\) not contained in \(R\) are precisely the \(4-q\) leaves
not contained in \(R\).  By Theorem~\ref{2.4}, the fiber \(\overline F\) is
obtained by deleting the components contained in \(R\), preserving
multiplicities.  By the observation above, the images of the remaining leaves
are smooth rational curves passing through \(y\).  Therefore \(\overline F\)
is of type
\[
\mathfrak S_4(A_1),\quad
\mathfrak S_3(A_2),\quad
\mathfrak S_2(A_3),\quad
\mathfrak S_1(D_4)
\]
according as \(q=0,1,2,3\).  This proves \textup{(1)}.

We assume that \(C_0\not\subset R\).  Then
\[
R=C_{i_1}+\cdots+C_{i_s}
\]
for some \(1\leq s\leq 4\).  Each \(C_{i_j}\) is contracted to a rational
double point \(y_j:=\nu(C_{i_j})\) of type \(A_1\), and these rational double
points are mutually distinct.

By Theorem~\ref{2.4}, the fiber \(\overline F\) is obtained by deleting the
components contained in \(R\), preserving multiplicities.  The rational double
points \(y_1,\dots,y_s\) lie on the component \(\nu(C_0)\).  By the observation
above, \(\nu(C_0)\) is a smooth rational curve.  The images of the remaining
leaves are also smooth rational curves, since those leaves are disjoint from
\(R\).  Moreover, each remaining leaf meets \(\nu(C_0)\) transversely at a
smooth point.  Hence, by Definition~\ref{dfn:new-I0star-fibers},
\(\overline F\) is of type
\[
\mathfrak D_4^{\mathrm{mark}}(A_1),\quad
\mathfrak A_3^{\mathrm{mark}}(A_1^{\oplus 2}),\quad
\mathfrak A_2^{\mathrm{mark}}(A_1^{\oplus 3}),\quad
\mathfrak A_1^{\mathrm{mark}}(A_1^{\oplus 4})
\]
according as \(s=1,2,3,4\).  This proves \textup{(2)}.
\end{proof}

\section{Fibers of type \(I_n^*\) for \(n\geq 1\)}
Fibers of type \(I_n^*\) correspond to affine Dynkin diagrams of type
\(\widetilde D_{n+4}\).  Elliptic \(K3\) surfaces with prescribed fibers
of type \(I_4^*\) were studied by Utsumi \cite{u10}, and large fibers of
type \(I_n^*\), in particular the maximal case \(I_{14}^*\), have been
studied in connection with maximal singular fibers of elliptic \(K3\)
surfaces \cite{s03,sch07,schs13}.

Throughout this section, we use the following notation. Let \(F\) be a singular
fiber of Kodaira type \(I_n^*\), where \(n\geq 1\). We write
\[
F=
C_0'+C_0''+2C_0+2C_1+\cdots+2C_n+C_n'+C_n'',
\]
where \(C_0,C_1,\dots,C_n\) form the central chain, \(C_0'\) and \(C_0''\)
meet \(C_0\), and \(C_n'\) and \(C_n''\) meet \(C_n\). The dual graph of \(F\)
is the affine Dynkin diagram \(\widetilde D_{n+4}\), and it is given as follows:
\[
\begin{tikzpicture}[baseline=(current bounding box.center)]
\node (C0p)  at (0,1) {$C_0'$};
\node (C0)   at (0,0) {$C_0$};
\node (C0pp) at (0,-1) {$C_0''$};

\node (C1)   at (1,0) {$C_1$};
\node (dots) at (2,0) {$\cdots$};
\node (Cn)   at (3,0) {$C_n$};

\node (Cnp)  at (3,1) {$C_n'$};
\node (Cnpp) at (3,-1) {$C_n''$};

\draw[thick] (C0.north) -- (C0p.south);
\draw[thick] (C0.south) -- (C0pp.north);

\draw[thick] (C0.east) -- (C1.west);
\draw[thick] (C1.east) -- (dots.west);
\draw[thick] (dots.east) -- (Cn.west);

\draw[thick] (Cn.north) -- (Cnp.south);
\draw[thick] (Cn.south) -- (Cnpp.north);
\end{tikzpicture}
\]

We classify ADERs in \(F\) according to the position of their support with
respect to the two end components \(C_0\) and \(C_n\) of the central chain.  We
divide the argument into the following three cases:
\begin{itemize}
\item \(C_0,C_n\not\subset \operatorname{Supp}(R)\);
\item \(C_0\not\subset \operatorname{Supp}(R)\) and
\(C_n\subset \operatorname{Supp}(R)\);
\item \(C_0,C_n\subset \operatorname{Supp}(R)\).
\end{itemize}
By symmetry, the second case also treats the case where
\[
C_0\subset \operatorname{Supp}(R),
\qquad
C_n\not\subset \operatorname{Supp}(R).
\]
For each case, we first classify the possible ADERs supported on \(F\), and then
describe the singular fiber obtained after contracting them.

\begin{lem}\label{lem:Instar-smooth-components}
We assume that \(F\) is of Kodaira type \(I_n^*\) with \(n\geq 1\).  Let \(R\)
be an ADER in \(F\), and let $\nu\colon X\rightarrow Y$
be the contraction of \(R\).  Let \(E\) be an irreducible component of \(F\) not
contained in \(R\). We put
\[
C:=\nu(E).
\]
Then \(C\) is a smooth rational curve.  In particular, every irreducible
component of \(\overline F\) is a smooth rational curve.
\end{lem}
\begin{proof}
Since \(E\) is a smooth rational curve and \(\nu\) is an isomorphism at the
generic point of \(E\), the curve \(C\) is rational, and the restriction
\[
\nu|_E\colon E\rightarrow C
\]
is the normalization of \(C\).
Let
\[
y\in C\cap \Sing(Y).
\]
Then \(\nu^{-1}(y)\) is a connected component of \(R\).  Since the dual graph of
an \(I_n^*\)-fiber is a tree, and since each connected component of \(R\) is a
subtree, the component \(E\) meets \(\nu^{-1}(y)\) in at most one point.  Since
\(y\in C\), this intersection is non-empty.  Hence
\[
E\cap \nu^{-1}(y)=\{x\}.
\]
Moreover, the components of an \(I_n^*\)-fiber meet transversely, and therefore
\[
(\nu^{-1}(y)\cdot E)_x=1.
\]
By Proposition~\ref{prop:smoothness-criterion}, the curve \(C\) is smooth at
\(y\).  Since \(\nu\) is an isomorphism away from the exceptional locus, \(C\)
is smooth away from \(\Sing(Y)\).  Thus \(C\) is a smooth rational curve.
By Theorem~\ref{2.4}, the irreducible components of \(\overline F\) are
precisely the images of the irreducible components of \(F\) not contained in
\(R\).  Hence every irreducible component of \(\overline F\) is a smooth
rational curve.
\end{proof}

\begin{prop}\label{prop:Instar-no-end}
We assume that \(F\) is of Kodaira type \(I_n^*\) with \(n\geq 1\), with
notation as above.  Let \(R\) be an ADER in \(F\).  We assume that
\[
C_0,\ C_n\not\subset \operatorname{Supp}(R).
\]
Then every connected component of \(R\) is of type \(A_k\).  More precisely,
one of the following holds.

\begin{enumerate}[\rm(1)]
\item If $\operatorname{Supp}(R)\subset
\operatorname{Supp}(C_1+\cdots+C_{n-1})$,
then \(R\) is of type
\[
A_{\lambda_1}\oplus\cdots\oplus A_{\lambda_r}
\]
for some \(\lambda_1,\dots,\lambda_r\geq 1\) satisfying
\[
\sum_{i=1}^r\lambda_i+r-1\leq n-1.
\]

\item We assume that $\operatorname{Supp}(R)\not\subset
\operatorname{Supp}(C_1+\cdots+C_{n-1})$.
Then, up to symmetry between the left and right ends, one of the following
occurs.

\begin{enumerate}[\rm(a)]
\item The divisor \(R\) contains leaves only at the left end.  In this case
\[
R=B_L+R_M,
\]
where
\[
B_L\in\{C_0',\,C_0'',\,C_0'+C_0''\},
\]
and \(R_M\) is supported on mutually disjoint consecutive subchains of
\[
C_1-C_2-\cdots-C_{n-1}.
\]
If \(R_M\neq 0\) and \(R_M\) is of type
\[
A_{\lambda_1}\oplus\cdots\oplus A_{\lambda_r},
\]
then
\[
\sum_{i=1}^r\lambda_i+r-1\leq n-1.
\]

\item The divisor \(R\) contains leaves at both ends.  In this case
\[
R=B_L+R_M+B_R,
\]
where
\[
B_L\in\{C_0',\,C_0'',\,C_0'+C_0''\},
\qquad
B_R\in\{C_n',\,C_n'',\,C_n'+C_n''\},
\]
and \(R_M\) is supported on mutually disjoint consecutive subchains of
\[
C_1-C_2-\cdots-C_{n-1}.
\]
If \(R_M\neq 0\) and \(R_M\) is of type
\[
A_{\lambda_1}\oplus\cdots\oplus A_{\lambda_r},
\]
then
\[
\sum_{i=1}^r\lambda_i+r-1\leq n-1.
\]
\end{enumerate}
\end{enumerate}

Conversely, every reduced subdivisor satisfying one of the above conditions is
an ADER.
\end{prop}

\begin{proof}
Since $C_0,\ C_n\not\subset \operatorname{Supp}(R)$,
the support of \(R\) is contained in the disjoint union
\[
\operatorname{Supp}(C_1+\cdots+C_{n-1})
\sqcup
\{C_0',C_0''\}
\sqcup
\{C_n',C_n''\}.
\]
Here the chain \(C_1-C_2-\cdots-C_{n-1}\) is understood to be empty if \(n=1\).
Let \(R_M\) be the part of \(R\) supported on the middle chain
\[
C_1-C_2-\cdots-C_{n-1}.
\]
Every connected component of \(R_M\) is a chain of consecutive vertices, and
hence is of type \(A_{\lambda_i}\) for some \(\lambda_i\geq 1\).  If
\(R_M\neq 0\) and \(R_M\) is of type
\[
A_{\lambda_1}\oplus\cdots\oplus A_{\lambda_r},
\]
then distinct connected components of \(R_M\) must be separated by at least one
unused vertex of the middle chain.  Therefore
\[
\sum_{i=1}^r\lambda_i+(r-1)\leq n-1.
\]
Since \(C_0\not\subset \operatorname{Supp}(R)\), the leaves \(C_0'\) and
\(C_0''\) are isolated from the middle chain and from each other.  Hence the
part of \(R\) supported on the left leaves is either \(0\) or one of
\[
C_0',\qquad C_0'',\qquad C_0'+C_0''.
\]
Similarly, since \(C_n\not\subset \operatorname{Supp}(R)\), the part of \(R\)
supported on the right leaves is either \(0\) or one of
\[
C_n',\qquad C_n'',\qquad C_n'+C_n''.
\]
If \(R\) is supported entirely on the middle chain, then we obtain
\textup{(1)}.  Otherwise \(R\) contains at least one leaf.  Up to symmetry
between the two ends, either \(R\) contains leaves only at the left end, or it contains leaves at both ends.  In the first case we obtain \textup{(2)(a)}, and
in the second case we obtain \textup{(2)(b)}.

Conversely, we suppose that a reduced subdivisor satisfies one of the conditions
\textup{(1)}, \textup{(2)(a)}, or \textup{(2)(b)}.  Then \(R_M\) is a disjoint
union of chains, and \(B_L,B_R\), if present, are disjoint unions of isolated
vertices.  Therefore the dual graph is a disjoint union of Dynkin diagrams of
type \(A\).  Hence the subdivisor is an ADER.
\end{proof}

\begin{dfn}\label{dfn:marked-fibers}
Let
\[
\Gamma,\ \Gamma_i,\ \Gamma_L,\ \Gamma_R,\ 
\Gamma_{L,i},\ \Gamma_{R,i}\quad (i=1,2)
\]
be ADE types.
We also allow a finite, possibly empty, list
\[
\Delta_1,\dots,\Delta_t
\]
of ADE types.
In the definitions below, the singularities of types \(\Gamma\), \(\Gamma_1\), \(\Gamma_2\), \(\Gamma_L\), \(\Gamma_R\), \(\Gamma_{L,i}\), and \(\Gamma_{R,i}\) are marked rational double points lying on the indicated irreducible components and not lying at the intersection points with other components. 
The singularities of types \(\Delta_1,\dots,\Delta_t\), when present, are rational double points lying at some of the intersection points of adjacent components.

\begin{enumerate}[\rm(1)]
\item Let \(m\geq 4\). We say that a singular fiber is of type 
\[ \mathfrak D_m^{\mathrm{mark},1}(\Gamma;\Delta_1,\dots,\Delta_t) \]
if, as a divisor, it is of the form \[ C_0+2C_1+\cdots+2C_{m-3}+C_{m-2}+C_{m-1}, \] where \(C_0,C_1,\dots,C_{m-1}\) are smooth rational curves whose support graph is as follows: 
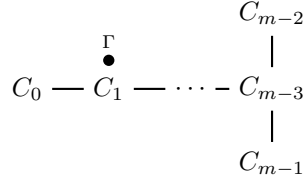
\[ \begin{tikzpicture}
[ baseline=(P), scale=0.9, rdp/.style={circle, fill, inner sep=1.5pt}, every label/.style={font=\scriptsize, inner sep=2pt} ] 
\useasboundingbox (-0.7,-1.6) rectangle (4.4,1.6); 
\coordinate (P) at (0,0); \node (C0) at (0,0) {$C_0$}; \node (C1) at (1.2,0) {$C_1$}; \node (dots) at (2.4,0) {$\cdots$}; \node (Cm3) at (3.6,0) {$C_{m-3}$}; \node (Cm2) at (3.6,1.1) {$C_{m-2}$}; \node (Cm1) at (3.6,-1.1) {$C_{m-1}$}; \draw[thick] (C0.east) -- (C1.west); \draw[thick] (C1.east) -- (dots.west); \draw[thick] (dots.east) -- (Cm3.west); \draw[thick] (Cm3.north) -- (Cm2.south); \draw[thick] (Cm3.south) -- (Cm1.north); \node[rdp, label=above:{$\Gamma$}] at (1.2,0.42) {}; 
\end{tikzpicture} \] 
Moreover, \(C_1\) contains a rational double point of type \(\Gamma\), and there may be rational double points of types \(\Delta_1,\dots,\Delta_t\) at some of the intersection points of adjacent components.

\item Let \(m\geq 4\). 
We say that a singular fiber is of type 
\[ \mathfrak D_m^{\mathrm{mark},2}(\Gamma_1,\Gamma_2;\Delta_1,\dots,\Delta_t) \] 
if, as a divisor, it is of the form 
\[ 2C_1+\cdots+2C_{m-2}+C_{m-1}+C_m, \]
where \(C_1,\dots,C_m\) are smooth rational curves whose support graph is as follows: 
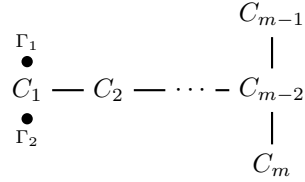
\[ \begin{tikzpicture}[ baseline=(P), scale=0.9, rdp/.style={circle, fill, inner sep=1.5pt}, every label/.style={font=\scriptsize, inner sep=2pt} ] \useasboundingbox (-0.7,-1.6) rectangle (4.4,1.6); \coordinate (P) at (0,0); \node (C1) at (0,0) {$C_1$}; \node (C2) at (1.2,0) {$C_2$}; \node (dots) at (2.4,0) {$\cdots$}; \node (Cm2) at (3.6,0) {$C_{m-2}$}; \node (Cm1) at (3.6,1.1) {$C_{m-1}$}; \node (Cm) at (3.6,-1.1) {$C_m$}; \draw[thick] (C1.east) -- (C2.west); \draw[thick] (C2.east) -- (dots.west); \draw[thick] (dots.east) -- (Cm2.west); \draw[thick] (Cm2.north) -- (Cm1.south); \draw[thick] (Cm2.south) -- (Cm.north); \node[rdp, label=above:{$\Gamma_1$}] at (0,0.42) {}; \node[rdp, label=below:{$\Gamma_2$}] at (0,-0.42) {}; \end{tikzpicture} \] 
Moreover, \(C_1\) contains two rational double points of types \(\Gamma_1\) and \(\Gamma_2\), and there may be rational double points of types \(\Delta_1,\dots,\Delta_t\) at some of the intersection points of adjacent components.

\item Let \(m\geq 3\). We say that a singular fiber is of type 
\[ \mathfrak A_m^{\mathrm{mark},1,1}(\Gamma_L,\Gamma_R;\Delta_1,\dots,\Delta_t) \]
if, as a divisor, it is of the form \[ C_0+2C_1+\cdots+2C_{m-2}+C_{m-1}, \] where \(C_0,C_1,\dots,C_{m-1}\) are smooth rational curves whose support graph is as follows: 
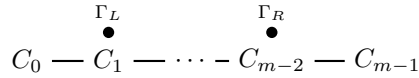
\[ \begin{tikzpicture}[ baseline=(P), scale=0.9, rdp/.style={circle, fill, inner sep=1.5pt}, every label/.style={font=\scriptsize, inner sep=2pt} ] \useasboundingbox (-0.7,-1.0) rectangle (5.2,1.0); \coordinate (P) at (0,0); \node (C0) at (0,0) {$C_0$}; \node (C1) at (1.2,0) {$C_1$}; \node (dots) at (2.4,0) {$\cdots$}; \node (Cm2) at (3.6,0) {$C_{m-2}$}; \node (Cm1) at (5.3,0) {$C_{m-1}$}; \draw[thick] (C0.east) -- (C1.west); \draw[thick] (C1.east) -- (dots.west); \draw[thick] (dots.east) -- (Cm2.west); \draw[thick] (Cm2.east) -- (Cm1.west); \node[rdp, label=above:{$\Gamma_L$}] at (1.2,0.42) {}; \node[rdp, label=above:{$\Gamma_R$}] at (3.6,0.42) {}; \end{tikzpicture} \] 
Moreover, \(C_1\) contains a rational double point of type \(\Gamma_L\), \(C_{m-2}\) contains a rational double point of type \(\Gamma_R\), and there may be rational double points of types \(\Delta_1,\dots,\Delta_t\) at some of the intersection points of adjacent components.

\item Let \(m\geq 2\). We say that a singular fiber is of type 
\[ \mathfrak A_m^{\mathrm{mark},1,0} (\Gamma_L;\Gamma_{R,1},\Gamma_{R,2};\Delta_1,\dots,\Delta_t) \]
if, as a divisor, it is of the form \[ C_0+2C_1+\cdots+2C_{m-1}, \] where \(C_0,C_1,\dots,C_{m-1}\) are smooth rational curves whose support graph is as follows: 
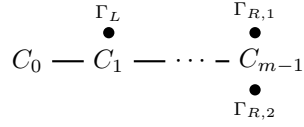
\[ \begin{tikzpicture}[ baseline=(P), scale=0.9, rdp/.style={circle, fill, inner sep=1.5pt}, every label/.style={font=\scriptsize, inner sep=2pt} ] \useasboundingbox (-0.7,-1.0) rectangle (4.4,1.0); 
\coordinate (P) at (0,0);
\node (C0) at (0,0) {$C_0$}; \node (C1) at (1.2,0) {$C_1$}; \node (dots) at (2.4,0) {$\cdots$}; \node (Cm1) at (3.6,0) {$C_{m-1}$}; 
\draw[thick] (C0.east) -- (C1.west); \draw[thick] (C1.east) -- (dots.west); \draw[thick] (dots.east) -- (Cm1.west); \node[rdp, label=above:{$\Gamma_L$}] at (1.2,0.42) {}; \node[rdp, label=above:{$\Gamma_{R,1}$}] at (3.35,0.42) {}; \node[rdp, label=below:{$\Gamma_{R,2}$}] at (3.35,-0.42) {}; \end{tikzpicture} \] 
Moreover, \(C_1\) contains a rational double point of type \(\Gamma_L\), \(C_{m-1}\) contains two rational double points of types \(\Gamma_{R,1}\) and \(\Gamma_{R,2}\), and there may be rational double points of types \(\Delta_1,\dots,\Delta_t\) at some of the intersection points of adjacent components. The symmetric type \[ \mathfrak A_m^{\mathrm{mark},0,1} (\Gamma_{L,1},\Gamma_{L,2};\Gamma_R;\Delta_1,\dots,\Delta_t) \] is defined similarly.

\item Let \(m\geq 1\). We say that a singular fiber is of type \[ \mathfrak A_m^{\mathrm{mark},0,0} (\Gamma_{L,1},\Gamma_{L,2};\Gamma_{R,1},\Gamma_{R,2}; \Delta_1,\dots,\Delta_t) \] if, as a divisor, it is of the form \[ 2C_1+\cdots+2C_m, \] where \(C_1,\dots,C_m\) are smooth rational curves whose support graph is as follows: 
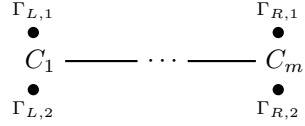
\[ \begin{tikzpicture}[ baseline=(P), scale=0.9, rdp/.style={circle, fill, inner sep=1.5pt}, every label/.style={font=\scriptsize, inner sep=2pt} ] \useasboundingbox (-0.7,-1.0) rectangle (4.4,1.0); 
\coordinate (P) at (0,0);
\node (C1) at (0,0) {$C_1$}; \node (dots) at (1.8,0) {$\cdots$}; \node (Cm) at (3.6,0) {$C_m$}; \draw[thick] (C1.east) -- (dots.west); \draw[thick] (dots.east) -- (Cm.west);
\node[rdp, label=above:{$\Gamma_{L,1}$}] at (-0.1,0.42) {}; \node[rdp, label=below:{$\Gamma_{L,2}$}] at (-0.1,-0.42) {}; 
\node[rdp, label=above:{$\Gamma_{R,1}$}] at (3.5,0.42) {}; \node[rdp, label=below:{$\Gamma_{R,2}$}] at (3.5,-0.42) {}; \end{tikzpicture} \] 
Moreover, \(C_1\) contains two rational double points of types \(\Gamma_{L,1}\) and \(\Gamma_{L,2}\), \(C_m\) contains two rational double points of types \(\Gamma_{R,1}\) and \(\Gamma_{R,2}\), and there may be rational double points of types \(\Delta_1,\dots,\Delta_t\) at some of the intersection points of adjacent components. 
\end{enumerate}

If there are no rational double points at the intersections of adjacent components, we omit the part \(\Delta_1,\dots,\Delta_t\) from the notation.
\end{dfn}

\begin{prop}\label{prop:Instar-no-end-contraction}
We assume that \(F\) is of Kodaira type \(I_n^*\) with \(n\geq 1\), with
notation as above.  Let \(R\) be an ADER in \(F\), and assume that
\[
C_0,\ C_n\not\subset \operatorname{Supp}(R).
\]
Let \(R_M\) be the part of \(R\) supported on the middle chain
\[
C_1-C_2-\cdots-C_{n-1}.
\]
We put
\[
q:=\#\operatorname{Irr}(R_M).
\]
If \(R_M=0\), we set \(q=0\).  If \(R_M\neq 0\), we write the type of \(R_M\)
as
\[
A_{\lambda_1}\oplus\cdots\oplus A_{\lambda_r}.
\]
Let $\nu\colon X\rightarrow Y$
be the contraction of \(R\).  Then every irreducible component of
\(\overline F\) is a smooth rational curve.  More precisely, the following hold.

\begin{enumerate}[\rm(1)]
\item We assume that \(R\) is supported only on the middle chain.  Then
\(R=R_M\).  If \(R\) is of type $A_{\lambda_1}\oplus\cdots\oplus A_{\lambda_r}$,
then the rational double points of types
\[
A_{\lambda_1},\dots,A_{\lambda_r}
\]
lie at some of the intersection points of adjacent components of
\(\overline F\).  Moreover, \(\overline F\) has underlying Kodaira type
\[
I_{n-q}^*.
\]

\item We assume that \(R\) contains leaves only at one end.  By symmetry, we may
assume that this end is the left end.  We write $R=B_L+R_M$, where $B_L\in\{C_0',\,C_0'',\,C_0'+C_0''\}$.
Then the following hold.
\begin{enumerate}[\rm(a)]
\item If \(B_L\) consists of one component, then the contracted leaf gives one
rational double point of type \(A_1\) lying on the image \(\nu(C_0)\).  Moreover,
if \(R_M=0\), then \(\overline F\) is of type
\[
\mathfrak D_{\,n-q+4}^{\mathrm{mark},1}(A_1),
\]
and if \(R_M\neq 0\), then \(\overline F\) is of type
\[
\mathfrak D_{\,n-q+4}^{\mathrm{mark},1}
(A_1;A_{\lambda_1},\dots,A_{\lambda_r})
\]
in the sense of Definition~\ref{dfn:marked-fibers}.

\item If \(B_L=C_0'+C_0''\), then the two contracted leaves give two rational
double points of type \(A_1\) lying on the image \(\nu(C_0)\).  Moreover, if
\(R_M=0\), then \(\overline F\) is of type
\[
\mathfrak D_{\,n-q+3}^{\mathrm{mark},2}(A_1,A_1),
\]
and if \(R_M\neq 0\), then \(\overline F\) is of type
\[
\mathfrak D_{\,n-q+3}^{\mathrm{mark},2}
(A_1,A_1;A_{\lambda_1},\dots,A_{\lambda_r})
\]
in the sense of Definition~\ref{dfn:marked-fibers}.
\end{enumerate}

\item We assume that \(R\) contains leaves at both ends.  We write $R=B_L+R_M+B_R$, where $B_L\in\{C_0',\,C_0'',\,C_0'+C_0''\}$ and $B_R\in\{C_n',\,C_n'',\,C_n'+C_n''\}$.
Let
\[
s_L:=2-\#\operatorname{Irr}(B_L),
\qquad
s_R:=2-\#\operatorname{Irr}(B_R).
\]
Thus \(s_L,s_R\in\{0,1\}\), and \(s_L\) and \(s_R\) are the numbers of
surviving leaves at the left and right ends, respectively.  Then the following
hold.
\begin{enumerate}[\rm(a)]
\item If \(s_L=s_R=1\), then one leaf survives at each end.  If \(R_M=0\), then
\(\overline F\) is of type
\[
\mathfrak A_{\,n-q+3}^{\mathrm{mark},1,1}(A_1,A_1),
\]
and if \(R_M\neq 0\), then \(\overline F\) is of type
\[
\mathfrak A_{\,n-q+3}^{\mathrm{mark},1,1}
(A_1,A_1;A_{\lambda_1},\dots,A_{\lambda_r})
\]
in the sense of Definition~\ref{dfn:marked-fibers}.

\item If \(s_L=1\) and \(s_R=0\), then one leaf survives at the left end and no
leaf survives at the right end.  If \(R_M=0\), then \(\overline F\) is of type
\[
\mathfrak A_{\,n-q+2}^{\mathrm{mark},1,0}(A_1;A_1,A_1),
\]
and if \(R_M\neq 0\), then \(\overline F\) is of type
\[
\mathfrak A_{\,n-q+2}^{\mathrm{mark},1,0}
(A_1;A_1,A_1;A_{\lambda_1},\dots,A_{\lambda_r})
\]
in the sense of Definition~\ref{dfn:marked-fibers}.

\item If \(s_L=0\) and \(s_R=1\), then no leaf survives at the left end and one leaf
survives at the right end.  If \(R_M=0\), then \(\overline F\) is of type
\[
\mathfrak A_{\,n-q+2}^{\mathrm{mark},0,1}(A_1,A_1;A_1),
\]
and if \(R_M\neq 0\), then \(\overline F\) is of type
\[
\mathfrak A_{\,n-q+2}^{\mathrm{mark},0,1}
(A_1,A_1;A_1;A_{\lambda_1},\dots,A_{\lambda_r})
\]
in the sense of Definition~\ref{dfn:marked-fibers}.

\item If \(s_L=s_R=0\), then no leaf survives at either end.  If \(R_M=0\),
then \(\overline F\) is of type
\[
\mathfrak A_{\,n-q+1}^{\mathrm{mark},0,0}(A_1,A_1;A_1,A_1),
\]
and if \(R_M\neq 0\), then \(\overline F\) is of type
\[
\mathfrak A_{\,n-q+1}^{\mathrm{mark},0,0}
(A_1,A_1;A_1,A_1;A_{\lambda_1},\dots,A_{\lambda_r})
\]
in the sense of Definition~\ref{dfn:marked-fibers}.
\end{enumerate}
\end{enumerate}
\end{prop}
\begin{proof}
The singular points of \(Y\) lying on \(\overline F\) are exactly the rational
double points obtained by contracting the connected components of \(R\).
By Lemma~\ref{lem:Instar-smooth-components}, every irreducible component of
\(\overline F\) is a smooth rational curve.  Moreover, by Theorem~\ref{2.4},
the fiber \(\overline F\) is obtained from \(F\) by deleting the components
contained in \(R\), while preserving the multiplicities of the remaining
components.

We now describe \(\overline F\) case by case.
First, we suppose that \(R\) is supported only on the middle chain.  Then
\(R=R_M\).  If \(R\) is of type $A_{\lambda_1}\oplus\cdots\oplus A_{\lambda_r}$,
then
\[
q=\#\operatorname{Irr}(R_M)=\sum_{i=1}^r\lambda_i.
\]
Contracting the connected components of \(R\) shortens the central chain by
\(q\) components.  The two forked ends remain unchanged.  Hence \(\overline F\)
has the same component configuration and multiplicities as a fiber of type
\[
I_{n-q}^*.
\]
The rational double points of types
$A_{\lambda_1},\dots,A_{\lambda_r}$
lie at some of the intersection points of adjacent components of
\(\overline F\).

Next, we suppose that \(R\) contains leaves only at one end.  By symmetry, we may
assume that this end is the left end.  We write
\[R=B_L+R_M,\]
where
$B_L\in\{C_0',\,C_0'',\,C_0'+C_0''\}$.
The right fork survives unchanged, while the left end changes according to the
number of leaves contained in \(B_L\).

If \(B_L\) consists of one component, then one left leaf survives and the other
left leaf is contracted.  The contracted leaf gives one rational double point of
type \(A_1\) lying on the image \(\nu(C_0)\).  The point \(\nu(B_L)\) does not
lie at an intersection point of two irreducible components of \(\overline F\).
After contracting \(R_M\), the remaining components have incidence graph of type
\(D_{n-q+4}\), with the multiplicities inherited from \(F\).  If \(R_M=0\),
then \(\overline F\) is of type
\[
\mathfrak D_{\,n-q+4}^{\mathrm{mark},1}(A_1).
\]
If \(R_M\neq0\), then the rational double points coming from the connected
components of \(R_M\) lie at some of the intersection points of adjacent
components, and \(\overline F\) is of type
\[
\mathfrak D_{\,n-q+4}^{\mathrm{mark},1}
(A_1;A_{\lambda_1},\dots,A_{\lambda_r})
\]
in the sense of Definition~\ref{dfn:marked-fibers}.

If $B_L=C_0'+C_0''$,
then both left leaves are contracted.  They give two rational double points of
type \(A_1\) lying on the image \(\nu(C_0)\).  These two points do not lie at
intersection points of irreducible components of \(\overline F\).  After
contracting \(R_M\), the remaining components have incidence graph of type
\(D_{n-q+3}\), with the multiplicities inherited from \(F\).  If \(R_M=0\),
then \(\overline F\) is of type
\[
\mathfrak D_{\,n-q+3}^{\mathrm{mark},2}(A_1,A_1).
\]
If \(R_M\neq0\), then the rational double points coming from the connected
components of \(R_M\) lie at some of the intersection points of adjacent
components, and \(\overline F\) is of type
\[
\mathfrak D_{\,n-q+3}^{\mathrm{mark},2}
(A_1,A_1;A_{\lambda_1},\dots,A_{\lambda_r})
\]
in the sense of Definition~\ref{dfn:marked-fibers}.

Finally, we suppose that \(R\) contains leaves at both ends.  We write
\[R=B_L+R_M+B_R,\]
where $B_L\in\{C_0',\,C_0'',\,C_0'+C_0''\}$ and $B_R\in\{C_n',\,C_n'',\,C_n'+C_n''\}$.
After contracting \(R_M\), the surviving part of the central chain has
\[
n+1-q
\]
irreducible components, all with multiplicity \(2\).  The remaining leaves, if
any, have multiplicity \(1\).  The rational double points obtained by
contracting the leaves in \(B_L\) and \(B_R\) lie on the images of the end
components \(C_0\) and \(C_n\), respectively, and they do not lie at intersection
points of irreducible components of \(\overline F\).

If \(s_L=s_R=1\), then one leaf survives at each end.  If \(R_M=0\), then
\(\overline F\) is of type
\[
\mathfrak A_{\,n-q+3}^{\mathrm{mark},1,1}(A_1,A_1).
\]
If \(R_M\neq0\), then the rational double points coming from the connected
components of \(R_M\) lie at some of the intersection points of adjacent
components, and \(\overline F\) is of type
\[
\mathfrak A_{\,n-q+3}^{\mathrm{mark},1,1}
(A_1,A_1;A_{\lambda_1},\dots,A_{\lambda_r})
\]
in the sense of Definition~\ref{dfn:marked-fibers}.

If \(s_L=1\) and \(s_R=0\), then one leaf survives at the left end and no leaf
survives at the right end.  If \(R_M=0\), then \(\overline F\) is of type
\[
\mathfrak A_{\,n-q+2}^{\mathrm{mark},1,0}(A_1;A_1,A_1).
\]
If \(R_M\neq0\), then the rational double points coming from the connected
components of \(R_M\) lie at some of the intersection points of adjacent
components, and \(\overline F\) is of type
\[
\mathfrak A_{\,n-q+2}^{\mathrm{mark},1,0}
(A_1;A_1,A_1;A_{\lambda_1},\dots,A_{\lambda_r})
\]
in the sense of Definition~\ref{dfn:marked-fibers}.

If \(s_L=0\) and \(s_R=1\), then no leaf survives at the left end and one leaf
survives at the right end.  If \(R_M=0\), then \(\overline F\) is of type
\[
\mathfrak A_{\,n-q+2}^{\mathrm{mark},0,1}(A_1,A_1;A_1).
\]
If \(R_M\neq0\), then the rational double points coming from the connected
components of \(R_M\) lie at some of the intersection points of adjacent
components, and \(\overline F\) is of type
\[
\mathfrak A_{\,n-q+2}^{\mathrm{mark},0,1}
(A_1,A_1;A_1;A_{\lambda_1},\dots,A_{\lambda_r})
\]
in the sense of Definition~\ref{dfn:marked-fibers}.

If \(s_L=s_R=0\), then no leaf survives at either end.  If \(R_M=0\), then
\(\overline F\) is of type
\[
\mathfrak A_{\,n-q+1}^{\mathrm{mark},0,0}(A_1,A_1;A_1,A_1).
\]
If \(R_M\neq0\), then the rational double points coming from the connected
components of \(R_M\) lie at some of the intersection points of adjacent
components, and \(\overline F\) is of type
\[
\mathfrak A_{\,n-q+1}^{\mathrm{mark},0,0}
(A_1,A_1;A_1,A_1;A_{\lambda_1},\dots,A_{\lambda_r})
\]
in the sense of Definition~\ref{dfn:marked-fibers}.
\end{proof}

Next, we consider the case where
\[
C_0\subset \operatorname{Supp}(R),
\qquad
C_n\not\subset \operatorname{Supp}(R).
\]
We denote by \(A_3^{\mathrm{sp}}\) the special \(A_3\)-configuration supported
at one end of an \(I_n^*\)-fiber:
\[
C_0'+C_0+C_0''
\]
at the left end, or symmetrically
\[
C_n'+C_n+C_n''
\]
at the right end.

Let \(\Gamma_L\) be the connected component of \(R\) containing \(C_0\).
We put
\[ c_L:= \# \bigl( \operatorname{Supp}(\Gamma_L)\cap\{C_0,\dots,C_n\} \bigr). \]
Since \(C_n\not\subset \operatorname{Supp}(R)\), the remaining part of \(R\) on the central chain is supported, after a gap, on the chain
\[ C_{c_L+1}-C_{c_L+2}-\cdots-C_{n-1}. \]
We denote this part by \(R_M\). We also denote by \(B_R\) the part of \(R\) supported on the right leaves \(C_n'\) and \(C_n''\). 
Thus
\[ R=\Gamma_L+R_M+B_R, \]
where
\[ B_R\in \{0,\ C_n',\ C_n'',\ C_n'+C_n''\}. \]
Here the chain \(C_{c_L+1}-\cdots-C_{n-1}\) is understood to be empty if \(c_L\geq n-1\).

\begin{prop}\label{prop:Instar-left-end}
We assume that \(F\) is of Kodaira type \(I_n^*\) with \(n\geq 1\), with
notation as above.  Let \(R\) be an ADER in \(F\), and assume that
\[
C_0\subset \operatorname{Supp}(R),
\qquad
C_n\not\subset \operatorname{Supp}(R).
\]
We write $R=\Gamma_L+R_M+B_R$
as above, where \(\Gamma_L\) is the connected component of \(R\) containing
\(C_0\), \(R_M\) is the part supported on the remaining middle chain, and
\(B_R\) is the part supported on the right leaves.

Then \(R_M\) is either zero or of type
\[
A_{\lambda_1}\oplus\cdots\oplus A_{\lambda_r}
\]
for some \(\lambda_1,\dots,\lambda_r\geq 1\) satisfying
\[
\sum_{i=1}^r\lambda_i+r\leq n-c_L.
\]
Moreover, \(\Gamma_L\) is one of the following types.
\begin{enumerate}[\rm(1)]
\item
\[
\Gamma_L=C_0+C_1+\cdots+C_{p-1}
\]
for some \(1\leq p\leq n\).  In this case \(\Gamma_L\) is of type \(A_p\).

\item
\[
\Gamma_L=L+C_0+C_1+\cdots+C_{p-1}
\]
for some $L\in\{C_0',C_0''\}$
and \(1\leq p\leq n\).  In this case \(\Gamma_L\) is of type \(A_{p+1}\).

\item
\[
\Gamma_L=C_0'+C_0+C_0''.
\]
In this case \(\Gamma_L\) is of type \(A_3^{\mathrm{sp}}\).

\item
\[
\Gamma_L=C_0'+C_0''+C_0+C_1+\cdots+C_{p-1}
\]
for some \(2\leq p\leq n\).  In this case \(\Gamma_L\) is of type \(D_{p+2}\).
\end{enumerate}
In cases \textup{(1)}, \textup{(2)}, and \textup{(4)}, one has \(c_L=p\),
while in case \textup{(3)} one has \(c_L=1\).

Conversely, every reduced subdivisor satisfying the above conditions is an
ADER.
\end{prop}

\begin{proof}
Since \(\Gamma_L\) is connected and \(C_n\not\subset\operatorname{Supp}(R)\),
the central-chain part of \(\Gamma_L\) is an initial segment
\(C_0+C_1+\cdots+C_{p-1}\) for some \(1\leq p\leq n\).
If \(\Gamma_L\)
contains no left leaf, then
\[
\Gamma_L=C_0+C_1+\cdots+C_{p-1},
\]
and \(\Gamma_L\) is of type \(A_p\).
If \(\Gamma_L\) contains exactly one left leaf $L\in\{C_0',C_0''\}$,
then
\[
\Gamma_L=L+C_0+C_1+\cdots+C_{p-1},
\]
and \(\Gamma_L\) is of type \(A_{p+1}\).
If \(\Gamma_L\) contains both left leaves and no component \(C_i\) with
\(i\geq 1\), then
\[
\Gamma_L=C_0'+C_0+C_0'',
\]
and \(\Gamma_L\) is of type \(A_3^{\mathrm{sp}}\).
Finally, if \(\Gamma_L\) contains both left leaves and also \(C_1\), then
\[
\Gamma_L=C_0'+C_0''+C_0+C_1+\cdots+C_{p-1}
\]
for some \(2\leq p\leq n\), and \(\Gamma_L\) is of type \(D_{p+2}\).  These are exactly the possibilities
for \(\Gamma_L\).

We assume that \(R_M\neq 0\).  Since \(R_M\) is a disjoint union of consecutive subchains of $C_{c_L+1}-C_{c_L+2}-\cdots-C_{n-1}$,
\(R_M\) is of type
\[
A_{\lambda_1}\oplus\cdots\oplus A_{\lambda_r}
\]
for some \(\lambda_1,\dots,\lambda_r\geq 1\).
The chain $C_{c_L+1}-C_{c_L+2}-\cdots-C_{n-1}$
has \(n-c_L-1\) vertices.  Since the \(r\) connected components of \(R_M\) are mutually disjoint and separated by at least one unused vertex, we have
\[
\sum_{i=1}^r\lambda_i+(r-1)\leq n-c_L-1.
\]
Equivalently, $\sum_{i=1}^r\lambda_i+r\leq n-c_L$.

Conversely, any reduced subdivisor satisfying the stated conditions is a
disjoint union of one Dynkin diagram of type \(A\) or \(D\), together with
possibly some additional \(A\)-type chains and isolated \(A_1\)'s.  Hence its
dual graph is a disjoint union of Dynkin diagrams of ADE type.  Therefore it is an ADER.
\end{proof}

\begin{dfn}\label{dfn:left-end-marked-fibers}
Let \(\Gamma\) be an ADE type, and let
\[
\Delta_1,\dots,\Delta_t
\]
be a finite, possibly empty, list of ADE types.  Let
\(\varepsilon\in\{0,1,2\}\).  The integer \(\varepsilon\) denotes the number of
contracted leaves at the right end.

In the diagrams below, chains with very few vertices are interpreted in the
evident way; the dotted part is omitted when there are no intermediate
vertices.  In particular, when \(m=0\) in
\(\mathfrak I_m^{*,\mathrm{fork},\varepsilon}\), the chain consists only of
\(C_0\).  If \(\varepsilon=2\), the terms
\(R_1+\cdots+R_{2-\varepsilon}\) are omitted.

In the definitions below, the rational double points of types
\(\Delta_1,\dots,\Delta_t\), when present, lie at some of the intersection
points of adjacent components in the indicated chain.  The marked rational
double points of type \(A_1\) at the right end are not intersection points of
two irreducible components of the fiber.

\begin{enumerate}[\rm(1)]
\item Let \(m\geq 0\).  We say that a singular fiber is of type
\[
\mathfrak I_m^{*,\mathrm{fork},\varepsilon}
(\Gamma;\Delta_1,\dots,\Delta_t)
\]
if, as a divisor, it is of the form
\[
L_1+L_2+2C_0+\cdots+2C_m+R_1+\cdots+R_{2-\varepsilon},
\]
where all components are smooth rational curves and the support graph is as
follows:
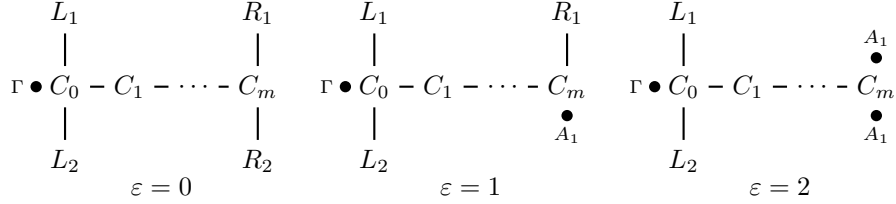
\[
\begin{array}{ccc}
\begin{tikzpicture}[
  baseline=(P),
  scale=0.85,
  rdp/.style={circle, fill, inner sep=1.5pt},
  every label/.style={font=\scriptsize, inner sep=2pt}
]
  \useasboundingbox (-1.1,-1.8) rectangle (3.3,1.8);
  \coordinate (P) at (0,0);

  \node (C0)   at (0,0) {$C_0$};
  \node (L1)   at (0,1.15) {$L_1$};
  \node (L2)   at (0,-1.15) {$L_2$};
  \node (C1)   at (1.0,0) {$C_1$};
  \node (dots) at (2.0,0) {$\cdots$};
  \node (Cm)   at (3.0,0) {$C_m$};
  \node (R1)   at (3.0,1.15) {$R_1$};
  \node (R2)   at (3.0,-1.15) {$R_2$};

  \draw[thick] (C0.north) -- (L1.south);
  \draw[thick] (C0.south) -- (L2.north);
  \draw[thick] (C0.east) -- (C1.west);
  \draw[thick] (C1.east) -- (dots.west);
  \draw[thick] (dots.east) -- (Cm.west);
  \draw[thick] (Cm.north) -- (R1.south);
  \draw[thick] (Cm.south) -- (R2.north);

  \node[rdp, label=left:{$\Gamma$}] at (-0.45,0) {};

\node at (1.5,-1.55) {$\varepsilon=0$};
\end{tikzpicture}
&
\begin{tikzpicture}[
  baseline=(P),
  scale=0.85,
  rdp/.style={circle, fill, inner sep=1.5pt},
  every label/.style={font=\scriptsize, inner sep=2pt}
]
  \useasboundingbox (-1.1,-1.8) rectangle (3.3,1.8);
  \coordinate (P) at (0,0);

  \node (C0)   at (0,0) {$C_0$};
  \node (L1)   at (0,1.15) {$L_1$};
  \node (L2)   at (0,-1.15) {$L_2$};
  \node (C1)   at (1.0,0) {$C_1$};
  \node (dots) at (2.0,0) {$\cdots$};
  \node (Cm)   at (3.0,0) {$C_m$};
  \node (R1)   at (3.0,1.15) {$R_1$};

  \draw[thick] (C0.north) -- (L1.south);
  \draw[thick] (C0.south) -- (L2.north);
  \draw[thick] (C0.east) -- (C1.west);
  \draw[thick] (C1.east) -- (dots.west);
  \draw[thick] (dots.east) -- (Cm.west);
  \draw[thick] (Cm.north) -- (R1.south);

  \node[rdp, label=left:{$\Gamma$}] at (-0.45,0) {};
  \node[rdp, label=below:{$A_1$}] at (3.0,-0.45) {};

  \node at (1.5,-1.55) {$\varepsilon=1$};
\end{tikzpicture}
&
\begin{tikzpicture}[
  baseline=(P),
  scale=0.85,
  rdp/.style={circle, fill, inner sep=1.5pt},
  every label/.style={font=\scriptsize, inner sep=2pt}
]
  \useasboundingbox (-1.1,-1.8) rectangle (3.3,1.8);
  \coordinate (P) at (0,0);

  \node (C0)   at (0,0) {$C_0$};
  \node (L1)   at (0,1.15) {$L_1$};
  \node (L2)   at (0,-1.15) {$L_2$};
  \node (C1)   at (1.0,0) {$C_1$};
  \node (dots) at (2.0,0) {$\cdots$};
  \node (Cm)   at (3.0,0) {$C_m$};

  \draw[thick] (C0.north) -- (L1.south);
  \draw[thick] (C0.south) -- (L2.north);
  \draw[thick] (C0.east) -- (C1.west);
  \draw[thick] (C1.east) -- (dots.west);
  \draw[thick] (dots.east) -- (Cm.west);

  \node[rdp, label=left:{$\Gamma$}] at (-0.45,0) {};
  \node[rdp, label=above:{$A_1$}] at (3.0,0.45) {};
  \node[rdp, label=below:{$A_1$}] at (3.0,-0.45) {};

  \node at (1.5,-1.55) {$\varepsilon=2$};
\end{tikzpicture}
\end{array}
\]
Here \(L_1\), \(L_2\), and \(C_0\) pass through one rational double point of
type \(\Gamma\).  The components \(R_1,\dots,R_{2-\varepsilon}\), if any, meet
\(C_m\) at smooth points of the surface.  The component \(C_m\) contains
\(\varepsilon\) marked rational double points of type \(A_1\), none of which is
an intersection point of two irreducible components of the fiber.  There may
also be rational double points of types \(\Delta_1,\dots,\Delta_t\) at some of
the intersection points of adjacent components in the chain
\(C_0-C_1-\cdots-C_m\).

\item Let \(m\geq 4\). We say that a singular fiber is of type \[ \mathfrak D_m^{\mathrm{edge},\varepsilon} (\Gamma;\Delta_1,\dots,\Delta_t) \] if, as a divisor, it is of the form 
\[ C_0+2C_1+\cdots+2C_{m-3}+R_1+\cdots+R_{2-\varepsilon}, \]
where all components are smooth rational curves and the support graph is as follows: 
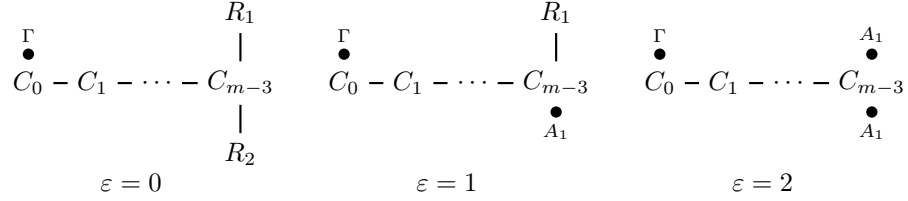
\[ \begin{array}{ccc} \begin{tikzpicture}[ baseline=(P), scale=0.85, rdp/.style={circle, fill, inner sep=1.5pt}, every label/.style={font=\scriptsize, inner sep=2pt} ] 
\useasboundingbox (-1.1,-1.8) rectangle (3.4,1.8);
\coordinate (P) at (0,0); 
\node (C0) at (0,0) {$C_0$};
\node (C1) at (1.0,0) {$C_1$};
\node (dots) at (2.0,0) {$\cdots$};
\node (Cm3) at (3.3,0) {$C_{m-3}$};
\node (R1) at (3.3,1.1) {$R_1$};
\node (R2) at (3.3,-1.1) {$R_2$}; 
\draw[thick] (C0.east) -- (C1.west);
\draw[thick] (C1.east) -- (dots.west);
\draw[thick] (dots.east) -- (Cm3.west);
\draw[thick] (Cm3.north) -- (R1.south);
\draw[thick] (Cm3.south) -- (R2.north);
\node[rdp, label=above:{$\Gamma$}] at (0,0.45) {};
\node at (1.6,-1.55) {$\varepsilon=0$}; \end{tikzpicture} & \begin{tikzpicture}[ baseline=(P), scale=0.85, rdp/.style={circle, fill, inner sep=1.5pt}, every label/.style={font=\scriptsize, inner sep=2pt} ]
\useasboundingbox (-1.1,-1.8) rectangle (3.4,1.8);
\coordinate (P) at (0,0);
\node (C0) at (0,0) {$C_0$};
\node (C1) at (1,0) {$C_1$}; 
\node (dots) at (2,0) {$\cdots$};
\node (Cm3) at (3.3,0) {$C_{m-3}$};
\node (R1) at (3.3,1.1) {$R_1$};
\draw[thick] (C0.east) -- (C1.west);
\draw[thick] (C1.east) -- (dots.west); 
\draw[thick] (dots.east) -- (Cm3.west);
\draw[thick] (Cm3.north) -- (R1.south);
\node[rdp, label=above:{$\Gamma$}] at (0,0.45) {}; \node[rdp, label=below:{$A_1$}] at (3.3,-0.45) {}; \node at (1.6,-1.55) {$\varepsilon=1$}; \end{tikzpicture} & \begin{tikzpicture}[ baseline=(P), scale=0.85, rdp/.style={circle, fill, inner sep=1.5pt}, every label/.style={font=\scriptsize, inner sep=2pt} ] 
\useasboundingbox (-1.1,-1.8) rectangle (3.4,1.8);
\coordinate (P) at (0,0); 
\node (C0) at (0,0) {$C_0$}; \node (C1) at (1,0) {$C_1$}; \node (dots) at (2,0) {$\cdots$}; \node (Cm3) at (3.3,0) {$C_{m-3}$}; \draw[thick] (C0.east) -- (C1.west); \draw[thick] (C1.east) -- (dots.west); \draw[thick] (dots.east) -- (Cm3.west); 
\node[rdp, label=above:{$\Gamma$}] at (0,0.45) {}; 
\node[rdp, label=above:{$A_1$}] at (3.3,0.45) {}; 
\node[rdp, label=below:{$A_1$}] at (3.3,-0.45) {}; \node at (1.6,-1.55) {$\varepsilon=2$}; \end{tikzpicture} \end{array} \] 
Here \(C_0\cap C_1\) is a rational double point of type \(\Gamma\). The components \(R_1,\dots,R_{2-\varepsilon}\), if any, meet \(C_{m-3}\) at smooth points of the surface. The component \(C_{m-3}\) contains \(\varepsilon\) marked rational double points of type \(A_1\). There may also be rational double points of types \(\Delta_1,\dots,\Delta_t\) at some of the other intersection points of adjacent components in the chain \(C_0-C_1-\cdots-C_{m-3}\). When \(\varepsilon=0\), the incidence graph is of type \(D_m\).

\item Let \(m\geq 3\). We say that a singular fiber is of type \[ \mathfrak D_m^{\mathrm{mark},\varepsilon} (\Gamma;\Delta_1,\dots,\Delta_t) \] if, as a divisor, it is of the form \[ 2C_1+\cdots+2C_{m-2}+R_1+\cdots+R_{2-\varepsilon}, \] where all components are smooth rational curves and the support graph is as follows: 
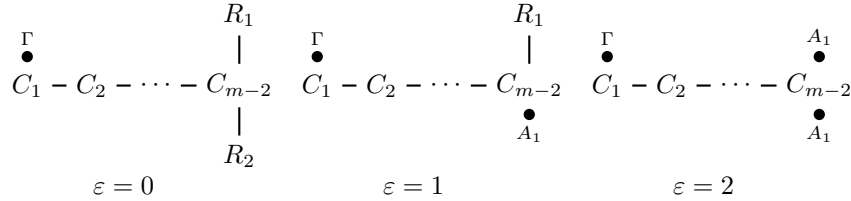
\[ \begin{array}{ccc} \begin{tikzpicture}[ baseline=(P), scale=0.85, rdp/.style={circle, fill, inner sep=1.5pt}, every label/.style={font=\scriptsize, inner sep=2pt} ] 
\useasboundingbox (-0.7,-1.8) rectangle (3.4,1.8); \coordinate (P) at (0,0); \node (C1) at (0,0) {$C_1$}; \node (C2) at (1,0) {$C_2$}; \node (dots) at (2,0) {$\cdots$}; \node (Cm2) at (3.3,0) {$C_{m-2}$}; \node (R1) at (3.3,1.1) {$R_1$}; \node (R2) at (3.3,-1.1) {$R_2$}; \draw[thick] (C1.east) -- (C2.west); \draw[thick] (C2.east) -- (dots.west); \draw[thick] (dots.east) -- (Cm2.west); \draw[thick] (Cm2.north) -- (R1.south); \draw[thick] (Cm2.south) -- (R2.north); \node[rdp, label=above:{$\Gamma$}] at (0,0.45) {}; \node at (1.5,-1.55) {$\varepsilon=0$}; \end{tikzpicture} & \begin{tikzpicture}[ baseline=(P), scale=0.85, rdp/.style={circle, fill, inner sep=1.5pt}, every label/.style={font=\scriptsize, inner sep=2pt} ] 
\useasboundingbox (-0.7,-1.8) rectangle (3.4,1.8); \coordinate (P) at (0,0); \node (C1) at (0,0) {$C_1$}; \node (C2) at (1,0) {$C_2$}; \node (dots) at (2,0) {$\cdots$}; \node (Cm2) at (3.3,0) {$C_{m-2}$}; \node (R1) at (3.3,1.1) {$R_1$}; \draw[thick] (C1.east) -- (C2.west); \draw[thick] (C2.east) -- (dots.west); \draw[thick] (dots.east) -- (Cm2.west); \draw[thick] (Cm2.north) -- (R1.south); \node[rdp, label=above:{$\Gamma$}] at (0,0.45) {}; \node[rdp, label=below:{$A_1$}] at (3.3,-0.45) {}; \node at (1.5,-1.55) {$\varepsilon=1$}; \end{tikzpicture} & \begin{tikzpicture}[ baseline=(P), scale=0.85, rdp/.style={circle, fill, inner sep=1.5pt}, every label/.style={font=\scriptsize, inner sep=2pt} ] 
\useasboundingbox (-0.7,-1.8) rectangle (3.4,1.8); \coordinate (P) at (0,0); \node (C1) at (0,0) {$C_1$}; \node (C2) at (1,0) {$C_2$}; \node (dots) at (2,0) {$\cdots$}; \node (Cm2) at (3.3,0) {$C_{m-2}$}; \draw[thick] (C1.east) -- (C2.west); \draw[thick] (C2.east) -- (dots.west); \draw[thick] (dots.east) -- (Cm2.west); \node[rdp, label=above:{$\Gamma$}] at (0,0.45) {}; \node[rdp, label=above:{$A_1$}] at (3.3,0.45) {}; \node[rdp, label=below:{$A_1$}] at (3.3,-0.45) {}; \node at (1.5,-1.55) {$\varepsilon=2$}; \end{tikzpicture} \end{array} \] Here \(C_1\) contains a marked rational double point of type \(\Gamma\). The components \(R_1,\dots,R_{2-\varepsilon}\), if any, meet \(C_{m-2}\) at smooth points of the surface. The component \(C_{m-2}\) contains \(\varepsilon\) marked rational double points of type \(A_1\). There may also be rational double points of types \(\Delta_1,\dots,\Delta_t\) at some of the intersection points of adjacent components in the chain \(C_1-C_2-\cdots-C_{m-2}\). When \(\varepsilon=0\), the incidence graph is of type \(D_m\), with the convention \(D_3=A_3\). \end{enumerate} If there are no rational double points at the intersections of adjacent components, we omit the part \(\Delta_1,\dots,\Delta_t\) from the notation. If \(\varepsilon=2\), the terms \(R_1+\cdots+R_{2-\varepsilon}\) are omitted.
\end{dfn}

For the contraction statement, we use the following notation.  We write
\[
R=\Gamma_L+R_M+B_R
\]
as in Proposition~\ref{prop:Instar-left-end}.  We put
\[
\varepsilon_R:=\#\operatorname{Irr}(B_R).
\]
Thus \(\varepsilon_R\in\{0,1,2\}\).  If \(R_M\neq 0\), we write the type of
\(R_M\) as
\[
A_{\lambda_1}\oplus\cdots\oplus A_{\lambda_r}
\]
and put
\[
q:=\sum_{i=1}^r\lambda_i.
\]
If \(R_M=0\), we set \(q=0\).

\begin{prop}\label{prop:Instar-left-end-contraction}
We assume that \(F\) is of Kodaira type \(I_n^*\) with \(n\geq 1\), with
notation as above.  Let \(R\) be an ADER in \(F\). We assume that
\[
C_0\subset \operatorname{Supp}(R),
\qquad
C_n\not\subset \operatorname{Supp}(R).
\]
We write $R=\Gamma_L+R_M+B_R$ as above.
Let $\nu\colon X\rightarrow Y$
be the contraction of \(R\).  Then every irreducible component of
\(\overline F\) is a smooth rational curve.  More precisely, the following
hold.

\begin{enumerate}[\rm(1)]
\item We assume that $\Gamma_L=C_0+C_1+\cdots+C_{p-1}$
for some \(1\leq p\leq n\).  Then \(\Gamma_L\) is of type \(A_p\).  If
\(R_M=0\), then \(\overline F\) is of type
\[
\mathfrak I_{\,n-p-q}^{*,\mathrm{fork},\varepsilon_R}(A_p)
\]
in the sense of Definition~\ref{dfn:left-end-marked-fibers}.  If
\(R_M\neq 0\), then \(\overline F\) is of type
\[
\mathfrak I_{\,n-p-q}^{*,\mathrm{fork},\varepsilon_R}
(A_p;A_{\lambda_1},\dots,A_{\lambda_r})
\]
in the sense of Definition~\ref{dfn:left-end-marked-fibers}.

\item We assume that $\Gamma_L=L+C_0+C_1+\cdots+C_{p-1}$ for some
$L\in\{C_0',C_0''\}$ and \(1\leq p\leq n\).  Then \(\Gamma_L\) is of type \(A_{p+1}\).  If
\(R_M=0\), then \(\overline F\) is of type
\[
\mathfrak D_{\,n-p+4-q}^{\mathrm{edge},\varepsilon_R}(A_{p+1})
\]
in the sense of Definition~\ref{dfn:left-end-marked-fibers}.  If
\(R_M\neq 0\), then \(\overline F\) is of type
\[
\mathfrak D_{\,n-p+4-q}^{\mathrm{edge},\varepsilon_R}
(A_{p+1};A_{\lambda_1},\dots,A_{\lambda_r})
\]
in the sense of Definition~\ref{dfn:left-end-marked-fibers}.

\item We assume that $\Gamma_L=C_0'+C_0+C_0''$.
Then \(\Gamma_L\) is of type \(A_3^{\mathrm{sp}}\).  If \(R_M=0\), then
\(\overline F\) is of type
\[
\mathfrak D_{\,n-q+2}^{\mathrm{mark},\varepsilon_R}(A_3)
\]
in the sense of Definition~\ref{dfn:left-end-marked-fibers}.  If
\(R_M\neq 0\), then \(\overline F\) is of type
\[
\mathfrak D_{\,n-q+2}^{\mathrm{mark},\varepsilon_R}
(A_3;A_{\lambda_1},\dots,A_{\lambda_r})
\]
in the sense of Definition~\ref{dfn:left-end-marked-fibers}.

\item We assume that $\Gamma_L=C_0'+C_0''+C_0+C_1+\cdots+C_{p-1}$
for some \(2\leq p\leq n\).  Then \(\Gamma_L\) is of type \(D_{p+2}\).  If
\(R_M=0\), then \(\overline F\) is of type
\[
\mathfrak D_{\,n-p+3-q}^{\mathrm{mark},\varepsilon_R}(D_{p+2})
\]
in the sense of Definition~\ref{dfn:left-end-marked-fibers}.  If
\(R_M\neq 0\), then \(\overline F\) is of type
\[
\mathfrak D_{\,n-p+3-q}^{\mathrm{mark},\varepsilon_R}
(D_{p+2};A_{\lambda_1},\dots,A_{\lambda_r})
\]
in the sense of Definition~\ref{dfn:left-end-marked-fibers}.
\end{enumerate}
\end{prop}

\begin{proof}
The singular points of \(Y\) lying on \(\overline F\) are exactly the rational
double points obtained by contracting the connected components of \(R\).

By Lemma~\ref{lem:Instar-smooth-components}, every irreducible component of
\(\overline F\) is a smooth rational curve.  Moreover, by Theorem~\ref{2.4},
the fiber \(\overline F\) is obtained from \(F\) by deleting the components
contained in \(R\), while preserving the multiplicities of the remaining
components.

We describe \(\overline F\) according to the four possibilities for
\(\Gamma_L\).
We suppose that $\Gamma_L=C_0+C_1+\cdots+C_{p-1}$.
Then \(\Gamma_L\) is of type \(A_p\).  The two left leaves \(C_0'\) and
\(C_0''\) survive, and the first surviving component on the central chain is the
image of \(C_p\).  These three components pass through the rational double point
of type \(A_p\) obtained by contracting \(\Gamma_L\).  After contracting
\(R_M\), the surviving central chain has
\[
n-p+1-q
\]
components of multiplicity \(2\).  At the right end, exactly
\(\varepsilon_R\) leaves are contracted, producing \(\varepsilon_R\) marked
rational double points of type \(A_1\) on the image of the right end component.
If \(R_M=0\), then \(\overline F\) is of type
\[
\mathfrak I_{\,n-p-q}^{*,\mathrm{fork},\varepsilon_R}(A_p).
\]
If \(R_M\neq0\), then the rational double points coming from the connected
components of \(R_M\) have types
\[
A_{\lambda_1},\dots,A_{\lambda_r}
\]
and lie at some of the intersection points of adjacent components in the
surviving central chain.  Hence \(\overline F\) is of type
\[
\mathfrak I_{\,n-p-q}^{*,\mathrm{fork},\varepsilon_R}
(A_p;A_{\lambda_1},\dots,A_{\lambda_r}).
\]

Next, we suppose that $\Gamma_L=L+C_0+C_1+\cdots+C_{p-1}$ for some $L\in\{C_0',C_0''\}$.
Then \(\Gamma_L\) is of type \(A_{p+1}\).  Exactly one left leaf survives.  The
image of this leaf and the first surviving central component meet at the
rational double point of type \(A_{p+1}\) obtained by contracting
\(\Gamma_L\).  After contracting \(R_M\), the surviving central chain has
\[
n-p+1-q
\]
components of multiplicity \(2\).  At the right end, exactly
\(\varepsilon_R\) leaves are contracted, producing \(\varepsilon_R\) marked
rational double points of type \(A_1\) on the image of the right end component.
If \(R_M=0\), then \(\overline F\) is of type
\[
\mathfrak D_{\,n-p+4-q}^{\mathrm{edge},\varepsilon_R}(A_{p+1}).
\]
If \(R_M\neq0\), then the rational double points coming from the connected
components of \(R_M\) lie at some of the intersection points of adjacent
components.  Hence \(\overline F\) is of type
\[
\mathfrak D_{\,n-p+4-q}^{\mathrm{edge},\varepsilon_R}
(A_{p+1};A_{\lambda_1},\dots,A_{\lambda_r}).
\]

Next, we suppose that $
\Gamma_L=C_0'+C_0+C_0''$.
Then \(\Gamma_L\) is of type \(A_3^{\mathrm{sp}}\).  The rational double point
of type \(A_3\) obtained by contracting \(\Gamma_L\) lies as a marked point on
the image of the first surviving central component \(C_1\).  After contracting
\(R_M\), the surviving central chain has
\[
n-q
\]
components of multiplicity \(2\).  At the right end, exactly
\(\varepsilon_R\) leaves are contracted, producing \(\varepsilon_R\) marked
rational double points of type \(A_1\) on the image of the right end component.
If \(R_M=0\), then \(\overline F\) is of type
\[
\mathfrak D_{\,n-q+2}^{\mathrm{mark},\varepsilon_R}(A_3).
\]
If \(R_M\neq0\), then the rational double points coming from the connected
components of \(R_M\) lie at some of the intersection points of adjacent
components.  Hence \(\overline F\) is of type
\[
\mathfrak D_{\,n-q+2}^{\mathrm{mark},\varepsilon_R}
(A_3;A_{\lambda_1},\dots,A_{\lambda_r}).
\]

Finally, we suppose that $\Gamma_L=C_0'+C_0''+C_0+C_1+\cdots+C_{p-1}$
for some \(2\leq p\leq n\).  Then \(\Gamma_L\) is of type \(D_{p+2}\).  The
rational double point of type \(D_{p+2}\) obtained by contracting \(\Gamma_L\)
lies as a marked point on the image of the first surviving central component
\(C_p\).  After contracting \(R_M\), the surviving central chain has
\[
n-p+1-q
\]
components of multiplicity \(2\).  At the right end, exactly
\(\varepsilon_R\) leaves are contracted, producing \(\varepsilon_R\) marked
rational double points of type \(A_1\) on the image of the right end component.
If \(R_M=0\), then \(\overline F\) is of type
\[
\mathfrak D_{\,n-p+3-q}^{\mathrm{mark},\varepsilon_R}(D_{p+2}).
\]
If \(R_M\neq0\), then the rational double points coming from the connected
components of \(R_M\) lie at some of the intersection points of adjacent
components.  Hence \(\overline F\) is of type
\[
\mathfrak D_{\,n-p+3-q}^{\mathrm{mark},\varepsilon_R}
(D_{p+2};A_{\lambda_1},\dots,A_{\lambda_r}).
\]
\end{proof}

We finally treat the case where the support of \(R\) contains both end
components of the central chain:
\[
C_0,\ C_n\subset \operatorname{Supp}(R).
\]

When the connected component of \(R\) containing \(C_0\) is different from the
connected component of \(R\) containing \(C_n\), we use the following notation.
Let \(\Gamma_L\) be the connected component of \(R\) containing \(C_0\), and let
\(\Gamma_R\) be the connected component of \(R\) containing \(C_n\).  We put
\[
c_L:=
\#
\bigl(
\operatorname{Supp}(\Gamma_L)\cap\{C_0,\dots,C_n\}
\bigr),
\qquad
c_R:=
\#
\bigl(
\operatorname{Supp}(\Gamma_R)\cap\{C_0,\dots,C_n\}
\bigr).
\]
Then the remaining part of \(R\) on the central chain is supported, after gaps,
on the chain
\[
C_{c_L+1}-C_{c_L+2}-\cdots-C_{n-c_R-1}.
\]
We denote this part by \(R_M\).  Thus, in this case, we write
\[
R=\Gamma_L+R_M+\Gamma_R.
\]
Here the chain \(C_{c_L+1}-\cdots-C_{n-c_R-1}\) is understood to be empty if
\(c_L+c_R\geq n-1\).
\begin{prop}\label{prop:Instar-both-ends}
We assume that \(F\) is of Kodaira type \(I_n^*\) with \(n\geq 1\), with
notation as above.  Let \(R\) be an ADER in \(F\), and assume that
\[
C_0,\ C_n\subset \operatorname{Supp}(R).
\]
Then exactly one of the following two cases occurs.

\begin{enumerate}[\rm(1)]
\item The connected component of \(R\) containing \(C_0\) is different from the
connected component of \(R\) containing \(C_n\).

In this case, with the notation above, we have
\[
R=\Gamma_L+R_M+\Gamma_R.
\]
The divisor \(R_M\) is either zero or of type
\[
A_{\lambda_1}\oplus\cdots\oplus A_{\lambda_r}
\]
for some \(\lambda_1,\dots,\lambda_r\geq 1\) satisfying
\[
\sum_{i=1}^r\lambda_i+r\leq n-c_L-c_R.
\]
If \(R_M=0\), then
\[
c_L+c_R\leq n.
\]

Moreover, \(\Gamma_L\) is one of the following types:
\begin{enumerate}[\rm(a)]
\item
\[
\Gamma_L=C_0+C_1+\cdots+C_{p_L-1}
\]
for some \(1\leq p_L\leq n\).  In this case \(\Gamma_L\) is of type
\(A_{p_L}\).

\item
\[
\Gamma_L=L+C_0+C_1+\cdots+C_{p_L-1}
\]
for some $L\in\{C_0',C_0''\}$
and \(1\leq p_L\leq n\).  In this case \(\Gamma_L\) is of type
\(A_{p_L+1}\).

\item
\[
\Gamma_L=C_0'+C_0+C_0''.
\]
In this case \(\Gamma_L\) is of type \(A_3^{\mathrm{sp}}\).

\item
\[
\Gamma_L=C_0'+C_0''+C_0+C_1+\cdots+C_{p_L-1}
\]
for some \(2\leq p_L\leq n\).  In this case \(\Gamma_L\) is of type
\(D_{p_L+2}\).
\end{enumerate}
In cases \textup{(a)}, \textup{(b)}, and \textup{(d)}, one has
\(c_L=p_L\), while in case \textup{(c)} one has \(c_L=1\).

The connected component \(\Gamma_R\) is described symmetrically at the right
end.  We define \(p_R\) and \(c_R\) in the same way.  In particular,
\(c_R=p_R\) in the three non-special cases, while \(c_R=1\) in the special case
\[
\Gamma_R=C_n'+C_n+C_n'',
\]
where \(\Gamma_R\) is of type \(A_3^{\mathrm{sp}}\).

\item A single connected component of \(R\) contains both \(C_0\) and \(C_n\).

Then \(R\) is connected and one of the following occurs:
\begin{enumerate}[\rm(a)]
\item
\[
R=C_0+C_1+\cdots+C_n.
\]
In this case \(R\) is of type \(A_{n+1}\).

\item
\[
R=L+C_0+C_1+\cdots+C_n
\]
for some $L\in\{C_0',C_0'',C_n',C_n''\}$.
In this case \(R\) is of type \(A_{n+2}\).

\item
\[
R=L_0+C_0+C_1+\cdots+C_n+L_n
\]
for some $L_0\in\{C_0',C_0''\}$ and $L_n\in\{C_n',C_n''\}$.
In this case \(R\) is of type \(A_{n+3}\).

\item
\[
R=C_0'+C_0''+C_0+C_1+\cdots+C_n,
\]
or symmetrically
\[
R=C_0+C_1+\cdots+C_n+C_n'+C_n''.
\]
In this case \(R\) is of type \(D_{n+3}\).

\item
\[
R=C_0'+C_0''+C_0+C_1+\cdots+C_n+L_n
\]
for some $L_n\in\{C_n',C_n''\}$,
or symmetrically
\[
R=L_0+C_0+C_1+\cdots+C_n+C_n'+C_n''
\]
for some $L_0\in\{C_0',C_0''\}$.
In this case \(R\) is of type \(D_{n+4}\).
\end{enumerate}
\end{enumerate}

Conversely, every reduced subdivisor satisfying one of the above conditions is
an ADER.
\end{prop}
\begin{proof}
Let \(\Gamma_L\) be the connected component of \(R\) containing \(C_0\), and let
\(\Gamma_R\) be the connected component of \(R\) containing \(C_n\).

First, we suppose that \(\Gamma_L\neq \Gamma_R\).  Since the dual graph of an
\(I_n^*\)-fiber is a tree, the intersection of
\(\operatorname{Supp}(\Gamma_L)\) with the central chain is an initial segment $C_0+C_1+\cdots+C_{p_L-1}$,
except for the special configuration $C_0'+C_0+C_0''$.
Thus, by the same argument as in Proposition~\ref{prop:Instar-left-end}, the
component \(\Gamma_L\) is one of the four listed left-end types.  Similarly,
\(\Gamma_R\) is one of the corresponding right-end types.
Since \(\Gamma_L\) and \(\Gamma_R\) are distinct connected components, at least
one central component must be omitted between them.  Therefore the remaining
central part of \(R\), if nonzero, is supported on mutually disjoint consecutive
subchains of $C_{c_L+1}-C_{c_L+2}-\cdots-C_{n-c_R-1}$.
If \(R_M\neq 0\), then \(R_M\) is of type
\[
A_{\lambda_1}\oplus\cdots\oplus A_{\lambda_r}
\]
for some \(\lambda_1,\dots,\lambda_r\geq 1\).  The chain $C_{c_L+1}-C_{c_L+2}-\cdots-C_{n-c_R-1}$
has \(n-c_L-c_R-1\) vertices.  Since the \(r\) connected components of \(R_M\)
are mutually disjoint and separated by at least one unused vertex, we have
\[
\sum_{i=1}^r\lambda_i+(r-1)\leq n-c_L-c_R-1.
\]
Equivalently, $\sum_{i=1}^r\lambda_i+r\leq n-c_L-c_R$.
If \(R_M=0\), the condition that \(\Gamma_L\) and \(\Gamma_R\) are distinct
gives
\[
c_L+c_R\leq n.
\]
This proves \textup{(1)}.

Next, we suppose that \(\Gamma_L=\Gamma_R\).  Since this connected component
contains both \(C_0\) and \(C_n\), it contains the whole central chain
\[
C_0+C_1+\cdots+C_n.
\]
Thus any leaf contained in \(R\) is attached to this same connected component,
and hence \(R\) is connected.
Let \(\alpha_L\) be the number of left leaves among \(C_0',C_0''\) contained in
\(R\), and let \(\alpha_R\) be the number of right leaves among \(C_n',C_n''\)
contained in \(R\).  Since \(R\) is an ADER, its connected dual graph is a
Dynkin diagram.  In particular, it cannot have two vertices of valency \(3\).
Therefore we cannot have
\[
\alpha_L=\alpha_R=2.
\]
If \(\alpha_L,\alpha_R\leq 1\), then the dual graph of \(R\) is a chain.
According as
\[
\alpha_L+\alpha_R=0,\quad 1,\quad 2,
\]
we obtain types
\[
A_{n+1},\qquad A_{n+2},\qquad A_{n+3}.
\]
The case \(\alpha_L+\alpha_R=2\) means that one leaf is chosen at each end.
If exactly one of \(\alpha_L,\alpha_R\) is equal to \(2\), then the dual graph
has a unique branching vertex.  If no leaf is chosen at the other end, the type
is \(D_{n+3}\).  If one leaf is chosen at the other end, the type is
\(D_{n+4}\).  These are precisely the cases listed in \textup{(2)}.

Conversely, every subdivisor described in \textup{(1)} is a disjoint union of
Dynkin diagrams of type \(A\) and possibly \(D\), and every subdivisor described
in \textup{(2)} is a Dynkin diagram of type \(A\) or \(D\).  Hence in all cases
the dual graph is a disjoint union of Dynkin diagrams of ADE type.  Therefore
the subdivisor is an ADER.
\end{proof}

\begin{dfn}\label{dfn:both-end-marked-fibers}
Let \(\Gamma_L\) and \(\Gamma_R\) be ADE types, and let
\[
\tau_L,\tau_R\in\{\mathrm{fork},\mathrm{edge},\mathrm{mark}\}.
\]
The symbols \(\tau_L\) and \(\tau_R\) describe the left and right end behavior,
respectively.  We also allow a finite, possibly empty, list
\[
\Delta_1,\dots,\Delta_t
\]
of ADE types.

In the definitions below, the rational double points specified by
\(\Gamma_L\) and \(\Gamma_R\) describe the singularities at the left and right
ends, respectively.  These two singular points are distinct.  The rational
double points of types \(\Delta_1,\dots,\Delta_t\), when present, lie at some of
the intersection points of adjacent components in the central chain.

In the diagrams below, chains with very few vertices are interpreted in the
evident way; the dotted part is omitted when there are no intermediate
vertices.  In particular, when \(m=1\), the left and right end contributions
are attached to the same component \(C_1\), and the two rational double points
specified by the left and right end data are required to be distinct.

\begin{enumerate}[\rm(1)]
\item Let \(m\geq 1\).  We say that a singular fiber is of type
\[
\mathfrak A_m^{\tau_L,\tau_R}
(\Gamma_L,\Gamma_R;\Delta_1,\dots,\Delta_t)
\]
if, as a divisor, it is of the form
\[
2C_1+\cdots+2C_m+L(\tau_L)+R(\tau_R),
\]
where \(C_1,\dots,C_m\) are smooth rational curves forming the central chain
\[
C_1-C_2-\cdots-C_m.
\]

The left end contribution \(L(\tau_L)\) is defined as follows:
\[
\begin{array}{ccc}
\begin{tikzpicture}[
  baseline=(P),
  scale=0.9,
  rdp/.style={circle, fill, inner sep=1.5pt},
  every label/.style={font=\scriptsize, inner sep=2pt}
]
  \useasboundingbox (-1.4,-1.45) rectangle (2.1,1.45);
  \coordinate (P) at (0,0);

  \node (C1) at (0,0) {$C_1$};
  \node (L1) at (0,1.0) {$L_1$};
  \node (L2) at (0,-1.0) {$L_2$};
  \node (C2) at (1.2,0) {$C_2$};

  \draw[thick] (C1.north) -- (L1.south);
  \draw[thick] (C1.south) -- (L2.north);
  \draw[thick] (C1.east) -- (C2.west);

  \node[rdp, label=left:{$\Gamma_L$}] at (-0.45,0) {};
  \node at (0,-1.35) {$\tau_L=\mathrm{fork}$};
\end{tikzpicture}
&
\begin{tikzpicture}[
  baseline=(P),
  scale=0.9,
  rdp/.style={circle, fill, inner sep=1.5pt},
  every label/.style={font=\scriptsize, inner sep=2pt}
]
  \useasboundingbox (-1.4,-1.45) rectangle (2.1,1.45);
  \coordinate (P) at (0,0);

  \node (L)  at (-1.0,0) {$L$};
  \node (C1) at (0,0) {$C_1$};
  \node (C2) at (1.2,0) {$C_2$};

  \draw[thick] (L.east) -- (C1.west);
  \draw[thick] (C1.east) -- (C2.west);

  \node[rdp, label=above:{$\Gamma_L$}] at (0,0.45) {};
  \node at (0,-1.35) {$\tau_L=\mathrm{edge}$};
\end{tikzpicture}
&
\begin{tikzpicture}[
  baseline=(P),
  scale=0.9,
  rdp/.style={circle, fill, inner sep=1.5pt},
  every label/.style={font=\scriptsize, inner sep=2pt}
]
  \useasboundingbox (-1.4,-1.45) rectangle (2.1,1.45);
  \coordinate (P) at (0,0);

  \node (C1) at (0,0) {$C_1$};
  \node (C2) at (1.2,0) {$C_2$};

  \draw[thick] (C1.east) -- (C2.west);

  \node[rdp, label=above:{$\Gamma_L$}] at (0,0.45) {};
  \node at (0,-1.35) {$\tau_L=\mathrm{mark}$};
\end{tikzpicture}
\end{array}
\]

More precisely, 
\begin{enumerate}[\rm(a)]
\item if \(\tau_L=\mathrm{fork}\), then
\[
L(\tau_L)=L_1+L_2,
\]
and the three curves \(L_1\), \(L_2\), and \(C_1\) pass through one rational
double point of type \(\Gamma_L\).
\item If \(\tau_L=\mathrm{edge}\), then
\[
L(\tau_L)=L,
\]
and \(L\) meets \(C_1\) at a rational double point of type \(\Gamma_L\).  
\item If
\(\tau_L=\mathrm{mark}\), then
\[
L(\tau_L)=0,
\]
and \(C_1\) contains a rational double point of type \(\Gamma_L\), which is not
an intersection point of two irreducible components of the fiber.
\end{enumerate}

\item The right end contribution \(R(\tau_R)\) is defined symmetrically, by replacing
\[
(C_1,C_2,L_1,L_2,L,\Gamma_L,\tau_L)
\]
with
\[
(C_m,C_{m-1},R_1,R_2,R,\Gamma_R,\tau_R),
\]
respectively.
Moreover, there may be rational double points of types
\(\Delta_1,\dots,\Delta_t\) at some of the intersection points of adjacent
components in the central chain
\[
C_1-C_2-\cdots-C_m.
\]
\end{enumerate}

If there are no rational double points at the intersections of adjacent
components in the central chain, we omit the part
\[
\Delta_1,\dots,\Delta_t
\]
from the notation.
\end{dfn}

For the contraction statement in the case of
Proposition~\ref{prop:Instar-both-ends}\textup{(1)}, we use the following
notation.  We write
\[
R=\Gamma_L+R_M+\Gamma_R
\]
as in Proposition~\ref{prop:Instar-both-ends}\textup{(1)}.  If \(R_M\neq 0\),
we write the type of \(R_M\) as
\[
A_{\lambda_1}\oplus\cdots\oplus A_{\lambda_r}
\]
and put
\[
q:=\sum_{i=1}^r\lambda_i.
\]
If \(R_M=0\), we set \(q=0\).  We also put
\[
m:=n+1-c_L-c_R-q.
\]

Let \(\Delta_L\) and \(\Delta_R\) be the Dynkin types of \(\Gamma_L\) and
\(\Gamma_R\), respectively.  We define
\[
\tau_L\in\{\mathrm{fork},\mathrm{edge},\mathrm{mark}\}
\]
by
\[
\tau_L=
\begin{cases}
\mathrm{fork} & \text{in case \textup{(1)(a)} of Proposition~\ref{prop:Instar-both-ends}},\\
\mathrm{edge} & \text{in case \textup{(1)(b)} of Proposition~\ref{prop:Instar-both-ends}},\\
\mathrm{mark} & \text{in cases \textup{(1)(c)} and \textup{(1)(d)} of Proposition~\ref{prop:Instar-both-ends}}.
\end{cases}
\]
We define \(\tau_R\) symmetrically at the right end.
\begin{prop}\label{prop:Instar-both-ends-contraction}
We assume that \(F\) is of Kodaira type \(I_n^*\) with \(n\geq 1\), with
notation as above.  Let \(R\) be an ADER in \(F\), and assume that
\[
C_0,\ C_n\subset \operatorname{Supp}(R).
\]
Let $\nu\colon X\rightarrow Y$
be the contraction of \(R\).  Then every irreducible component of
\(\overline F\) is a smooth rational curve.  More precisely, the following
hold.

\begin{enumerate}[\rm(1)]
\item We suppose that the connected component of \(R\) containing \(C_0\) is
different from the connected component of \(R\) containing \(C_n\).  If
\(R_M=0\), then \(\overline F\) is of type
\[
\mathfrak A_m^{\tau_L,\tau_R}(\Delta_L,\Delta_R)
\]
in the sense of Definition~\ref{dfn:both-end-marked-fibers}.  If \(R_M\neq 0\),
then \(\overline F\) is of type
\[
\mathfrak A_m^{\tau_L,\tau_R}
(\Delta_L,\Delta_R;A_{\lambda_1},\dots,A_{\lambda_r})
\]
in the sense of Definition~\ref{dfn:both-end-marked-fibers}.

\item We suppose that a single connected component of \(R\) contains both \(C_0\)
and \(C_n\).  Then \(R\) is connected, and the resulting fiber is described as
follows.
\begin{enumerate}[\rm(a)]
\item If \(R\) is of type \(A_{n+1}\), namely $R=C_0+C_1+\cdots+C_n$,
then \(\overline F\) is of type
\[
\mathfrak S_4(A_{n+1})
\]
in the sense of Definition~\ref{dfn:both-end-marked-fibers}.

\item If \(R\) is of type \(A_{n+2}\), namely $R=L+C_0+C_1+\cdots+C_n$
for some $L\in\{C_0',C_0'',C_n',C_n''\}$,
then \(\overline F\) is of type
\[
\mathfrak S_3(A_{n+2})
\]
in the sense of Definition~\ref{dfn:both-end-marked-fibers}.

\item If \(R\) is of type \(A_{n+3}\), namely
$R=L_0+C_0+C_1+\cdots+C_n+L_n$
for some $L_0\in\{C_0',C_0''\}$ and $
L_n\in\{C_n',C_n''\}$,
then \(\overline F\) is of type
\[
\mathfrak S_2(A_{n+3})
\]
in the sense of Definition~\ref{dfn:both-end-marked-fibers}.

\item If \(R\) is of type \(D_{n+3}\), namely $R=C_0'+C_0''+C_0+C_1+\cdots+C_n$,
or symmetrically $R=C_0+C_1+\cdots+C_n+C_n'+C_n''$,
then \(\overline F\) is of type
\[
\mathfrak S_2(D_{n+3})
\]
in the sense of Definition~\ref{dfn:both-end-marked-fibers}.

\item If \(R\) is of type \(D_{n+4}\), namely $R=C_0'+C_0''+C_0+C_1+\cdots+C_n+L_n$
for some $L_n\in\{C_n',C_n''\}$,
or symmetrically $R=L_0+C_0+C_1+\cdots+C_n+C_n'+C_n''$
for some $L_0\in\{C_0',C_0''\}$,
then \(\overline F\) is of type
\[
\mathfrak S_1(D_{n+4})
\]
in the sense of Definition~\ref{dfn:both-end-marked-fibers}.

\end{enumerate}
\end{enumerate}
\end{prop}

\begin{proof}
The singular points of \(Y\) lying on \(\overline F\) are exactly the rational
double points obtained by contracting the connected components of \(R\).
By Lemma~\ref{lem:Instar-smooth-components}, every irreducible component of
\(\overline F\) is a smooth rational curve.  Moreover, by Theorem~\ref{2.4},
the fiber \(\overline F\) is obtained from \(F\) by deleting the components
contained in \(R\), while preserving the multiplicities of the remaining
components.

We first consider the case where the connected component of \(R\) containing
\(C_0\) is different from the connected component of \(R\) containing \(C_n\).
Then
\[
R=\Gamma_L+R_M+\Gamma_R.
\]
After contracting \(\Gamma_L\) and \(\Gamma_R\), the left and right ends of the
resulting fiber are described by
\[
(\tau_L,\Delta_L)
\qquad\text{and}\qquad
(\tau_R,\Delta_R),
\]
respectively.  More precisely, the type \(\tau_L\) records whether the
singularity obtained by contracting \(\Gamma_L\) appears at the left end as a
fork, an edge, or a marked point; the type \(\tau_R\) is defined similarly at
the right end.

The part \(R_M\) may be disconnected.  If \(R_M\neq0\), then each connected
component of \(R_M\) is a chain of type \(A_{\lambda_i}\).  Contracting these
components removes
\[
q=\sum_{i=1}^r\lambda_i
\]
components from the middle part of the central chain.  Hence the number of
surviving central components between the left and right end singularities is
\[
m=n+1-c_L-c_R-q.
\]
If \(R_M\neq0\), the rational double points obtained by contracting the
connected components of \(R_M\) are of types
\[
A_{\lambda_1},\dots,A_{\lambda_r},
\]
and they lie at some of the intersection points of adjacent components in the
surviving central chain.  Thus, by Definition~\ref{dfn:both-end-marked-fibers},
\(\overline F\) is of type
\[
\mathfrak A_m^{\tau_L,\tau_R}
(\Delta_L,\Delta_R;A_{\lambda_1},\dots,A_{\lambda_r})
\]
if \(R_M\neq0\), and of type
\[
\mathfrak A_m^{\tau_L,\tau_R}(\Delta_L,\Delta_R)
\]
if \(R_M=0\).  This proves \textup{(1)}.

It remains to consider the case where a single connected component of \(R\)
contains both \(C_0\) and \(C_n\).  Then \(R\) contains the whole central chain
\[
C_0+C_1+\cdots+C_n,
\]
and \(R\) is connected.  Therefore the contraction of \(R\) produces exactly one
rational double point, whose type is the Dynkin type of \(R\).  The only
surviving components are the leaves not contained in \(R\), and all of them have
multiplicity \(1\).

If \(R\) is of type \(A_{n+1}\), then no leaf is contained in \(R\).  Hence four
leaves survive, and all of them pass through the rational double point of type
\(A_{n+1}\).  Therefore $\overline F$
is of type
\[
\mathfrak S_4(A_{n+1}).
\]

If \(R\) is of type \(A_{n+2}\), then exactly one leaf is contained in \(R\).
Hence three leaves survive, and all of them pass through the rational double
point of type \(A_{n+2}\).  Therefore \(\overline F\) is of type
\[
\mathfrak S_3(A_{n+2}).
\]

If \(R\) is of type \(A_{n+3}\), then exactly two leaves, one at each end, are
contained in \(R\).  Hence two leaves survive, and both pass through the
rational double point of type \(A_{n+3}\).  Therefore \(\overline F\) is of type
\[
\mathfrak S_2(A_{n+3}).
\]

If \(R\) is of type \(D_{n+3}\), then two leaves at one end are contained in
\(R\), and no leaf at the other end is contained in \(R\).  Hence two leaves
survive, and both pass through the rational double point of type \(D_{n+3}\).
Therefore \(\overline F\) is of type
\[
\mathfrak S_2(D_{n+3}).
\]

If \(R\) is of type \(D_{n+4}\), then two leaves at one end and one leaf at the
other end are contained in \(R\).  Hence exactly one leaf survives, and its
image contains the rational double point of type \(D_{n+4}\).  Therefore
\(\overline F\) is of type
\[
\mathfrak S_1(D_{n+4}).
\]
\end{proof}

\section{Fibers of type \(II^*\)}
Fibers of type \(II^*\) correspond to the affine Dynkin diagram
\(\widetilde E_8\) in Kodaira's classification \cite{k63}. These play a particularly important role in the theory of singular $K3$ surfaces \cite{si77}. Their
occurrence on elliptic \(K3\) surfaces is also included in the ADE-type and
extremal classification results of Shimada and Shimada--Zhang
\cite{s00,sz01}.

Throughout this section, we use the following notation. Let \(F\) be a singular
fiber of Kodaira type \(II^*\). We write
\[
F=
6C+3S_1+4M_1+2M_2+5L_1+4L_2+3L_3+2L_4+L_5,
\]
where \(C\) is the unique component meeting three other components. The dual
graph of \(F\) is the affine Dynkin diagram \(\widetilde E_8\), and it is given
as follows:
\[
\begin{tikzpicture}[baseline=(current bounding box.center)]
\node (L5) at (-5,0) {$L_5$};
\node (L4) at (-4,0) {$L_4$};
\node (L3) at (-3,0) {$L_3$};
\node (L2) at (-2,0) {$L_2$};
\node (L1) at (-1,0) {$L_1$};
\node (C)  at (0,0) {$C$};
\node (M1) at (1,0) {$M_1$};
\node (M2) at (2,0) {$M_2$};

\node (S1) at (0,1) {$S_1$};

\draw[thick] (L5.east) -- (L4.west);
\draw[thick] (L4.east) -- (L3.west);
\draw[thick] (L3.east) -- (L2.west);
\draw[thick] (L2.east) -- (L1.west);
\draw[thick] (L1.east) -- (C.west);

\draw[thick] (C.east) -- (M1.west);
\draw[thick] (M1.east) -- (M2.west);

\draw[thick] (C.north) -- (S1.south);
\end{tikzpicture}
\]

\begin{prop}\label{prop:ADE-IIstar-no-C}
We assume that \(F\) is of Kodaira type \(II^*\), with notation as above. Let
\(R\) be an ADER in \(F\). 
We assume that
\[
C\not\subset \operatorname{Supp}(R).
\]
Then \(R\) decomposes uniquely as a disjoint sum
\[
R=R_S+R_M+R_L,
\]
where \(R_S\in\{0,S_1\}\),
and \(R_M\) and \(R_L\) are either zero or disjoint unions of mutually disjoint
consecutive subchains of \(M_1-M_2\) and
\(L_1-L_2-L_3-L_4-L_5\), respectively.
\begin{enumerate}[\rm(1)]
\item If \(R_M\neq0\), then \(R_M\) is of type
\[
A_{\mu_1}\oplus\cdots\oplus A_{\mu_a}
\]
for some \(\mu_1,\dots,\mu_a\geq1\) satisfying
\[
\sum_{i=1}^a\mu_i+a-1\leq2.
\]

\item If \(R_L\neq0\), then \(R_L\) is of type
\[
A_{\lambda_1}\oplus\cdots\oplus A_{\lambda_b}
\]
for some \(\lambda_1,\dots,\lambda_b\geq1\) satisfying
\[
\sum_{j=1}^b\lambda_j+b-1\leq5.
\]
\end{enumerate}

Conversely, every reduced subdivisor satisfying the above conditions is an
ADER.
\end{prop}
\begin{proof}
Since \(C\not\subset \operatorname{Supp}(R)\), the support of \(R\) is contained
in the union of the three branches
\[
S_1,\qquad M_1-M_2,\qquad L_1-L_2-L_3-L_4-L_5.
\]
These three branches meet each other only through the component \(C\). They are
mutually disconnected inside \(\operatorname{Supp}(R)\). Therefore \(R\)
decomposes uniquely as a disjoint sum
\[
R=R_S+R_M+R_L,
\]
where \(R_S\), \(R_M\), and \(R_L\) are supported on the \(S\)-, \(M\)-, and
\(L\)-branches, respectively.
Since the \(S\)-branch consists of the single component \(S_1\),
\(R_S\in\{0,S_1\}\). Since \(M_1-M_2\) and
\(L_1-L_2-L_3-L_4-L_5\) are chains, every connected component of \(R_M\) and
\(R_L\) is a consecutive subchain. Thus, if \(R_M\neq0\), then \(R_M\) is of type
\[
A_{\mu_1}\oplus\cdots\oplus A_{\mu_a},
\]
and if \(R_L\neq0\), then \(R_L\) is of type
\[
A_{\lambda_1}\oplus\cdots\oplus A_{\lambda_b}.
\]
Distinct connected components must be separated by at least one unused vertex.
Since the \(M\)-branch has \(2\) vertices and the \(L\)-branch has \(5\)
vertices, we obtain
\[
\sum_{i=1}^a\mu_i+(a-1)\leq2
\]
for \(R_M\), and
\[
\sum_{j=1}^b\lambda_j+(b-1)\leq5
\]
for \(R_L\).

Conversely, any reduced subdivisor satisfying the above conditions is a
disjoint union of chains, and hence its dual graph is a disjoint union of
Dynkin diagrams of type \(A\). Therefore it is an ADER.
\end{proof}

In the case where \(C\subset \operatorname{Supp}(R)\), we use the following
notation. Let \(\Gamma\) be the connected component of \(R\) containing \(C\).
Since \(\Gamma\) is connected, its intersection with each branch is an initial
segment. Thus there exist integers
\[
\alpha\in\{0,1\},
\qquad
\beta\in\{0,1,2\},
\qquad
\gamma\in\{0,1,2,3,4,5\}
\]
such that
\[
\Gamma=
C+S_1+\cdots+S_\alpha
+M_1+\cdots+M_\beta
+L_1+\cdots+L_\gamma,
\]
where the terms \(S_1+\cdots+S_\alpha\), \(M_1+\cdots+M_\beta\), and
\(L_1+\cdots+L_\gamma\) are omitted if the corresponding integer is \(0\).
We denote by \(R_M^{\mathrm{tail}}\) and \(R_L^{\mathrm{tail}}\) the remaining
parts of \(R\) supported on the \(M\)- and \(L\)-branches, respectively. Thus,
in this case, we write
\[
R=\Gamma+R_M^{\mathrm{tail}}+R_L^{\mathrm{tail}}.
\]
\begin{prop}\label{prop:ADE-IIstar-C}
We assume that \(F\) is of Kodaira type \(II^*\), with notation as above. Let
\(R\) be an ADER in \(F\). We assume that
\[
C\subset \operatorname{Supp}(R).
\]
We write $R=\Gamma+R_M^{\mathrm{tail}}+R_L^{\mathrm{tail}}$
as above, where $
\Gamma=
C+S_1+\cdots+S_\alpha
+M_1+\cdots+M_\beta
+L_1+\cdots+L_\gamma$.
Then the triple \((\alpha,\beta,\gamma)\) satisfies one of the following
conditions.

\begin{enumerate}[\rm(1)]
\item At most two of \(\alpha,\beta,\gamma\) are nonzero. In this case
\(\Gamma\) is of type $A_{1+\alpha+\beta+\gamma}$.
\item The three integers \(\alpha,\beta,\gamma\) are all nonzero. Then
\(\alpha=1\), and
\[
(\beta,\gamma)\neq(2,5).
\]
The type of \(\Gamma\) is as follows:
\begin{enumerate}[\rm(a)]
\item If \(\beta=1\), then \(\Gamma\) is of type $D_{\gamma+3}$.

\item If \((\beta,\gamma)=(2,1)\), then \(\Gamma\) is of type $D_5$.

\item If \((\beta,\gamma)=(2,2)\), then \(\Gamma\) is of type $E_6$.

\item If \((\beta,\gamma)=(2,3)\), then \(\Gamma\) is of type $E_7$.

\item If \((\beta,\gamma)=(2,4)\), then \(\Gamma\) is of type $E_8$.
\end{enumerate}
\end{enumerate}

Moreover, \(R_M^{\mathrm{tail}}=0\) if \(\beta\geq1\).  If
\(\beta=0\), then \(R_M^{\mathrm{tail}}\) is either zero or equal to \(M_2\);
in the latter case it is of type \(A_1\).
The divisor \(R_L^{\mathrm{tail}}\) is either zero or a disjoint union of
mutually disjoint consecutive subchains of
\[
L_{\gamma+2}-L_{\gamma+3}-\cdots-L_5
\]
if \(\gamma\leq3\), and it is zero if \(\gamma\geq4\).  If
\(R_L^{\mathrm{tail}}\neq0\), then it is of type
\[
A_{\lambda_1}\oplus\cdots\oplus A_{\lambda_b}
\]
for some \(\lambda_1,\dots,\lambda_b\geq1\) satisfying
\[
\sum_{j=1}^b\lambda_j+b-1\leq4-\gamma.
\]

Conversely, every reduced subdivisor satisfying the above conditions is an
ADER.
\end{prop}
\begin{proof}
If at most two of \(\alpha,\beta,\gamma\) are nonzero, then the dual graph of
\(\Gamma\) is a chain. Therefore \(\Gamma\) is of type $A_{1+\alpha+\beta+\gamma}$.

Next, we assume that \(\alpha,\beta,\gamma\) are all nonzero. Then
\(\alpha=1\), and \(C\) is the unique branching vertex of \(\Gamma\). The three
branch lengths of \(\Gamma\) starting from \(C\) are
\[
1,\qquad \beta,\qquad \gamma.
\]

If \(\beta=1\), then the branch lengths are
\[
1,\quad 1,\quad \gamma,
\]
and hence \(\Gamma\) is of type $D_{\gamma+3}$
where \(1\leq\gamma\leq5\).

If \(\beta=2\) and \(\gamma=1\), then the branch lengths are
\[
1,\quad 1,\quad 2,
\]
and hence \(\Gamma\) is of type
$D_5$.

If \(\beta=2\) and \(\gamma=2\), then the branch lengths are
\[
1,\quad 2,\quad 2,
\]
and hence \(\Gamma\) is of type
$E_6$.

If \(\beta=2\) and \(\gamma=3\), then the branch lengths are
\[
1,\quad 2,\quad 3,
\]
and hence \(\Gamma\) is of type
$E_7$.

If \(\beta=2\) and \(\gamma=4\), then the branch lengths are
\[
1,\quad 2,\quad 4,
\]
and hence \(\Gamma\) is of type
$E_8$.

The remaining case is $(\beta,\gamma)=(2,5)$.
In this case \(\Gamma\) is the whole affine diagram \(\widetilde E_8\), and
therefore it is not of ADE type. Hence this case cannot occur for an ADER.
This proves the stated possibilities for \(\Gamma\).

Since \(\Gamma\) is a connected component of \(R\), the vertex immediately
adjacent to \(\Gamma\) on any branch cannot belong to \(R\). On the \(S\)-branch
there is no remaining component. 
On the \(M\)-branch, if \(\beta=0\), then \(M_1\) is adjacent to
\(\Gamma\), and hence cannot belong to \(R\).  Thus the only possible
remaining component on the \(M\)-branch is \(M_2\).  If \(\beta\geq1\), then
there is no possible remaining component on the \(M\)-branch.  Hence
\(R_M^{\mathrm{tail}}=0\) if \(\beta\geq1\), while if \(\beta=0\), then
\(R_M^{\mathrm{tail}}\) is either zero or \(M_2\).

On the \(L\)-branch, the component \(L_{\gamma+1}\), if it exists, is adjacent
to \(\Gamma\), and hence cannot belong to \(R\).  Therefore the remaining
components on the \(L\)-branch are supported on
\[
L_{\gamma+2}-L_{\gamma+3}-\cdots-L_5
\]
if \(\gamma\leq3\), and there are no remaining components if
\(\gamma\geq4\).  Thus \(R_L^{\mathrm{tail}}\) is either zero or a disjoint
union of mutually disjoint consecutive subchains of this chain.  The stated
inequality for \(R_L^{\mathrm{tail}}\) follows because this chain has
\(4-\gamma\) vertices.

Conversely, suppose that a reduced subdivisor satisfies the above conditions.
Then its dual graph is a disjoint union of the graph of \(\Gamma\), which is of
type \(A\), \(D\), \(E_6\), \(E_7\), or \(E_8\), and possibly some additional
chains of type \(A\). Therefore its dual graph is a disjoint union of Dynkin
diagrams of ADE type. Hence the subdivisor is an ADER.
\end{proof}

Before describing the contraction, we introduce notation for the fibers obtained
from a fiber of type \(II^*\) by contracting an ADER.

\begin{dfn}\label{dfn:IIstar-contracted-fibers}
Let
\[
F=
6C+3S_1+4M_1+2M_2+5L_1+4L_2+3L_3+2L_4+L_5
=\sum_E m_EE
\]
be a fiber of type \(II^*\), and let \(R\) be an ADER in \(F\).
Let $\nu\colon X\rightarrow Y$ be the contraction of \(R\). 
For every irreducible component \(E\not\subset R\), we write
\(\overline E:=\nu(E)\). Then the resulting fiber is $\overline F=\sum_{E\not\subset R}m_E\overline E$.

\begin{enumerate}[\rm(1)]
\item If \(C\not\subset \operatorname{Supp}(R)\), and
\(R=R_S+R_M+R_L\)
is as in Proposition~\ref{prop:ADE-IIstar-no-C}, then we denote the resulting
fiber by
\[
\mathfrak{II}^{*,\mathrm{nb}}(R_S,R_M,R_L).
\]

\item If \(C\subset \operatorname{Supp}(R)\), and
\(R=\Gamma+R_M^{\mathrm{tail}}+R_L^{\mathrm{tail}}\)
is as in Proposition~\ref{prop:ADE-IIstar-C}, then we denote the resulting
fiber by
\[
\mathfrak{II}^{*,\mathrm{br}}
(\Gamma;R_M^{\mathrm{tail}},R_L^{\mathrm{tail}}).
\]
\end{enumerate}
\end{dfn}

In both cases, the notation refers to the fiber obtained from \(F\) by deleting
the components contained in \(R\), while preserving the multiplicities of all
remaining components. The rational double points are determined as follows:
each connected component of \(R\) is contracted to a rational double point of
the corresponding Dynkin type, and the surviving components adjacent to it pass
through the resulting rational double point.

The explicit lists of the fibers of types
\[
\mathfrak{II}^{*,\mathrm{nb}}(R_S,R_M,R_L)
\qquad\text{and}\qquad
\mathfrak{II}^{*,\mathrm{br}}
(\Gamma;R_M^{\mathrm{tail}},R_L^{\mathrm{tail}})
\]
are given in Appendices~\ref{app:IIstar-no-branch-list} and
\ref{app:IIstar-with-branch-list}, respectively.

\begin{prop}\label{prop:IIstar-contraction}
We assume that \(F\) is of Kodaira type \(II^*\), with notation as above.
Let \(R\) be an ADER in \(F\), and let \(\nu\colon X\rightarrow Y\)
be the contraction of \(R\). Then every irreducible component of
\(\overline F\) is a smooth rational curve. 
Moreover, with the notation of
Propositions~\ref{prop:ADE-IIstar-no-C} and \ref{prop:ADE-IIstar-C}, \(\overline F\) is of
type
\[
\mathfrak{II}^{*,\mathrm{nb}}(R_S,R_M,R_L)
\]
if \(C\not\subset \operatorname{Supp}(R)\), and of type
\[
\mathfrak{II}^{*,\mathrm{br}}
(\Gamma;R_M^{\mathrm{tail}},R_L^{\mathrm{tail}})
\]
if \(C\subset \operatorname{Supp}(R)\), in the sense of
Definition~\ref{dfn:IIstar-contracted-fibers}.
\end{prop}
\begin{proof}
Let \(E\) be an irreducible component of \(F\) not contained in
\(R\). We put $C_E:=\nu(E)$.
Let $y\in C_E\cap\Sing(Y)$.
Then \(\nu^{-1}(y)\) is a connected component of \(R\), say \(\Delta\).  Since
the dual graph of a \(II^*\)-fiber is a tree, and since \(\Delta\) is a subtree,
the component \(E\) meets \(\Delta\) in at most one point.  Since \(y\in C_E\),
this intersection is nonempty.  Hence
\[
E\cap \nu^{-1}(y)=\{x\}.
\]
Moreover, the components of a \(II^*\)-fiber meet transversely, and therefore
\[
(\nu^{-1}(y)\cdot E)_x=1.
\]
The restriction $\nu|_E\colon E\rightarrow C_E$
is the normalization of \(C_E\).  Therefore, by
Proposition~\ref{prop:smoothness-criterion}, the curve \(C_E\) is smooth at
\(y\).  Since \(\nu\) is an isomorphism away from the exceptional locus,
\(C_E\) is smooth away from \(\Sing(Y)\).  Hence every irreducible component of
\(\overline F\) is a smooth rational curve.
\end{proof}

\section{Fibers of type $III^*$}
Fibers of type \(III^*\) correspond to the affine Dynkin diagram
\(\widetilde E_7\).  As in the case of \(II^*\)-fibers, their occurrence
on elliptic \(K3\) surfaces is covered by the ADE-type and extremal
classification results \cite{s00,sz01}.

Throughout this section, we use the following notation. Let $F$ be a singular
fiber of Kodaira type \(III^*\). We write
\[
F=
4C+2S_1+3M_1+2M_2+M_3+3L_1+2L_2+L_3
\]
where \(C\) is the unique component meeting three other components. The dual
graph of \(F\) is the affine Dynkin diagram \(\widetilde E_7\), and it is given
as follows:
\[
\begin{tikzpicture}[baseline=(current bounding box.center)]
\node (L3) at (-3,0) {$L_3$};
\node (L2) at (-2,0) {$L_2$};
\node (L1) at (-1,0) {$L_1$};
\node (C) at (0,0) {$C$};
\node (M1) at (1,0) {$M_1$};
\node (M2) at (2,0) {$M_2$};
\node (M3) at (3,0) {$M_3$};

\node (S1) at (0,1) {$S_1$};

\draw[thick] (L3.east) -- (L2.west);
\draw[thick] (L2.east) -- (L1.west);
\draw[thick] (L1.east) -- (C.west);

\draw[thick] (C.east) -- (M1.west);
\draw[thick] (M1.east) -- (M2.west);
\draw[thick] (M2.east) -- (M3.west);

\draw[thick] (C.north) -- (S1.south);
\end{tikzpicture}
\]
\begin{prop}\label{prop:ADE-IIIstar-no-C}
We assume that \(F\) is of Kodaira type \(III^*\), with notation as above. Let \(R\) be an ADER in \(F\). 
We assume that
\[
C\not\subset \operatorname{Supp}(R).
\]
Then \(R\) decomposes uniquely as a disjoint sum
\[
R=R_S+R_M+R_L,
\]
where $R_S\in\{0,S_1\}$,
and \(R_M\) and \(R_L\) are either zero or disjoint unions of mutually disjoint
consecutive subchains of $M_1-M_2-M_3$ and $L_1-L_2-L_3$,
respectively. 
\begin{enumerate}[\rm(1)]
\item If \(R_M\neq0\), then \(R_M\) is of type
\[
A_{\mu_1}\oplus\cdots\oplus A_{\mu_a}
\]
for some \(\mu_1,\dots,\mu_a\geq1\) satisfying
\[
\sum_{i=1}^a\mu_i+a-1\leq3.
\]
\item If \(R_L\neq0\), then \(R_L\) is of type
\[
A_{\lambda_1}\oplus\cdots\oplus A_{\lambda_b}
\]
for some \(\lambda_1,\dots,\lambda_b\geq1\) satisfying
\[
\sum_{j=1}^b\lambda_j+b-1\leq3.
\]
\end{enumerate}

Conversely, every reduced subdivisor satisfying the above conditions is an
ADER.
\end{prop}
\begin{proof}
Since \(C\not\subset \operatorname{Supp}(R)\), the support of \(R\) is contained
in the union of the three branches
\[
S_1,\qquad M_1-M_2-M_3,\qquad L_1-L_2-L_3.
\]
These three branches meet each other only through the component \(C\). They are mutually disconnected inside
\(\operatorname{Supp}(R)\). Therefore \(R\) decomposes uniquely as a disjoint sum $R=R_S+R_M+R_L$
where \(R_S\), \(R_M\), and \(R_L\) are supported on the \(S\)-, \(M\)-, and
\(L\)-branches, respectively.

Since the \(S\)-branch consists of the single component \(S_1\), $R_S\in\{0,S_1\}$.
Since $M_1-M_2-M_3$ and $L_1-L_2-L_3$
are chains, every connected component of \(R_M\) and \(R_L\) is a consecutive
subchain. Thus, if \(R_M\neq0\), then \(R_M\) is of type
\[
A_{\mu_1}\oplus\cdots\oplus A_{\mu_a},
\]
and if \(R_L\neq0\), then \(R_L\) is of type
\[
A_{\lambda_1}\oplus\cdots\oplus A_{\lambda_b}.
\]
Distinct connected components must be separated by at least one unused vertex.
Since both the \(M\)- and \(L\)-branches have \(3\) vertices, we obtain
\[
\sum_{i=1}^a\mu_i+(a-1)\leq3
\]
for \(R_M\), and
\[
\sum_{j=1}^b\lambda_j+(b-1)\leq3
\]
for \(R_L\).

Conversely, any reduced subdivisor satisfying the above conditions is a
disjoint union of chains, and hence its dual graph is a disjoint union of
Dynkin diagrams of type \(A\). Therefore it is an ADER.
\end{proof}
In the case where \(C\subset \operatorname{Supp}(R)\), we use the following
notation. Let \(\Gamma\) be the connected component of \(R\) containing \(C\).
Since \(\Gamma\) is connected, its intersection with each branch is an initial
segment. Thus there exist integers
\[
\alpha\in\{0,1\},
\qquad
\beta\in\{0,1,2,3\},
\qquad
\gamma\in\{0,1,2,3\}
\]
such that
\[
\Gamma=
C+S_1+\cdots+S_\alpha
+M_1+\cdots+M_\beta
+L_1+\cdots+L_\gamma,
\]
where the terms \(S_1+\cdots+S_\alpha\), \(M_1+\cdots+M_\beta\), and
\(L_1+\cdots+L_\gamma\) are omitted if the corresponding integer is \(0\).

We denote by \(R_M^{\mathrm{tail}}\) and \(R_L^{\mathrm{tail}}\) the remaining
parts of \(R\) supported on the \(M\)- and \(L\)-branches, respectively. Thus,
in this case, we write
\[
R=\Gamma+R_M^{\mathrm{tail}}+R_L^{\mathrm{tail}}.
\]
\begin{prop}\label{prop:ADE-IIIstar-C}
We assume that \(F\) is of Kodaira type \(III^*\), with notation as above. Let
\(R\) be an ADER in \(F\). We assume that
\[
C\subset \operatorname{Supp}(R).
\]
We write $R=\Gamma+R_M^{\mathrm{tail}}+R_L^{\mathrm{tail}}$
as above, where $\Gamma=C+S_1+\cdots+S_\alpha+M_1+\cdots+M_\beta+L_1+\cdots+L_\gamma$.
Then the triple \((\alpha,\beta,\gamma)\) satisfies one of the following
conditions.

\begin{enumerate}[\rm(1)]
\item At most two of \(\alpha,\beta,\gamma\) are nonzero. In this case
\(\Gamma\) is of type
\[
A_{1+\alpha+\beta+\gamma}.
\]

\item The three integers \(\alpha,\beta,\gamma\) are all nonzero. By the
symmetry interchanging the \(M\)- and \(L\)-branches, we may assume, when
listing the possibilities, that
\[
\beta\leq\gamma.
\]
The remaining cases are obtained by interchanging the \(M\)- and \(L\)-branches.
Under this convention, we have
\[
(\beta,\gamma)\neq(3,3),
\]
and the type of \(\Gamma\) is as follows:
\begin{enumerate}[\rm(a)]
\item If \(\beta=1\), then \(\Gamma\) is of type $D_{\gamma+3}$.

\item If \((\beta,\gamma)=(2,2)\), then \(\Gamma\) is of type $E_6$.

\item If \((\beta,\gamma)=(2,3)\), then \(\Gamma\) is of type $E_7$.
\end{enumerate}
\end{enumerate}

Moreover, \(R_M^{\mathrm{tail}}\) is either zero or a disjoint union of mutually
disjoint consecutive subchains of
\[
M_{\beta+2}-M_{\beta+3}-\cdots-M_3.
\]
Here this chain means \(M_2-M_3\) if \(\beta=0\), \(M_3\) if \(\beta=1\), and
the empty chain if \(\beta\geq2\). If
\(R_M^{\mathrm{tail}}\neq0\), then it is of type
\[
A_{\mu_1}\oplus\cdots\oplus A_{\mu_a}
\]
for some \(\mu_1,\dots,\mu_a\geq1\) satisfying
\[
\sum_{i=1}^a\mu_i+a-1\leq2-\beta.
\]
The same assertion holds for \(R_L^{\mathrm{tail}}\), with
\(M,\beta,\mu_i,a\) replaced by \(L,\gamma,\lambda_j,b\), respectively.

Conversely, every reduced subdivisor satisfying the above conditions, or one
obtained from such a subdivisor by interchanging the \(M\)- and \(L\)-branches,
is an ADER.
\end{prop}

\begin{proof}
If at most two of \(\alpha,\beta,\gamma\) are nonzero, then the dual graph of
\(\Gamma\) is a chain. Therefore \(\Gamma\) is of type
\[
A_{1+\alpha+\beta+\gamma}.
\]

Next, we assume that \(\alpha,\beta,\gamma\) are all nonzero. Then \(\alpha=1\),
and \(C\) is the unique branching vertex of \(\Gamma\). The three branch
lengths of \(\Gamma\) starting from \(C\) are
\[
1,\qquad \beta,\qquad \gamma.
\]
By the symmetry interchanging the \(M\)- and \(L\)-branches, we may assume,
when listing the possibilities, that $\beta\leq\gamma$.
The remaining cases are obtained by interchanging the \(M\)- and \(L\)-branches.

Under this convention, if \(\beta=1\), then the branch lengths are
\[1,\quad 1,\quad \gamma,\]
and hence \(\Gamma\) is of type $D_{\gamma+3}$
where \(\gamma=1,2,3\).

If \(\beta=2\) and \(\gamma=2\), then the branch lengths are
\[1,\quad 2,\quad 2,\]
and hence \(\Gamma\) is of type $E_6$.

If \(\beta=2\) and \(\gamma=3\), then the branch lengths are 
\[1,\quad 2,\quad 3,\]
and hence \(\Gamma\) is of type $E_7$.

The remaining case under the convention \(\beta\leq\gamma\) is $(\beta,\gamma)=(3,3)$.
In this case \(\Gamma\) is the whole affine diagram \(\widetilde E_7\), and
therefore it is not of ADE type. Hence this case cannot occur for an ADER.
This proves the stated possibilities for \(\Gamma\).

Since \(\Gamma\) is a connected component of \(R\), the vertex immediately
adjacent to \(\Gamma\) on any branch cannot belong to \(R\). On the \(S\)-branch
there is no remaining component. On the \(M\)-branch, the component
\(M_{\beta+1}\), if it exists, is adjacent to \(\Gamma\), and hence cannot
belong to \(R\). 
Therefore the remaining components on the \(M\)-branch are
supported on
\[
M_{\beta+2}-M_{\beta+3}-\cdots-M_3.
\]
Here this chain means \(M_2-M_3\) if \(\beta=0\), \(M_3\) if \(\beta=1\), and
the empty chain if \(\beta\geq2\).
Thus \(R_M^{\mathrm{tail}}\) is either zero or a disjoint union of mutually disjoint
consecutive subchains of this chain.

If \(R_M^{\mathrm{tail}}\neq0\), then each connected component of
\(R_M^{\mathrm{tail}}\) is a chain, and hence is of type \(A\). Thus
\(R_M^{\mathrm{tail}}\) is of type
\[
A_{\mu_1}\oplus\cdots\oplus A_{\mu_a}
\]
for some \(\mu_1,\dots,\mu_a\geq1\). Since the chain
\[
M_{\beta+2}-M_{\beta+3}-\cdots-M_3
\]
has \(2-\beta\) vertices, and distinct connected components must be separated
by at least one unused vertex, we obtain
\[
\sum_{i=1}^a\mu_i+(a-1)\leq2-\beta.
\]
The same argument applies to the \(L\)-branch, with
\(M,\beta,\mu_i,a\) replaced by \(L,\gamma,\lambda_j,b\), respectively. Hence
\(R_L^{\mathrm{tail}}\) has the asserted form and satisfies the stated
inequality.

Conversely, suppose that a reduced subdivisor satisfies the above conditions,
or is obtained from such a subdivisor by interchanging the \(M\)- and
\(L\)-branches. Then its dual graph is a disjoint union of the graph of
\(\Gamma\), which is of type \(A\), \(D\), \(E_6\), or \(E_7\), and possibly
some additional chains of type \(A\). Therefore its dual graph is a disjoint
union of Dynkin diagrams of ADE type. Hence the subdivisor is an ADER.
\end{proof}

Before describing the contraction, we introduce notation for the fibers obtained
from a fiber of type \(III^*\) by contracting an ADER.

\begin{dfn}\label{dfn:IIIstar-contracted-fibers}
Let
\[
F=
4C+2S_1+3M_1+2M_2+M_3+3L_1+2L_2+L_3
=\sum_E m_EE
\]
be a fiber of type \(III^*\), and let \(R\) be an ADER in \(F\).
Let \(\nu\colon X\rightarrow Y\) be the contraction of \(R\).
For every irreducible component \(E\not\subset R\), we write
\(\overline E:=\nu(E)\). Then the resulting fiber is
\[
\overline F=\sum_{E\not\subset R}m_E\overline E.
\]

\begin{enumerate}[\rm(1)]
\item If \(C\not\subset \operatorname{Supp}(R)\), and
$R=R_S+R_M+R_L$
is as in Proposition~\ref{prop:ADE-IIIstar-no-C}, then we denote the resulting
fiber by
\[
\mathfrak{III}^{*,\mathrm{nb}}(R_S,R_M,R_L).
\]

\item If \(C\subset \operatorname{Supp}(R)\), and
$R=\Gamma+R_M^{\mathrm{tail}}+R_L^{\mathrm{tail}}$
is as in Proposition~\ref{prop:ADE-IIIstar-C}, then we denote the resulting
fiber by
\[
\mathfrak{III}^{*,\mathrm{br}}
(\Gamma;R_M^{\mathrm{tail}},R_L^{\mathrm{tail}}).
\]
\end{enumerate}
\end{dfn}

In both cases, the notation refers to the fiber obtained from \(F\) by deleting
the components contained in \(R\), while preserving the multiplicities of all
remaining components. The rational double points are determined as follows:
each connected component of \(R\) is contracted to a rational double point of
the corresponding Dynkin type, and the surviving components adjacent to it pass
through the resulting rational double point.

The explicit lists of the fibers of types
\[
\mathfrak{III}^{*,\mathrm{nb}}(R_S,R_M,R_L)
\qquad\text{and}\qquad
\mathfrak{III}^{*,\mathrm{br}}
(\Gamma;R_M^{\mathrm{tail}},R_L^{\mathrm{tail}})
\]
are given in Appendices~\ref{app:IIIstar-no-branch-list} and
\ref{app:IIIstar-with-branch-list}, respectively.

\begin{prop}\label{prop:IIIstar-contraction}
We assume that \(F\) is of Kodaira type \(III^*\), with notation as above.
Let \(R\) be an ADER in \(F\), and let \(\nu\colon X\rightarrow Y\)
be the contraction of \(R\). Then every irreducible component of
\(\overline F\) is a smooth rational curve. 
Moreover, with the notation of
Propositions~\ref{prop:ADE-IIIstar-no-C} and \ref{prop:ADE-IIIstar-C},
\(\overline F\) is of type
\[
\mathfrak{III}^{*,\mathrm{nb}}(R_S,R_M,R_L)
\]
if \(C\not\subset \operatorname{Supp}(R)\), and of type
\[
\mathfrak{III}^{*,\mathrm{br}}
(\Gamma;R_M^{\mathrm{tail}},R_L^{\mathrm{tail}})
\]
if \(C\subset \operatorname{Supp}(R)\), in the sense of
Definition~\ref{dfn:IIIstar-contracted-fibers}.
\end{prop}
\begin{proof}
The proof is the same as that of Proposition~\ref{prop:IIstar-contraction}.
Since the dual graph of a \(III^*\)-fiber is a tree, every surviving component
meets each contracted connected component in at most one point, transversely.
Hence Proposition~\ref{prop:smoothness-criterion} shows that the image of every
surviving component is smooth. The explicit list follows from
Theorem~\ref{2.4} and Definition~\ref{dfn:IIIstar-contracted-fibers}.
\end{proof}

\section{Fibers of type \(IV^*\)}
Fibers of type \(IV^*\) correspond to the affine Dynkin diagram
\(\widetilde E_6\).  This is the remaining exceptional additive case, and
its occurrence on elliptic \(K3\) surfaces is covered by the classification
results in \cite{s00,sz01}.

Throughout this section, we use the following notation. Let $F$ be a singular fiber of Kodaira type $IV^*$. We write
\[
F = 3C + 2(B_1 + B_2 + B_3) + B'_1 + B'_2 + B'_3.
\]
The dual graph of \(F\) is given as follows:
\[
\begin{tikzpicture}[baseline=(current bounding box.center)]
  \node (B1p) at (-2,0) {$B'_1$};
  \node (B1)  at (-1,0) {$B_1$};
  \node (C)   at (0,0)  {$C$};
  \node (B2)  at (1,0)  {$B_2$};
  \node (B2p) at (2,0)  {$B'_2$};

  \node (B3)  at (0,1) {$B_3$};
  \node (B3p) at (0,2) {$B'_3$};

  \draw[thick] (B1p.east) -- (B1.west);
  \draw[thick] (B1.east)  -- (C.west);
  \draw[thick] (C.east)   -- (B2.west);
  \draw[thick] (B2.east)  -- (B2p.west);

  \draw[thick] (C.north)  -- (B3.south);
  \draw[thick] (B3.north) -- (B3p.south);
\end{tikzpicture}
\]
\begin{prop}\label{prop:ADE-IVstar-no-branch}
We assume that \(F\) is of Kodaira type \(IV^*\), with notation as above.  Let
\(R\) be an ADER in \(F\). We assume that
\[
C\not\subset \operatorname{Supp}(R).
\]
Then \(R\) decomposes uniquely as a disjoint sum
\[
R=R_1+R_2+R_3
\]
where $R_i\in\{0,\ B_i,\ B_i',\ B_i+B_i'\}$
for each \(i=1,2,3\).
Conversely, every nonzero reduced subdivisor satisfying these conditions is an
ADER.
\end{prop}
\begin{proof}
Since $C\not\subset \operatorname{Supp}(R)$,
the support of \(R\) is contained in the union of the three branches
\[
B_1-B_1',
\qquad
B_2-B_2',
\qquad
B_3-B_3'.
\]
These branches meet each other only through the component \(C\).  Hence, they are mutually disconnected inside \(\operatorname{Supp}(R)\).
Therefore \(R\) decomposes uniquely as a disjoint sum $R=R_1+R_2+R_3$
where \(R_i\) is supported on the branch \(B_i-B_i'\).
Since each branch is a chain with two vertices, the possible reduced subdivisors
supported on the \(i\)-th branch are
\[
0,\qquad B_i,\qquad B_i',\qquad B_i+B_i'.
\]

If \(R_i=B_i\) or \(R_i=B_i'\), then \(R_i\) is of type \(A_1\).  If
\(R_i=B_i+B_i'\), then \(R_i\) is of type \(A_2\).  
Thus every nonzero reduced
subdivisor satisfying the stated conditions is a disjoint union of Dynkin
diagrams of type \(A_1\) or \(A_2\), and hence is an ADER.
\end{proof}
In the case where $C\subset \operatorname{Supp}(R)$,
we use the following notation.  Let \(\Gamma\) be the connected component of
\(R\) containing \(C\).  Since \(\Gamma\) is connected, its intersection with
each branch is an initial segment.  Thus there exist integers $m_1,m_2,m_3\in\{0,1,2\}$
such that \(\Gamma\) contains exactly the first \(m_i\) components on the \(i\)-th branch.  
By symmetry among the three branches, we may renumber the
branches so that
\[
m_1\leq m_2\leq m_3.
\]
More explicitly,
\[
\Gamma
=
C+\sum_{i:m_i\geq1}B_i+\sum_{i:m_i=2}B_i'.
\]
We denote by \(B^{\mathrm{tail}}\) the remaining part of \(R\) supported on the
terminal components of branches with \(m_i=0\). 
More explicitly,
\[
B^{\mathrm{tail}}
\subset
\sum_{i:m_i=0}B_i'.
\]
Thus, in this case, we write
\[
R=\Gamma+B^{\mathrm{tail}}.
\]
\begin{prop}\label{prop:ADE-IVstar-with-branch}
We assume that \(F\) is of Kodaira type \(IV^*\), with notation as above.  Let
\(R\) be an ADER in \(F\).
We assume that
\[
C\subset \operatorname{Supp}(R).
\]
We write $R=\Gamma+B^{\mathrm{tail}}$
as above, where $
\Gamma
=
C+\sum_{i\,:\,m_i\geq1}B_i+\sum_{i\,:\,m_i=2}B_i'$.
After renumbering the three branches if necessary, we may assume that
$m_1\leq m_2\leq m_3$.
Then
\[
(m_1,m_2,m_3)\neq(2,2,2),
\]
and \(\Gamma\) is of type
\[
\begin{cases}
A_{1+m_1+m_2+m_3},
& \text{if at most two of } m_1,m_2,m_3 \text{ are nonzero},\\
D_4,
& \text{if } (m_1,m_2,m_3)=(1,1,1),\\
D_5,
& \text{if } (m_1,m_2,m_3)=(1,1,2),\\
E_6,
& \text{if } (m_1,m_2,m_3)=(1,2,2).
\end{cases}
\]
Moreover, \(B^{\mathrm{tail}}\) is a reduced subdivisor of
\[
\sum_{i\,:\,m_i=0}B_i'.
\]

Conversely, every reduced subdivisor satisfying these conditions is an ADER.
\end{prop}
\begin{proof}
If at most two of \(m_1,m_2,m_3\) are nonzero, then the dual graph of
\(\Gamma\) is a chain.  Hence \(\Gamma\) is of type
\[
A_{1+m_1+m_2+m_3}.
\]

We assume that all three of \(m_1,m_2,m_3\) are nonzero.  Then \(C\) is the unique
branching vertex of \(\Gamma\), and the three branch lengths from \(C\) are
\[
m_1,\qquad m_2,\qquad m_3.
\]
Since $m_1\leq m_2\leq m_3$
and \(m_i\in\{1,2\}\) for all \(i\), the possible triples are
\[
(1,1,1),\quad (1,1,2),\quad (1,2,2),\quad (2,2,2).
\]
The first three triples give respectively the Dynkin diagrams
\[
D_4,\qquad D_5,\qquad E_6.
\]
The last triple \((2,2,2)\) gives the full affine diagram
\(\widetilde E_6\), and hence is not of ADE type.  This gives the stated
possibilities for \(\Gamma\).

It remains to describe the components of \(R\) not contained in \(\Gamma\).
Since \(\Gamma\) is a connected component of \(R\), the vertex immediately
adjacent to \(\Gamma\) on any branch cannot belong to \(R\).  Hence a remaining
component can occur only on a branch with \(m_i=0\), and then the only possible
remaining component on that branch is the terminal component \(B_i'\).  Thus
\(B^{\mathrm{tail}}\) is a sum of some of the components
\[
B_i'\qquad (m_i=0).
\]

Conversely, every reduced subdivisor satisfying the stated conditions is a
disjoint union of \(\Gamma\), which is of ADE type, together with possibly some
isolated \(A_1\)-configurations.  Hence its dual graph is a disjoint union of
Dynkin diagrams of ADE type.  Therefore it is an ADER.
\end{proof}

Before describing the contraction, we introduce notation for the fibers obtained
from a fiber of type \(IV^*\) by contracting an ADER.

\begin{dfn}\label{dfn:IVstar-contracted-fibers}
Let
\[
F=3C+2(B_1+B_2+B_3)+B_1'+B_2'+B_3'=\sum_Em_EE
\]
be a fiber of type \(IV^*\), and let \(R\) be an ADER in \(F\).
Let \(\nu\colon X\rightarrow Y\) be the contraction of \(R\).
For every irreducible component \(E\not\subset R\), we write $\overline E:=\nu(E)$.
Then the resulting fiber is $\overline F=\sum_{E\not\subset R}m_E\overline E$.

\begin{enumerate}[\rm(1)]
\item If $C\not\subset \operatorname{Supp}(R)$,
and $R=R_1+R_2+R_3$
is as in Proposition~\ref{prop:ADE-IVstar-no-branch}, then we denote the
resulting fiber by
\[
\mathfrak{IV}^{*,\mathrm{nb}}(R_1,R_2,R_3).
\]

\item If $C\subset \operatorname{Supp}(R)$, and $R=\Gamma+B^{\mathrm{tail}}$
is as in Proposition~\ref{prop:ADE-IVstar-with-branch}, then we denote the
resulting fiber by
\[
\mathfrak{IV}^{*,\mathrm{br}}(\Gamma;B^{\mathrm{tail}}).
\]
\end{enumerate}
\end{dfn}
In both cases, the notation refers to the fiber obtained from \(F\) by deleting
the components contained in \(R\), while preserving the multiplicities of all
remaining components.  The rational double points are determined as follows:
each connected component of \(R\) is contracted to a rational double point of
the corresponding Dynkin type, and the surviving components adjacent to it pass
through the resulting rational double point.

The explicit lists of the fibers of types
\[
\mathfrak{IV}^{*,\mathrm{nb}}(R_1,R_2,R_3)
\qquad\text{and}\qquad
\mathfrak{IV}^{*,\mathrm{br}}(\Gamma;B^{\mathrm{tail}})
\]
are given in Appendices~\ref{app:IVstar-no-branch-list} and
\ref{app:IVstar-with-branch-list}, respectively.

\begin{prop}\label{prop:IVstar-contraction}
We assume that \(F\) is of Kodaira type \(IV^*\), with notation as above.
Let \(R\) be an ADER in \(F\), and let $\nu\colon X\rightarrow Y$
be the contraction of \(R\).  Then every irreducible component of
\(\overline F\) is a smooth rational curve.  
Moreover, with the notation of
Propositions~\ref{prop:ADE-IVstar-no-branch} and
\ref{prop:ADE-IVstar-with-branch}, \(\overline F\) is of type
\[
\mathfrak{IV}^{*,\mathrm{nb}}(R_1,R_2,R_3)
\]
if \(C\not\subset \operatorname{Supp}(R)\), and of type
\[
\mathfrak{IV}^{*,\mathrm{br}}(\Gamma;B^{\mathrm{tail}})
\]
if \(C\subset \operatorname{Supp}(R)\), in the sense of
Definition~\ref{dfn:IVstar-contracted-fibers}.
\end{prop}
\begin{proof}
The proof is the same as that of Proposition~\ref{prop:IIstar-contraction}.
Since the dual graph of a \(IV^*\)-fiber is a tree, every surviving component
meets each contracted connected component in at most one point, transversely.
Hence Proposition~\ref{prop:smoothness-criterion} shows that the image of every
surviving component is smooth.  The explicit list follows from
Theorem~\ref{2.4} and Definition~\ref{dfn:IVstar-contracted-fibers}.
\end{proof}

\appendix

\section{List of fibers of type
\texorpdfstring{\(\mathfrak{II}^{*,\mathrm{nb}}\)}{II* nb}}
\label{app:IIstar-no-branch-list}

In this appendix, we list the fibers of type
\[
\mathfrak{II}^{*,\mathrm{nb}}(R_S,R_M,R_L).
\]
Here $R=R_S+R_M+R_L$
is as in Proposition~\ref{prop:ADE-IIstar-no-C}. 
Let \(\nu\colon X\rightarrow Y\) be the contraction of \(R\). For an
irreducible component \(D\) of \(F\) with
\(D\not\subset \operatorname{Supp}(R)\), we write $\overline D:=\nu(D)$.
We write $R_L=L_{j_1}+\cdots+L_{j_s}$
with $0\leq s\leq5$ and $1\leq j_1<\cdots<j_s\leq5$,
where the sum is understood to be \(0\) if \(s=0\).
We put 
\[J=\{j_1,\dots,j_s\}\subset\{1,\dots,5\}.\]
When \(s=0\), we put \(J=\emptyset\).
We define
\[
\Lambda_J
:=
\sum_{i\in\{1,\dots,5\}\setminus J}(6-i)\overline{L_i}.
\]
When \(J=\emptyset\), this means
\[
\Lambda_J=
5\overline{L_1}+4\overline{L_2}+3\overline{L_3}
+2\overline{L_4}+\overline{L_5},
\]
and when \(J=\{1,\dots,5\}\), we understand that
\[
\Lambda_J=0.
\]

\subsection*{The case \(R_S=0\)}

\begin{longtable}{c|c|p{0.36\textwidth}|p{0.42\textwidth}}
No.
&
\(\#\operatorname{Irr}(R)\)
&
\(R\)
&
\(\overline F\)
\\
\hline
\endfirsthead

No.
&
\(\#\operatorname{Irr}(R)\)
&
\(R\)
&
\(\overline F\)
\\
\hline
\endhead

\(1\)
&
\(s\)
&
\(L_{j_1}+\cdots+L_{j_s}\quad (1\leq s\leq5)\)
&
\(6\overline C+3\overline{S_1}
+4\overline{M_1}+2\overline{M_2}+\Lambda_J\)
\\
\hline

\(2\)
&
\(1+s\)
&
\(M_1+L_{j_1}+\cdots+L_{j_s}\quad (0\leq s\leq5)\)
&
\(6\overline C+3\overline{S_1}
+2\overline{M_2}+\Lambda_J\)
\\
\hline

\(3\)
&
\(1+s\)
&
\(M_2+L_{j_1}+\cdots+L_{j_s}\quad (0\leq s\leq5)\)
&
\(6\overline C+3\overline{S_1}
+4\overline{M_1}+\Lambda_J\)
\\
\hline

\(4\)
&
\(2+s\)
&
\(M_1+M_2+L_{j_1}+\cdots+L_{j_s}\quad (0\leq s\leq5)\)
&
\(6\overline C+3\overline{S_1}+\Lambda_J\)
\\
\hline

\end{longtable}

\subsection*{The case \(R_S=S_1\)}

\begin{longtable}{c|c|p{0.36\textwidth}|p{0.42\textwidth}}
No.
&
\(\#\operatorname{Irr}(R)\)
&
\(R\)
&
\(\overline F\)
\\
\hline
\endfirsthead

No.
&
\(\#\operatorname{Irr}(R)\)
&
\(R\)
&
\(\overline F\)
\\
\hline
\endhead

\(5\)
&
\(1+s\)
&
\(S_1+L_{j_1}+\cdots+L_{j_s}\quad (0\leq s\leq5)\)
&
\(6\overline C+4\overline{M_1}
+2\overline{M_2}+\Lambda_J\)
\\
\hline

\(6\)
&
\(2+s\)
&
\(S_1+M_1+L_{j_1}+\cdots+L_{j_s}\quad (0\leq s\leq5)\)
&
\(6\overline C+2\overline{M_2}+\Lambda_J\)
\\
\hline

\(7\)
&
\(2+s\)
&
\(S_1+M_2+L_{j_1}+\cdots+L_{j_s}\quad (0\leq s\leq5)\)
&
\(6\overline C+4\overline{M_1}+\Lambda_J\)
\\
\hline

\(8\)
&
\(3+s\)
&
\(S_1+M_1+M_2+L_{j_1}+\cdots+L_{j_s}\quad (0\leq s\leq5)\)
&
\(6\overline C+\Lambda_J\)
\\
\hline
\end{longtable}

In the following table, we record only the unlabeled configuration of
\(\operatorname{Supp}(\overline F)\); in particular, we do not distinguish
whether a branch comes from the \(M\)-side or from the \(L\)-side. Thus two
graphs are identified whenever they are isomorphic as configurations of
irreducible components and their intersection points. The component
\(\overline C\), which appears in every case, is represented by a thin
horizontal line. A solid black line represents an irreducible component
different from \(\overline C\), while a black dot \((\bullet)\) represents
an intersection point of irreducible components.

\newcommand{\IIstarNbFbarGraphRSzeroC}[2]{%
\begin{tikzpicture}[scale=0.50, baseline=-0.5ex]
  \tikzset{
    Ccomp/.style={line width=.3pt},
    comp/.style={line width=1.2pt},
    intpt/.style={circle, fill, inner sep=1.7pt}
  }

  \draw[Ccomp] (-2.65,0)--(2.65,0);

  \draw[comp] (0,0)--(0,-1.05);
  \node[intpt] at (0,0) {};

  \ifnum#1>0
    \draw[comp] (-1,0)--(-1,0.75);
    \node[intpt] at (-1,0) {};
  \fi
  \ifnum#1>1
    \draw[comp] (-1,0.75)--(-1,1.50);
    \node[intpt] at (-1,0.75) {};
  \fi

  \ifnum#2>0
    \draw[comp] (1,0)--(1,0.75);
    \node[intpt] at (1,0) {};
  \fi
  \ifnum#2>1
    \draw[comp] (1,0.75)--(1,1.50);
    \node[intpt] at (1,0.75) {};
  \fi
  \ifnum#2>2
    \draw[comp] (1,1.50)--(1,2.25);
    \node[intpt] at (1,1.50) {};
  \fi
  \ifnum#2>3
    \draw[comp] (1,2.25)--(1,3.00);
    \node[intpt] at (1,2.25) {};
  \fi
  \ifnum#2>4
    \draw[comp] (1,3.00)--(1,3.75);
    \node[intpt] at (1,3.00) {};
  \fi
\end{tikzpicture}%
}

\begin{longtable}{c|c}
Corresponding cases
&
Graph of \(\operatorname{Supp}(\overline F)\)
\\
\hline
\endfirsthead

Corresponding cases
&
Graph of \(\operatorname{Supp}(\overline F)\)
\\
\hline
\endhead

No.~1, \(s=1\)
&
\IIstarNbFbarGraphRSzeroC{2}{4}
\\
\hline

No.~1, \(s=2\)
&
\IIstarNbFbarGraphRSzeroC{2}{3}
\\
\hline

No.~1, \(s=3\)
&
\IIstarNbFbarGraphRSzeroC{2}{2}
\\
\hline

No.~1, \(s=4\);
No.~2--No.~3, \(s=3\)
&
\IIstarNbFbarGraphRSzeroC{1}{2}
\\
\hline

No.~1, \(s=5\);
No.~4, \(s=3\)
&
\IIstarNbFbarGraphRSzeroC{0}{2}
\\
\hline

No.~2--No.~3, \(s=0\)
&
\IIstarNbFbarGraphRSzeroC{1}{5}
\\
\hline

No.~2--No.~3, \(s=1\)
&
\IIstarNbFbarGraphRSzeroC{1}{4}
\\
\hline

No.~2--No.~3, \(s=2\)
&
\IIstarNbFbarGraphRSzeroC{1}{3}
\\
\hline

No.~2--No.~3, \(s=4\)
&
\IIstarNbFbarGraphRSzeroC{1}{1}
\\
\hline

No.~2--No.~3, \(s=5\);
No.~4, \(s=4\)
&
\IIstarNbFbarGraphRSzeroC{0}{1}
\\
\hline

No.~4, \(s=0\)
&
\IIstarNbFbarGraphRSzeroC{0}{5}
\\
\hline

No.~4, \(s=1\)
&
\IIstarNbFbarGraphRSzeroC{0}{4}
\\
\hline

No.~4, \(s=2\)
&
\IIstarNbFbarGraphRSzeroC{0}{3}
\\
\hline

No.~4, \(s=5\)
&
\IIstarNbFbarGraphRSzeroC{0}{0}
\\
\hline

\end{longtable}

When \(R_S=S_1\), the component \(S_1\) is contained in \(R\), and hence it
does not appear in \(\operatorname{Supp}(\overline F)\). In this case,
\(\operatorname{Supp}(\overline F)\) has no branching and no cycle. Therefore
its configuration is a chain, and it is determined only by the number of
irreducible components of \(\operatorname{Supp}(\overline F)\).

\section{List of fibers of type
\texorpdfstring{\(\mathfrak{II}^{*,\mathrm{br}}\)}{II* br}}
\label{app:IIstar-with-branch-list}

In this appendix, we list the fibers of type
\[
\mathfrak{II}^{*,\mathrm{br}}
(\Gamma;R_M^{\mathrm{tail}},R_L^{\mathrm{tail}}).
\]
Here $R=\Gamma+R_M^{\mathrm{tail}}+R_L^{\mathrm{tail}}$
is as in Proposition~\ref{prop:ADE-IIstar-C}. 
Let
\(\nu\colon X\rightarrow Y\) be the contraction of \(R\). 
We write
\(\overline D:=\nu(D)\) for each irreducible component
\(D\not\subset \operatorname{Supp}(R)\) of \(F\).
We write
$\Gamma
=
C+\sum_{i=1}^{\alpha}S_i
+\sum_{i=1}^{\beta}M_i
+\sum_{j=1}^{\gamma}L_j$
where $
\alpha\in\{0,1\},
\beta\in\{0,1,2\}$, and
$\gamma\in\{0,1,2,3,4,5\}$.
Here an empty sum is understood to be \(0\). We put
\[
\Gamma_M:=\sum_{i=1}^{\beta}M_i,
\qquad
\Gamma_L:=\sum_{j=1}^{\gamma}L_j.
\]
Thus
\[
\Gamma=C+\sum_{i=1}^{\alpha}S_i+\Gamma_M+\Gamma_L.
\]

The tails are given as follows. If \(\beta=0\), then
\[
R_M^{\mathrm{tail}}=\varepsilon M_2
\qquad
(\varepsilon\in\{0,1\}).
\]
If \(\beta>0\), then we put
\[
R_M^{\mathrm{tail}}=0.
\]
If \(\gamma\leq3\), then
\[
R_L^{\mathrm{tail}}
=
\sum_{j\in J'}L_j,
\qquad
J'\subset\{\gamma+2,\ldots,5\}.
\]
If \(\gamma\geq4\), then we put
\[
R_L^{\mathrm{tail}}=0.
\]
We put \(t:=\#J'\).
For each row, we put
\[
J:=\{1,\ldots,\gamma\}\cup J',
\]
where \(\{1,\ldots,\gamma\}\) is understood to be empty if \(\gamma=0\).
Thus \(J\subset\{1,\ldots,5\}\) is the set of indices \(j\) such that
\(L_j\subset \Gamma_L+R_L^{\mathrm{tail}}\). We define
\[
\Lambda_J
:=
\sum_{i\in\{1,\dots,5\}\setminus J}(6-i)\overline{L_i}.
\]
When \(J=\emptyset\), this means
\[
\Lambda_J=
5\overline{L_1}+4\overline{L_2}+3\overline{L_3}
+2\overline{L_4}+\overline{L_5},
\]
and when \(J=\{1,\ldots,5\}\), we understand that
\[
\Lambda_J=0.
\]

For convenience, the cases are numbered consecutively as
No.~1, No.~2, \dots.
The corresponding configurations of
\(\operatorname{Supp}(\overline F)\)
are summarized after the list.

\subsection*{The case \(R_S=0\)}

In this case, $\Gamma=C+M_1+\cdots+M_\beta+L_1+\cdots+L_\gamma$.

\begin{longtable}{c|c|p{0.4\textwidth}|p{0.35\textwidth}}
No.
&
\(\#\operatorname{Irr}(R)\)
&
\((\Gamma_M,\Gamma_L,R_M^{\mathrm{tail}},R_L^{\mathrm{tail}})\)
&
\(\overline F\)
\\
\hline
\endfirsthead

No.
&
\(\#\operatorname{Irr}(R)\)
&
\((\Gamma_M,\Gamma_L,R_M^{\mathrm{tail}},R_L^{\mathrm{tail}})\)
&
\(\overline F\)
\\
\hline
\endhead

\(1\)
&
\(1+\varepsilon+t\)
&
\(\displaystyle
(0,0,\varepsilon M_2,\sum_{j\in J'}L_j)
\quad
\bigl(J'\subset\{2,3,4,5\}\bigr)\)
&
\(3\overline{S_1}+4\overline{M_1}
+(2-2\varepsilon)\overline{M_2}+\Lambda_J\)
\\
\hline

\(2\)
&
\(2+\varepsilon+t\)
&
\(\displaystyle
(0,L_1,\varepsilon M_2,\sum_{j\in J'}L_j)
\quad
\bigl(J'\subset\{3,4,5\}\bigr)\)
&
\(3\overline{S_1}+4\overline{M_1}
+(2-2\varepsilon)\overline{M_2}+\Lambda_J\)
\\
\hline

\(3\)
&
\(3+\varepsilon+t\)
&
\(\displaystyle
(0,L_1+L_2,\varepsilon M_2,\sum_{j\in J'}L_j)
\quad
\bigl(J'\subset\{4,5\}\bigr)\)
&
\(3\overline{S_1}+4\overline{M_1}
+(2-2\varepsilon)\overline{M_2}+\Lambda_J\)
\\
\hline

\(4\)
&
\(4+\varepsilon+t\)
&
\(\displaystyle
(0,L_1+L_2+L_3,\varepsilon M_2,\sum_{j\in J'}L_j)
\quad
\bigl(J'\subset\{5\}\bigr)\)
&
\(3\overline{S_1}+4\overline{M_1}
+(2-2\varepsilon)\overline{M_2}+\Lambda_J\)
\\
\hline

\(5\)
&
\(5+\varepsilon\)
&
\((0,L_1+L_2+L_3+L_4,\varepsilon M_2,0)\)
&
\(3\overline{S_1}+4\overline{M_1}
+(2-2\varepsilon)\overline{M_2}+\Lambda_J\)
\\
\hline

\(6\)
&
\(6+\varepsilon\)
&
\((0,L_1+L_2+L_3+L_4+L_5,\varepsilon M_2,0)\)
&
\(3\overline{S_1}+4\overline{M_1}
+(2-2\varepsilon)\overline{M_2}\)
\\
\hline

\(7\)
&
\(2+t\)
&
\(\displaystyle
(M_1,0,0,\sum_{j\in J'}L_j)
\quad
\bigl(J'\subset\{2,3,4,5\}\bigr)\)
&
\(3\overline{S_1}+2\overline{M_2}+\Lambda_J\)
\\
\hline

\(8\)
&
\(3+t\)
&
\(\displaystyle
(M_1,L_1,0,\sum_{j\in J'}L_j)
\quad
\bigl(J'\subset\{3,4,5\}\bigr)\)
&
\(3\overline{S_1}+2\overline{M_2}+\Lambda_J\)
\\
\hline

\(9\)
&
\(4+t\)
&
\(\displaystyle
(M_1,L_1+L_2,0,\sum_{j\in J'}L_j)
\quad
\bigl(J'\subset\{4,5\}\bigr)\)
&
\(3\overline{S_1}+2\overline{M_2}+\Lambda_J\)
\\
\hline

\(10\)
&
\(5+t\)
&
\(\displaystyle
(M_1,L_1+L_2+L_3,0,\sum_{j\in J'}L_j)
\quad
\bigl(J'\subset\{5\}\bigr)\)
&
\(3\overline{S_1}+2\overline{M_2}+\Lambda_J\)
\\
\hline

\(11\)
&
\(6\)
&
\((M_1,L_1+L_2+L_3+L_4,0,0)\)
&
\(3\overline{S_1}+2\overline{M_2}+\Lambda_J\)
\\
\hline

\(12\)
&
\(7\)
&
\((M_1,L_1+L_2+L_3+L_4+L_5,0,0)\)
&
\(3\overline{S_1}+2\overline{M_2}\)
\\
\hline

\(13\)
&
\(3+t\)
&
\(\displaystyle
(M_1+M_2,0,0,\sum_{j\in J'}L_j)
\quad
\bigl(J'\subset\{2,3,4,5\}\bigr)\)
&
\(3\overline{S_1}+\Lambda_J\)
\\
\hline

\(14\)
&
\(4+t\)
&
\(\displaystyle
(M_1+M_2,L_1,0,\sum_{j\in J'}L_j)
\quad
\bigl(J'\subset\{3,4,5\}\bigr)\)
&
\(3\overline{S_1}+\Lambda_J\)
\\
\hline

\(15\)
&
\(5+t\)
&
\(\displaystyle
(M_1+M_2,L_1+L_2,0,\sum_{j\in J'}L_j)
\quad
\bigl(J'\subset\{4,5\}\bigr)\)
&
\(3\overline{S_1}+\Lambda_J\)
\\
\hline

\(16\)
&
\(6+t\)
&
\(\displaystyle
(M_1+M_2,L_1+L_2+L_3,0,\sum_{j\in J'}L_j)
\quad
\bigl(J'\subset\{5\}\bigr)\)
&
\(3\overline{S_1}+\Lambda_J\)
\\
\hline

\(17\)
&
\(7\)
&
\((M_1+M_2,L_1+L_2+L_3+L_4,0,0)\)
&
\(3\overline{S_1}+\Lambda_J\)
\\
\hline

\(18\)
&
\(8\)
&
\((M_1+M_2,L_1+L_2+L_3+L_4+L_5,0,0)\)
&
\(3\overline{S_1}\)
\\
\hline

\end{longtable}

\subsection*{The case \(R_S=S_1\)}

In this case, $\Gamma=C+S_1+M_1+\cdots+M_\beta+L_1+\cdots+L_\gamma$.
The case \((\beta,\gamma)=(2,5)\) is omitted, because then \(\Gamma\) is the
whole affine diagram \(\widetilde E_8\), and hence is not of ADE type.

\begin{longtable}{c|c|p{0.5\textwidth}|p{0.3\textwidth}}
No.
&
\(\#\operatorname{Irr}(R)\)
&
\((\Gamma_M,\Gamma_L,R_M^{\mathrm{tail}},R_L^{\mathrm{tail}})\)
&
\(\overline F\)
\\
\hline
\endfirsthead

No.
&
\(\#\operatorname{Irr}(R)\)
&
\((\Gamma_M,\Gamma_L,R_M^{\mathrm{tail}},R_L^{\mathrm{tail}})\)
&
\(\overline F\)
\\
\hline
\endhead

\(19\)
&
\(2+\varepsilon+t\)
&
\(\displaystyle
(0,0,\varepsilon M_2,\sum_{j\in J'}L_j)
\quad
\bigl(J'\subset\{2,3,4,5\}\bigr)\)
&
\(4\overline{M_1}
+(2-2\varepsilon)\overline{M_2}+\Lambda_J\)
\\
\hline

\(20\)
&
\(3+\varepsilon+t\)
&
\(\displaystyle
(0,L_1,\varepsilon M_2,\sum_{j\in J'}L_j)
\quad
\bigl(J'\subset\{3,4,5\}\bigr)\)
&
\(4\overline{M_1}
+(2-2\varepsilon)\overline{M_2}+\Lambda_J\)
\\
\hline

\(21\)
&
\(4+\varepsilon+t\)
&
\(\displaystyle
(0,L_1+L_2,\varepsilon M_2,\sum_{j\in J'}L_j)
\quad
\bigl(J'\subset\{4,5\}\bigr)\)
&
\(4\overline{M_1}
+(2-2\varepsilon)\overline{M_2}+\Lambda_J\)
\\
\hline

\(22\)
&
\(5+\varepsilon+t\)
&
\(\displaystyle
(0,L_1+L_2+L_3,\varepsilon M_2,\sum_{j\in J'}L_j)
\quad
\bigl(J'\subset\{5\}\bigr)\)
&
\(4\overline{M_1}
+(2-2\varepsilon)\overline{M_2}+\Lambda_J\)
\\
\hline

\(23\)
&
\(6+\varepsilon\)
&
\((0,L_1+L_2+L_3+L_4,\varepsilon M_2,0)\)
&
\(4\overline{M_1}
+(2-2\varepsilon)\overline{M_2}+\Lambda_J\)
\\
\hline

\(24\)
&
\(7+\varepsilon\)
&
\((0,L_1+L_2+L_3+L_4+L_5,\varepsilon M_2,0)\)
&
\(4\overline{M_1}
+(2-2\varepsilon)\overline{M_2}\)
\\
\hline

\(25\)
&
\(3+t\)
&
\(\displaystyle
(M_1,0,0,\sum_{j\in J'}L_j)
\quad
\bigl(J'\subset\{2,3,4,5\}\bigr)\)
&
\(2\overline{M_2}+\Lambda_J\)
\\
\hline

\(26\)
&
\(4+t\)
&
\(\displaystyle
(M_1,L_1,0,\sum_{j\in J'}L_j)
\quad
\bigl(J'\subset\{3,4,5\}\bigr)\)
&
\(2\overline{M_2}+\Lambda_J\)
\\
\hline

\(27\)
&
\(5+t\)
&
\(\displaystyle
(M_1,L_1+L_2,0,\sum_{j\in J'}L_j)
\quad
\bigl(J'\subset\{4,5\}\bigr)\)
&
\(2\overline{M_2}+\Lambda_J\)
\\
\hline

\(28\)
&
\(6+t\)
&
\(\displaystyle
(M_1,L_1+L_2+L_3,0,\sum_{j\in J'}L_j)
\quad
\bigl(J'\subset\{5\}\bigr)\)
&
\(2\overline{M_2}+\Lambda_J\)
\\
\hline

\(29\)
&
\(7\)
&
\((M_1,L_1+L_2+L_3+L_4,0,0)\)
&
\(2\overline{M_2}+\Lambda_J\)
\\
\hline

\(30\)
&
\(8\)
&
\((M_1,L_1+L_2+L_3+L_4+L_5,0,0)\)
&
\(2\overline{M_2}\)
\\
\hline

\(31\)
&
\(4+t\)
&
\(\displaystyle
(M_1+M_2,0,0,\sum_{j\in J'}L_j)
\quad
\bigl(J'\subset\{2,3,4,5\}\bigr)\)
&
\(\Lambda_J\)
\\
\hline

\(32\)
&
\(5+t\)
&
\(\displaystyle
(M_1+M_2,L_1,0,\sum_{j\in J'}L_j)
\quad
\bigl(J'\subset\{3,4,5\}\bigr)\)
&
\(\Lambda_J\)
\\
\hline

\(33\)
&
\(6+t\)
&
\(\displaystyle
(M_1+M_2,L_1+L_2,0,\sum_{j\in J'}L_j)
\quad
\bigl(J'\subset\{4,5\}\bigr)\)
&
\(\Lambda_J\)
\\
\hline

\(34\)
&
\(7+t\)
&
\(\displaystyle
(M_1+M_2,L_1+L_2+L_3,0,\sum_{j\in J'}L_j)
\quad
\bigl(J'\subset\{5\}\bigr)\)
&
\(\Lambda_J\)
\\
\hline

\(35\)
&
\(8\)
&
\((M_1+M_2,L_1+L_2+L_3+L_4,0,0)\)
&
\(\Lambda_J\)
\\
\hline

\end{longtable}

We first consider the case \(R_S=0\).
In this case, the component \(S_1\) appears in \(\operatorname{Supp}(\overline F)\). 
Put \[ \ell:=5-\gamma-t, \]
which is the number of components on the \(L\)-side appearing in \(\operatorname{Supp}(\overline F)\). 
In the following table, we record only the unlabeled configuration of \(\operatorname{Supp}(\overline F)\); in particular, we do not distinguish whether a branch comes from the \(M\)-side or from the \(L\)-side. 
Thus two graphs are identified whenever they are isomorphic as configurations of irreducible components and their intersection points. A solid black line represents an irreducible component, while a black dot \((\bullet)\) represents an intersection point of irreducible components.

\newcommand{\IIstarBrFbarGraphAlphaZero}[2]{%
\begin{tikzpicture}[scale=0.50, baseline=-0.5ex]
  \tikzset{
    comp/.style={line width=1.2pt},
    intpt/.style={circle, fill, inner sep=1.7pt}
  }

  \coordinate (O) at (0,0);

  \coordinate (S) at (0,-1.05);

  \coordinate (Mone) at (-0.75,0.55);
  \coordinate (Mtwo) at (-1.50,1.10);

  \coordinate (Lone)   at (0.75,0.55);
  \coordinate (Ltwo)   at (1.50,1.10);
  \coordinate (Lthree) at (2.25,1.65);
  \coordinate (Lfour)  at (3.00,2.20);
  \coordinate (Lfive)  at (3.75,2.75);

  \draw[comp] (O)--(S);

  \ifnum#1>0
    \draw[comp] (O)--(Mone);
  \fi
  \ifnum#1>1
    \draw[comp] (Mone)--(Mtwo);
  \fi

  \ifnum#2>0
    \draw[comp] (O)--(Lone);
  \fi
  \ifnum#2>1
    \draw[comp] (Lone)--(Ltwo);
  \fi
  \ifnum#2>2
    \draw[comp] (Ltwo)--(Lthree);
  \fi
  \ifnum#2>3
    \draw[comp] (Lthree)--(Lfour);
  \fi
  \ifnum#2>4
    \draw[comp] (Lfour)--(Lfive);
  \fi

  \ifnum#1>0
    \node[intpt] at (O) {};
  \fi
  \ifnum#2>0
    \node[intpt] at (O) {};
  \fi

  \ifnum#1>1
    \node[intpt] at (Mone) {};
  \fi

  \ifnum#2>1
    \node[intpt] at (Lone) {};
  \fi
  \ifnum#2>2
    \node[intpt] at (Ltwo) {};
  \fi
  \ifnum#2>3
    \node[intpt] at (Lthree) {};
  \fi
  \ifnum#2>4
    \node[intpt] at (Lfour) {};
  \fi
\end{tikzpicture}%
}

\begingroup
\small
\setlength{\tabcolsep}{4pt}
\setlength{\LTleft}{0pt}
\setlength{\LTright}{0pt}

\begin{longtable}{@{}
>{\raggedright\arraybackslash}p{0.62\textwidth}|
>{\centering\arraybackslash}p{0.30\textwidth}
@{}}
Corresponding cases
&
Graph of \(\operatorname{Supp}(\overline F)\)
\\
\hline
\endfirsthead

Corresponding cases
&
Graph of \(\operatorname{Supp}(\overline F)\)
\\
\hline
\endhead

No.~1--No.~6, \(\varepsilon=0,\ \ell=5\)
&
\IIstarBrFbarGraphAlphaZero{2}{5}
\\
\hline

No.~1--No.~6, \(\varepsilon=0,\ \ell=4\)
&
\IIstarBrFbarGraphAlphaZero{2}{4}
\\
\hline

No.~1--No.~6, \(\varepsilon=0,\ \ell=3\)
&
\IIstarBrFbarGraphAlphaZero{2}{3}
\\
\hline

No.~1--No.~6, \(\varepsilon=0,\ \ell=2\)
&
\IIstarBrFbarGraphAlphaZero{2}{2}
\\
\hline

No.~1--No.~6, \(\varepsilon=0,\ \ell=1\);
No.~1--No.~6, \(\varepsilon=1,\ \ell=2\);
No.~7--No.~12, \(\ell=2\)
&
\IIstarBrFbarGraphAlphaZero{2}{1}
\\
\hline

No.~1--No.~6, \(\varepsilon=0,\ \ell=0\);
No.~13--No.~18, \(\ell=2\)
&
\IIstarBrFbarGraphAlphaZero{2}{0}
\\
\hline

No.~1--No.~6, \(\varepsilon=1,\ \ell=5\);
No.~7--No.~12, \(\ell=5\)
&
\IIstarBrFbarGraphAlphaZero{1}{5}
\\
\hline

No.~1--No.~6, \(\varepsilon=1,\ \ell=4\);
No.~7--No.~12, \(\ell=4\)
&
\IIstarBrFbarGraphAlphaZero{1}{4}
\\
\hline

No.~1--No.~6, \(\varepsilon=1,\ \ell=3\);
No.~7--No.~12, \(\ell=3\)
&
\IIstarBrFbarGraphAlphaZero{1}{3}
\\
\hline

No.~1--No.~6, \(\varepsilon=1,\ \ell=1\);
No.~7--No.~12, \(\ell=1\)
&
\IIstarBrFbarGraphAlphaZero{1}{1}
\\
\hline

No.~1--No.~6, \(\varepsilon=1,\ \ell=0\);
No.~7--No.~12, \(\ell=0\);
No.~13--No.~18, \(\ell=1\)
&
\IIstarBrFbarGraphAlphaZero{1}{0}
\\
\hline

No.~13--No.~18, \(\ell=5\)
&
\IIstarBrFbarGraphAlphaZero{0}{5}
\\
\hline

No.~13--No.~18, \(\ell=4\)
&
\IIstarBrFbarGraphAlphaZero{0}{4}
\\
\hline

No.~13--No.~18, \(\ell=3\)
&
\IIstarBrFbarGraphAlphaZero{0}{3}
\\
\hline

No.~13--No.~18, \(\ell=0\)
&
\IIstarBrFbarGraphAlphaZero{0}{0}
\\
\hline

\end{longtable}

\endgroup

When \(R_S=S_1\), the component \(S_1\) is contained in \(R\), and hence it
does not appear in \(\operatorname{Supp}(\overline F)\). 
In this case,
\(\operatorname{Supp}(\overline F)\) has no branching and no cycle.  Therefore
its configuration is a chain, and it is determined only by the number of
irreducible components of \(\operatorname{Supp}(\overline F)\).

\section{List of fibers of type
\texorpdfstring{\(\mathfrak{III}^{*,\mathrm{nb}}\)}{III* nb}}
\label{app:IIIstar-no-branch-list}
In this appendix, we list the fibers of type
\[
\mathfrak{III}^{*,\mathrm{nb}}(R_S,R_M,R_L).
\]
Here $R=R_S+R_M+R_L$
is as in Proposition~\ref{prop:ADE-IIIstar-no-C}. 
Let \(\nu\colon X\rightarrow Y\) be the contraction of \(R\). We
write \(\overline D:=\nu(D)\) for each irreducible component
\(D\not\subset \operatorname{Supp}(R)\) of \(F\).
We write
$R_L=L_{j_1}+\cdots+L_{j_s}$
with \(0\leq s\leq3\) and
\(1\leq j_1<\cdots<j_s\leq3\), where the sum is understood to be
\(0\) if \(s=0\). We put
\[
J=\{j_1,\dots,j_s\}\subset\{1,2,3\}.
\]
When \(s=0\), we put \(J=\emptyset\).
We define
\[
\Lambda_J
:=
\sum_{i\in\{1,2,3\}\setminus J}(4-i)\overline{L_i}.
\]
When \(J=\emptyset\), this means
\[
\Lambda_J=3\overline{L_1}+2\overline{L_2}+\overline{L_3},
\]
and when \(J=\{1,2,3\}\), we understand that
\[
\Lambda_J=0.
\]
By the symmetry interchanging the \(M\)-branch and the \(L\)-branch, we may
assume, when listing the possibilities, that
\[
\#\operatorname{Irr}(R_M)\leq \#\operatorname{Irr}(R_L).
\]
The remaining cases are obtained by interchanging the \(M\)- and
\(L\)-branches. Thus, if $R_L=L_{j_1}+\cdots+L_{j_s}$, and $J=\{j_1,\ldots,j_s\}\subset\{1,2,3\}$,
then the range of \(s\) in each row is restricted by this convention.

\subsection*{The case \(R_S=0\)}

\begin{longtable}{c|c|p{0.36\textwidth}|p{0.42\textwidth}}
No.
&
\(\#\operatorname{Irr}(R)\)
&
\(R\)
&
\(\overline F\)
\\
\hline
\endfirsthead

No.
&
\(\#\operatorname{Irr}(R)\)
&
\(R\)
&
\(\overline F\)
\\
\hline
\endhead

\(1\)
&
\(s\)
&
\(L_{j_1}+\cdots+L_{j_s}\quad (1\leq s\leq3)\)
&
\(4\overline C+2\overline{S_1}
+3\overline{M_1}+2\overline{M_2}+\overline{M_3}+\Lambda_J\)
\\
\hline

\(2\)
&
\(1+s\)
&
\(M_1+L_{j_1}+\cdots+L_{j_s}\quad (1\leq s\leq3)\)
&
\(4\overline C+2\overline{S_1}
+2\overline{M_2}+\overline{M_3}+\Lambda_J\)
\\
\hline

\(3\)
&
\(1+s\)
&
\(M_2+L_{j_1}+\cdots+L_{j_s}\quad (1\leq s\leq3)\)
&
\(4\overline C+2\overline{S_1}
+3\overline{M_1}+\overline{M_3}+\Lambda_J\)
\\
\hline

\(4\)
&
\(1+s\)
&
\(M_3+L_{j_1}+\cdots+L_{j_s}\quad (1\leq s\leq3)\)
&
\(4\overline C+2\overline{S_1}
+3\overline{M_1}+2\overline{M_2}+\Lambda_J\)
\\
\hline

\(5\)
&
\(2+s\)
&
\(M_1+M_2+L_{j_1}+\cdots+L_{j_s}\quad (2\leq s\leq3)\)
&
\(4\overline C+2\overline{S_1}
+\overline{M_3}+\Lambda_J\)
\\
\hline

\(6\)
&
\(2+s\)
&
\(M_1+M_3+L_{j_1}+\cdots+L_{j_s}\quad (2\leq s\leq3)\)
&
\(4\overline C+2\overline{S_1}
+2\overline{M_2}+\Lambda_J\)
\\
\hline

\(7\)
&
\(2+s\)
&
\(M_2+M_3+L_{j_1}+\cdots+L_{j_s}\quad (2\leq s\leq3)\)
&
\(4\overline C+2\overline{S_1}
+3\overline{M_1}+\Lambda_J\)
\\
\hline

\(8\)
&
\(6\)
&
\(M_1+M_2+M_3+L_1+L_2+L_3\)
&
\(4\overline C+2\overline{S_1}\)
\\
\hline
\end{longtable}

\subsection*{The case \(R_S=S_1\)}

\begin{longtable}{c|c|p{0.36\textwidth}|p{0.42\textwidth}}
No.
&
\(\#\operatorname{Irr}(R)\)
&
\(R\)
&
\(\overline F\)
\\
\hline
\endfirsthead

No.
&
\(\#\operatorname{Irr}(R)\)
&
\(R\)
&
\(\overline F\)
\\
\hline
\endhead

\(9\)
&
\(1+s\)
&
\(S_1+L_{j_1}+\cdots+L_{j_s}\quad (0\leq s\leq3)\)
&
\(4\overline C
+3\overline{M_1}+2\overline{M_2}+\overline{M_3}+\Lambda_J\)
\\
\hline

\(10\)
&
\(2+s\)
&
\(S_1+M_1+L_{j_1}+\cdots+L_{j_s}\quad (1\leq s\leq3)\)
&
\(4\overline C
+2\overline{M_2}+\overline{M_3}+\Lambda_J\)
\\
\hline

\(11\)
&
\(2+s\)
&
\(S_1+M_2+L_{j_1}+\cdots+L_{j_s}\quad (1\leq s\leq3)\)
&
\(4\overline C
+3\overline{M_1}+\overline{M_3}+\Lambda_J\)
\\
\hline

\(12\)
&
\(2+s\)
&
\(S_1+M_3+L_{j_1}+\cdots+L_{j_s}\quad (1\leq s\leq3)\)
&
\(4\overline C
+3\overline{M_1}+2\overline{M_2}+\Lambda_J\)
\\
\hline

\(13\)
&
\(3+s\)
&
\(S_1+M_1+M_2+L_{j_1}+\cdots+L_{j_s}\quad (2\leq s\leq3)\)
&
\(4\overline C+\overline{M_3}+\Lambda_J\)
\\
\hline

\(14\)
&
\(3+s\)
&
\(S_1+M_1+M_3+L_{j_1}+\cdots+L_{j_s}\quad (2\leq s\leq3)\)
&
\(4\overline C+2\overline{M_2}+\Lambda_J\)
\\
\hline

\(15\)
&
\(3+s\)
&
\(S_1+M_2+M_3+L_{j_1}+\cdots+L_{j_s}\quad (2\leq s\leq3)\)
&
\(4\overline C+3\overline{M_1}+\Lambda_J\)
\\
\hline

\(16\)
&
\(7\)
&
\(S_1+M_1+M_2+M_3+L_1+L_2+L_3\)
&
\(4\overline C\)
\\
\hline

\end{longtable}

The graphs below describe the configuration of
\(\operatorname{Supp}(\overline F)\) in the case \(R_S=0\).
The component \(\overline C\), which appears in every case, is represented by a
thin horizontal line. A solid black line represents an irreducible component
different from \(\overline C\), while a black dot \((\bullet)\) represents an
intersection point of irreducible components.

\newcommand{\IIIstarNbFbarGraphRSzeroC}[2]{%
\begin{tikzpicture}[scale=0.52, baseline=-0.5ex]
  \tikzset{
    Ccomp/.style={line width=.3pt},
    comp/.style={line width=1.2pt},
    intpt/.style={circle, fill, inner sep=1.7pt}
  }

  \draw[Ccomp] (-2.45,0)--(2.45,0);

  \draw[comp] (0,0)--(0,-1.05);
  \node[intpt] at (0,0) {};

  \ifnum#1>0
    \draw[comp] (-1,0)--(-1,0.75);
    \node[intpt] at (-1,0) {};
  \fi
  \ifnum#1>1
    \draw[comp] (-1,0.75)--(-1,1.50);
    \node[intpt] at (-1,0.75) {};
  \fi
  \ifnum#1>2
    \draw[comp] (-1,1.50)--(-1,2.25);
    \node[intpt] at (-1,1.50) {};
  \fi

  \ifnum#2>0
    \draw[comp] (1,0)--(1,0.75);
    \node[intpt] at (1,0) {};
  \fi
  \ifnum#2>1
    \draw[comp] (1,0.75)--(1,1.50);
    \node[intpt] at (1,0.75) {};
  \fi
  \ifnum#2>2
    \draw[comp] (1,1.50)--(1,2.25);
    \node[intpt] at (1,1.50) {};
  \fi
\end{tikzpicture}%
}

\begin{longtable}{c|c}
Corresponding cases
&
Graph of \(\operatorname{Supp}(\overline F)\)
\\
\hline
\endfirsthead

Corresponding cases
&
Graph of \(\operatorname{Supp}(\overline F)\)
\\
\hline
\endhead

No.~1, \(s=1\)
&
\IIIstarNbFbarGraphRSzeroC{3}{2}
\\
\hline

No.~1, \(s=2\)
&
\IIIstarNbFbarGraphRSzeroC{3}{1}
\\
\hline

No.~1, \(s=3\)
&
\IIIstarNbFbarGraphRSzeroC{3}{0}
\\
\hline

No.~2--No.~4, \(s=1\)
&
\IIIstarNbFbarGraphRSzeroC{2}{2}
\\
\hline

No.~2--No.~4, \(s=2\)
&
\IIIstarNbFbarGraphRSzeroC{2}{1}
\\
\hline

No.~2--No.~4, \(s=3\)
&
\IIIstarNbFbarGraphRSzeroC{2}{0}
\\
\hline

No.~5--No.~7, \(s=2\)
&
\IIIstarNbFbarGraphRSzeroC{1}{1}
\\
\hline

No.~5--No.~7, \(s=3\)
&
\IIIstarNbFbarGraphRSzeroC{1}{0}
\\
\hline

No.~8
&
\IIIstarNbFbarGraphRSzeroC{0}{0}
\\
\hline

\end{longtable}

When \(R_S=S_1\), the component \(S_1\) is contained in \(R\), and hence it
does not appear in \(\operatorname{Supp}(\overline F)\). In this case,
\(\operatorname{Supp}(\overline F)\) has no branching and no cycle. Therefore
its configuration is a chain, and it is determined only by the number of
irreducible components of \(\operatorname{Supp}(\overline F)\).

\section{List of fibers of type
\texorpdfstring{\(\mathfrak{III}^{*,\mathrm{br}}\)}{III* br}}
\label{app:IIIstar-with-branch-list}

In this appendix, we list the fibers of type
\[
\mathfrak{III}^{*,\mathrm{br}}
(\Gamma;R_M^{\mathrm{tail}},R_L^{\mathrm{tail}}).
\]
Here $R=\Gamma+R_M^{\mathrm{tail}}+R_L^{\mathrm{tail}}$
is as in Proposition~\ref{prop:ADE-IIIstar-C}. 
Let \(\nu\colon X\rightarrow Y\) be the contraction of \(R\).
We write \(\overline D:=\nu(D)\) for each irreducible component \(D\not\subset \operatorname{Supp}(R)\) of \(F\).
We write $\Gamma=C
+\sum_{i=1}^{\alpha}S_i
+\sum_{i=1}^{\beta}M_i
+\sum_{j=1}^{\gamma}L_j$
where $\alpha\in\{0,1\}$ and $\beta,\gamma\in\{0,1,2,3\}$.
Here an empty sum is understood to be \(0\). We put
\[
R_S:=\sum_{i=1}^{\alpha}S_i,\qquad
\Gamma_M:=\sum_{i=1}^{\beta}M_i,\qquad
\Gamma_L:=\sum_{j=1}^{\gamma}L_j.
\]
Thus
\[
\Gamma=C+R_S+\Gamma_M+\Gamma_L.
\]
By the symmetry interchanging the \(M\)- and \(L\)-branches, we list only the
cases satisfying
\[
\beta\leq\gamma.
\]
The remaining cases are obtained by interchanging the \(M\)- and
\(L\)-branches.

The tails are given as follows. If \(\beta\leq1\), then
\[
R_M^{\mathrm{tail}}
=
\sum_{i\in I}M_i,
\qquad
I\subset\{\beta+2,\ldots,3\}.
\]
If \(\beta\geq2\), then we put
\[
R_M^{\mathrm{tail}}=0.
\]
The same convention applies to \(R_L^{\mathrm{tail}}\), with
\(M,\beta,I\) replaced by \(L,\gamma,J\), respectively.

For convenience, the cases are numbered consecutively as
No.~1, No.~2, \dots.
The corresponding configurations of
\(\operatorname{Supp}(\overline F)\)
are summarized after the list.

\subsection*{The case \(R_S=0\)}

In this case, $\Gamma=C+M_1+\cdots+M_\beta+L_1+\cdots+L_\gamma$.

\begin{longtable}{c|c|p{0.28\textwidth}|p{0.45\textwidth}}
No.
&
\(\#\operatorname{Irr}(R)\)
&
\((\Gamma_M,\Gamma_L,R_M^{\mathrm{tail}},R_L^{\mathrm{tail}})\)
&
\(\overline F\)
\\
\hline
\endfirsthead

No.
&
\(\#\operatorname{Irr}(R)\)
&
\((\Gamma_M,\Gamma_L,R_M^{\mathrm{tail}},R_L^{\mathrm{tail}})\)
&
\(\overline F\)
\\
\hline
\endhead

\(1\)&\(1\)&\((0,0,0,0)\)
&
\(2\overline{S_1}
+3\overline{M_1}+2\overline{M_2}+\overline{M_3}
+3\overline{L_1}+2\overline{L_2}+\overline{L_3}\)
\\
\hline

\(2\)&\(2\)&\((0,0,0,L_2)\)
&
\(2\overline{S_1}
+3\overline{M_1}+2\overline{M_2}+\overline{M_3}
+3\overline{L_1}+\overline{L_3}\)
\\

\(3\)&\(2\)&\((0,0,0,L_3)\)
&
\(2\overline{S_1}
+3\overline{M_1}+2\overline{M_2}+\overline{M_3}
+3\overline{L_1}+2\overline{L_2}\)
\\

\(4\)&\(2\)&\((0,L_1,0,0)\)
&
\(2\overline{S_1}
+3\overline{M_1}+2\overline{M_2}+\overline{M_3}
+2\overline{L_2}+\overline{L_3}\)
\\

\(5\)&\(2\)&\((0,0,M_2,0)\)
&
\(2\overline{S_1}
+3\overline{M_1}+\overline{M_3}
+3\overline{L_1}+2\overline{L_2}+\overline{L_3}\)
\\

\(6\)&\(2\)&\((0,0,M_3,0)\)
&
\(2\overline{S_1}
+3\overline{M_1}+2\overline{M_2}
+3\overline{L_1}+2\overline{L_2}+\overline{L_3}\)
\\
\hline

\(7\)&\(3\)&\((0,0,0,L_2+L_3)\)
&
\(2\overline{S_1}
+3\overline{M_1}+2\overline{M_2}+\overline{M_3}
+3\overline{L_1}\)
\\

\(8\)&\(3\)&\((0,L_1,0,L_3)\)
&
\(2\overline{S_1}
+3\overline{M_1}+2\overline{M_2}+\overline{M_3}
+2\overline{L_2}\)
\\

\(9\)&\(3\)&\((0,L_1+L_2,0,0)\)
&
\(2\overline{S_1}
+3\overline{M_1}+2\overline{M_2}+\overline{M_3}
+\overline{L_3}\)
\\

\(10\)&\(3\)&\((0,0,M_2,L_2)\)
&
\(2\overline{S_1}
+3\overline{M_1}+\overline{M_3}
+3\overline{L_1}+\overline{L_3}\)
\\

\(11\)&\(3\)&\((0,0,M_2,L_3)\)
&
\(2\overline{S_1}
+3\overline{M_1}+\overline{M_3}
+3\overline{L_1}+2\overline{L_2}\)
\\

\(12\)&\(3\)&\((0,0,M_3,L_2)\)
&
\(2\overline{S_1}
+3\overline{M_1}+2\overline{M_2}
+3\overline{L_1}+\overline{L_3}\)
\\

\(13\)&\(3\)&\((0,0,M_3,L_3)\)
&
\(2\overline{S_1}
+3\overline{M_1}+2\overline{M_2}
+3\overline{L_1}+2\overline{L_2}\)
\\

\(14\)&\(3\)&\((0,L_1,M_2,0)\)
&
\(2\overline{S_1}
+3\overline{M_1}+\overline{M_3}
+2\overline{L_2}+\overline{L_3}\)
\\

\(15\)&\(3\)&\((0,L_1,M_3,0)\)
&
\(2\overline{S_1}
+3\overline{M_1}+2\overline{M_2}
+2\overline{L_2}+\overline{L_3}\)
\\

\(16\)&\(3\)&\((0,0,M_2+M_3,0)\)
&
\(2\overline{S_1}
+3\overline{M_1}
+3\overline{L_1}+2\overline{L_2}+\overline{L_3}\)
\\

\(17\)&\(3\)&\((M_1,L_1,0,0)\)
&
\(2\overline{S_1}
+2\overline{M_2}+\overline{M_3}
+2\overline{L_2}+\overline{L_3}\)
\\
\hline

\(18\)&\(4\)&\((0,L_1+L_2+L_3,0,0)\)
&
\(2\overline{S_1}
+3\overline{M_1}+2\overline{M_2}+\overline{M_3}\)
\\

\(19\)&\(4\)&\((0,0,M_2,L_2+L_3)\)
&
\(2\overline{S_1}
+3\overline{M_1}+\overline{M_3}
+3\overline{L_1}\)
\\

\(20\)&\(4\)&\((0,0,M_3,L_2+L_3)\)
&
\(2\overline{S_1}
+3\overline{M_1}+2\overline{M_2}
+3\overline{L_1}\)
\\

\(21\)&\(4\)&\((0,L_1,M_2,L_3)\)
&
\(2\overline{S_1}
+3\overline{M_1}+\overline{M_3}
+2\overline{L_2}\)
\\

\(22\)&\(4\)&\((0,L_1,M_3,L_3)\)
&
\(2\overline{S_1}
+3\overline{M_1}+2\overline{M_2}
+2\overline{L_2}\)
\\

\(23\)&\(4\)&\((0,L_1+L_2,M_2,0)\)
&
\(2\overline{S_1}
+3\overline{M_1}+\overline{M_3}
+\overline{L_3}\)
\\

\(24\)&\(4\)&\((0,L_1+L_2,M_3,0)\)
&
\(2\overline{S_1}
+3\overline{M_1}+2\overline{M_2}
+\overline{L_3}\)
\\

\(25\)&\(4\)&\((0,0,M_2+M_3,L_2)\)
&
\(2\overline{S_1}
+3\overline{M_1}
+3\overline{L_1}+\overline{L_3}\)
\\

\(26\)&\(4\)&\((0,0,M_2+M_3,L_3)\)
&
\(2\overline{S_1}
+3\overline{M_1}
+3\overline{L_1}+2\overline{L_2}\)
\\

\(27\)&\(4\)&\((0,L_1,M_2+M_3,0)\)
&
\(2\overline{S_1}
+3\overline{M_1}
+2\overline{L_2}+\overline{L_3}\)
\\

\(28\)&\(4\)&\((M_1,L_1,0,L_3)\)
&
\(2\overline{S_1}
+2\overline{M_2}+\overline{M_3}
+2\overline{L_2}\)
\\

\(29\)&\(4\)&\((M_1,L_1+L_2,0,0)\)
&
\(2\overline{S_1}
+2\overline{M_2}+\overline{M_3}
+\overline{L_3}\)
\\

\(30\)&\(4\)&\((M_1,L_1,M_3,0)\)
&
\(2\overline{S_1}
+2\overline{M_2}
+2\overline{L_2}+\overline{L_3}\)
\\
\hline

\(31\)&\(5\)&\((0,L_1+L_2+L_3,M_2,0)\)
&
\(2\overline{S_1}
+3\overline{M_1}+\overline{M_3}\)
\\

\(32\)&\(5\)&\((0,L_1+L_2+L_3,M_3,0)\)
&
\(2\overline{S_1}
+3\overline{M_1}+2\overline{M_2}\)
\\

\(33\)&\(5\)&\((0,0,M_2+M_3,L_2+L_3)\)
&
\(2\overline{S_1}
+3\overline{M_1}
+3\overline{L_1}\)
\\

\(34\)&\(5\)&\((0,L_1,M_2+M_3,L_3)\)
&
\(2\overline{S_1}
+3\overline{M_1}
+2\overline{L_2}\)
\\

\(35\)&\(5\)&\((0,L_1+L_2,M_2+M_3,0)\)
&
\(2\overline{S_1}
+3\overline{M_1}
+\overline{L_3}\)
\\

\(36\)&\(5\)&\((M_1,L_1+L_2+L_3,0,0)\)
&
\(2\overline{S_1}
+2\overline{M_2}+\overline{M_3}\)
\\

\(37\)&\(5\)&\((M_1,L_1,M_3,L_3)\)
&
\(2\overline{S_1}
+2\overline{M_2}
+2\overline{L_2}\)
\\

\(38\)&\(5\)&\((M_1,L_1+L_2,M_3,0)\)
&
\(2\overline{S_1}
+2\overline{M_2}
+\overline{L_3}\)
\\

\(39\)&\(5\)&\((M_1+M_2,L_1+L_2,0,0)\)
&
\(2\overline{S_1}
+\overline{M_3}+\overline{L_3}\)
\\
\hline

\(40\)&\(6\)&\((0,L_1+L_2+L_3,M_2+M_3,0)\)
&
\(2\overline{S_1}
+3\overline{M_1}\)
\\

\(41\)&\(6\)&\((M_1,L_1+L_2+L_3,M_3,0)\)
&
\(2\overline{S_1}
+2\overline{M_2}\)
\\

\(42\)&\(6\)&\((M_1+M_2,L_1+L_2+L_3,0,0)\)
&
\(2\overline{S_1}
+\overline{M_3}\)
\\
\hline

\(43\)&\(7\)&\((M_1+M_2+M_3,L_1+L_2+L_3,0,0)\)
&
\(2\overline{S_1}\)
\\
\hline

\end{longtable}

\subsection*{The case \(R_S=S_1\)}

In this case, $\Gamma=C+S_1+M_1+\cdots+M_\beta+L_1+\cdots+L_\gamma$.

\begin{longtable}{c|c|p{0.28\textwidth}|p{0.43\textwidth}}
No.
&
\(\#\operatorname{Irr}(R)\)
&
\((R_M,R_L,R_M^{\mathrm{tail}},R_L^{\mathrm{tail}})\)
&
\(\overline F\)
\\
\hline
\endfirsthead

No.
&
\(\#\operatorname{Irr}(R)\)
&
\((R_M,R_L,R_M^{\mathrm{tail}},R_L^{\mathrm{tail}})\)
&
\(\overline F\)
\\
\hline
\endhead

\(44\)&\(2\)&\((0,0,0,0)\)
&
\(3\overline{M_1}+2\overline{M_2}+\overline{M_3}
+3\overline{L_1}+2\overline{L_2}+\overline{L_3}\)
\\
\hline

\(45\)&\(3\)&\((0,0,0,L_2)\)
&
\(3\overline{M_1}+2\overline{M_2}+\overline{M_3}
+3\overline{L_1}+\overline{L_3}\)
\\

\(46\)&\(3\)&\((0,0,0,L_3)\)
&
\(3\overline{M_1}+2\overline{M_2}+\overline{M_3}
+3\overline{L_1}+2\overline{L_2}\)
\\

\(47\)&\(3\)&\((0,L_1,0,0)\)
&
\(3\overline{M_1}+2\overline{M_2}+\overline{M_3}
+2\overline{L_2}+\overline{L_3}\)
\\

\(48\)&\(3\)&\((0,0,M_2,0)\)
&
\(3\overline{M_1}+\overline{M_3}
+3\overline{L_1}+2\overline{L_2}+\overline{L_3}\)
\\

\(49\)&\(3\)&\((0,0,M_3,0)\)
&
\(3\overline{M_1}+2\overline{M_2}
+3\overline{L_1}+2\overline{L_2}+\overline{L_3}\)
\\
\hline

\(50\)&\(4\)&\((0,0,0,L_2+L_3)\)
&
\(3\overline{M_1}+2\overline{M_2}+\overline{M_3}
+3\overline{L_1}\)
\\

\(51\)&\(4\)&\((0,L_1,0,L_3)\)
&
\(3\overline{M_1}+2\overline{M_2}+\overline{M_3}
+2\overline{L_2}\)
\\

\(52\)&\(4\)&\((0,L_1+L_2,0,0)\)
&
\(3\overline{M_1}+2\overline{M_2}+\overline{M_3}
+\overline{L_3}\)
\\

\(53\)&\(4\)&\((0,0,M_2,L_2)\)
&
\(3\overline{M_1}+\overline{M_3}
+3\overline{L_1}+\overline{L_3}\)
\\

\(54\)&\(4\)&\((0,0,M_2,L_3)\)
&
\(3\overline{M_1}+\overline{M_3}
+3\overline{L_1}+2\overline{L_2}\)
\\

\(55\)&\(4\)&\((0,0,M_3,L_2)\)
&
\(3\overline{M_1}+2\overline{M_2}
+3\overline{L_1}+\overline{L_3}\)
\\

\(56\)&\(4\)&\((0,0,M_3,L_3)\)
&
\(3\overline{M_1}+2\overline{M_2}
+3\overline{L_1}+2\overline{L_2}\)
\\

\(57\)&\(4\)&\((0,L_1,M_2,0)\)
&
\(3\overline{M_1}+\overline{M_3}
+2\overline{L_2}+\overline{L_3}\)
\\

\(58\)&\(4\)&\((0,L_1,M_3,0)\)
&
\(3\overline{M_1}+2\overline{M_2}
+2\overline{L_2}+\overline{L_3}\)
\\

\(59\)&\(4\)&\((0,0,M_2+M_3,0)\)
&
\(3\overline{M_1}
+3\overline{L_1}+2\overline{L_2}+\overline{L_3}\)
\\

\(60\)&\(4\)&\((M_1,L_1,0,0)\)
&
\(2\overline{M_2}+\overline{M_3}
+2\overline{L_2}+\overline{L_3}\)
\\
\hline

\(61\)&\(5\)&\((0,L_1+L_2+L_3,0,0)\)
&
\(3\overline{M_1}+2\overline{M_2}+\overline{M_3}\)
\\

\(62\)&\(5\)&\((0,0,M_2,L_2+L_3)\)
&
\(3\overline{M_1}+\overline{M_3}
+3\overline{L_1}\)
\\

\(63\)&\(5\)&\((0,0,M_3,L_2+L_3)\)
&
\(3\overline{M_1}+2\overline{M_2}
+3\overline{L_1}\)
\\

\(64\)&\(5\)&\((0,L_1,M_2,L_3)\)
&
\(3\overline{M_1}+\overline{M_3}
+2\overline{L_2}\)
\\

\(65\)&\(5\)&\((0,L_1,M_3,L_3)\)
&
\(3\overline{M_1}+2\overline{M_2}
+2\overline{L_2}\)
\\

\(66\)&\(5\)&\((0,L_1+L_2,M_2,0)\)
&
\(3\overline{M_1}+\overline{M_3}
+\overline{L_3}\)
\\

\(67\)&\(5\)&\((0,L_1+L_2,M_3,0)\)
&
\(3\overline{M_1}+2\overline{M_2}
+\overline{L_3}\)
\\

\(68\)&\(5\)&\((0,0,M_2+M_3,L_2)\)
&
\(3\overline{M_1}
+3\overline{L_1}+\overline{L_3}\)
\\

\(69\)&\(5\)&\((0,0,M_2+M_3,L_3)\)
&
\(3\overline{M_1}
+3\overline{L_1}+2\overline{L_2}\)
\\

\(70\)&\(5\)&\((0,L_1,M_2+M_3,0)\)
&
\(3\overline{M_1}
+2\overline{L_2}+\overline{L_3}\)
\\

\(71\)&\(5\)&\((M_1,L_1,0,L_3)\)
&
\(2\overline{M_2}+\overline{M_3}
+2\overline{L_2}\)
\\

\(72\)&\(5\)&\((M_1,L_1+L_2,0,0)\)
&
\(2\overline{M_2}+\overline{M_3}
+\overline{L_3}\)
\\

\(73\)&\(5\)&\((M_1,L_1,M_3,0)\)
&
\(2\overline{M_2}
+2\overline{L_2}+\overline{L_3}\)
\\
\hline

\(74\)&\(6\)&\((0,L_1+L_2+L_3,M_2,0)\)
&
\(3\overline{M_1}+\overline{M_3}\)
\\

\(75\)&\(6\)&\((0,L_1+L_2+L_3,M_3,0)\)
&
\(3\overline{M_1}+2\overline{M_2}\)
\\

\(76\)&\(6\)&\((0,0,M_2+M_3,L_2+L_3)\)
&
\(3\overline{M_1}+3\overline{L_1}\)
\\

\(77\)&\(6\)&\((0,L_1,M_2+M_3,L_3)\)
&
\(3\overline{M_1}+2\overline{L_2}\)
\\

\(78\)&\(6\)&\((0,L_1+L_2,M_2+M_3,0)\)
&
\(3\overline{M_1}+\overline{L_3}\)
\\

\(79\)&\(6\)&\((M_1,L_1+L_2+L_3,0,0)\)
&
\(2\overline{M_2}+\overline{M_3}\)
\\

\(80\)&\(6\)&\((M_1,L_1,M_3,L_3)\)
&
\(2\overline{M_2}+2\overline{L_2}\)
\\

\(81\)&\(6\)&\((M_1,L_1+L_2,M_3,0)\)
&
\(2\overline{M_2}+\overline{L_3}\)
\\

\(82\)&\(6\)&\((M_1+M_2,L_1+L_2,0,0)\)
&
\(\overline{M_3}+\overline{L_3}\)
\\
\hline

\(83\)&\(7\)&\((0,L_1+L_2+L_3,M_2+M_3,0)\)
&
\(3\overline{M_1}\)
\\

\(84\)&\(7\)&\((M_1,L_1+L_2+L_3,M_3,0)\)
&
\(2\overline{M_2}\)
\\

\(85\)&\(7\)&\((M_1+M_2,L_1+L_2+L_3,0,0)\)
&
\(\overline{M_3}\)
\\
\hline

\end{longtable}

The graphs below describe the configuration of
\(\operatorname{Supp}(\overline F)\) in the case \(R_S=0\).
A solid black line represents an irreducible component, while a black dot
\((\bullet)\) represents a point where irreducible components meet.

\newcommand{\IIIstarFbarGraphAlphaZero}[2]{%
\begin{tikzpicture}[scale=0.55, baseline=-0.5ex]
  \tikzset{
    comp/.style={line width=1.2pt},
    intpt/.style={circle, fill, inner sep=1.7pt}
  }

  \coordinate (O) at (0,0);

  \coordinate (S) at (0,-1.05);

  \coordinate (Mone)   at (-0.75,0.55);
  \coordinate (Mtwo)   at (-1.50,1.10);
  \coordinate (Mthree) at (-2.25,1.65);

  \coordinate (Lone)   at (0.75,0.55);
  \coordinate (Ltwo)   at (1.50,1.10);
  \coordinate (Lthree) at (2.25,1.65);

  \draw[comp] (O)--(S);

  \ifnum#1>0
    \draw[comp] (O)--(Mone);
  \fi
  \ifnum#1>1
    \draw[comp] (Mone)--(Mtwo);
  \fi
  \ifnum#1>2
    \draw[comp] (Mtwo)--(Mthree);
  \fi

  \ifnum#2>0
    \draw[comp] (O)--(Lone);
  \fi
  \ifnum#2>1
    \draw[comp] (Lone)--(Ltwo);
  \fi
  \ifnum#2>2
    \draw[comp] (Ltwo)--(Lthree);
  \fi

  \ifnum#1>0
    \node[intpt] at (O) {};
  \fi
  \ifnum#2>0
    \node[intpt] at (O) {};
  \fi

  \ifnum#1>1
    \node[intpt] at (Mone) {};
  \fi
  \ifnum#1>2
    \node[intpt] at (Mtwo) {};
  \fi

  \ifnum#2>1
    \node[intpt] at (Lone) {};
  \fi
  \ifnum#2>2
    \node[intpt] at (Ltwo) {};
  \fi
\end{tikzpicture}%
}

\begin{longtable}{c|c}
Corresponding cases
&
Graph of \(\operatorname{Supp}(\overline F)\)
\\
\hline
\endfirsthead

Corresponding cases
&
Graph of \(\operatorname{Supp}(\overline F)\)
\\
\hline
\endhead

No.~1
&
\IIIstarFbarGraphAlphaZero{3}{3}
\\
\hline

No.~2--No.~6
&
\IIIstarFbarGraphAlphaZero{3}{2}
\\
\hline

No.~7--No.~9, No.~16
&
\IIIstarFbarGraphAlphaZero{3}{1}
\\
\hline

No.~10--No.~15, No.~17
&
\IIIstarFbarGraphAlphaZero{2}{2}
\\
\hline

No.~18
&
\IIIstarFbarGraphAlphaZero{3}{0}
\\
\hline

No.~19--No.~30
&
\IIIstarFbarGraphAlphaZero{2}{1}
\\
\hline

No.~31, No.~32, No.~36
&
\IIIstarFbarGraphAlphaZero{2}{0}
\\
\hline

No.~33--No.~35, No.~37--No.~39
&
\IIIstarFbarGraphAlphaZero{1}{1}
\\
\hline

No.~40--No.~42
&
\IIIstarFbarGraphAlphaZero{1}{0}
\\
\hline

No.~43
&
\IIIstarFbarGraphAlphaZero{0}{0}
\\
\hline

\end{longtable}

When \(R_S=S_1\), the component \(S_1\) is contained in \(R\), and hence it
does not appear in \(\operatorname{Supp}(\overline F)\).
In this case,
\(\operatorname{Supp}(\overline F)\) has no branching and no cycle.  Therefore
the configuration of \(\operatorname{Supp}(\overline F)\) is determined only
by the number of its irreducible components.

\section{List of fibers of type
\texorpdfstring{\(\mathfrak{IV}^{*,\mathrm{nb}}\)}{IV* nb}}
\label{app:IVstar-no-branch-list}
In this appendix, we list the fibers of type 
\[ \mathfrak{IV}^{*,\mathrm{nb}}(R_1,R_2,R_3). \] 
Here $R=R_1+R_2+R_3$ is as in Proposition~\ref{prop:ADE-IVstar-no-branch}. 
Let \(\nu\colon X\rightarrow Y\) be the contraction of \(R\). We
write \(\overline D:=\nu(D)\) for each irreducible component
\(D\not\subset \operatorname{Supp}(R)\) of \(F\).
Since the three branches are symmetric, it is enough to list the triples \((R_1,R_2,R_3)\) up to permutation. We choose representatives so that
\[ \#\operatorname{Irr}(R_1)\leq \#\operatorname{Irr}(R_2)\leq \#\operatorname{Irr}(R_3). \]
For each \(i=1,2,3\), the possible values of \(R_i\) are 
\[ 0,\qquad B_i,\qquad B_i',\qquad B_i+B_i'. \]
We put 
\[ T_i:=B_i+B_i' \]
for each $i=1,2,3$.
\begin{longtable}{c|c|c|c} 
No. & \(\#\operatorname{Irr}(R)\) & \((R_1,R_2,R_3)\) & \(\overline F\) \\ \hline \endfirsthead 
No. & \(\#\operatorname{Irr}(R)\) & \((R_1,R_2,R_3)\) & \(\overline F\) \\ \hline \endhead 
\(1\)&\(1\)&\((0,0,B_3)\) & \(3\overline C+2\overline B_1+\overline B_1' +2\overline B_2+\overline B_2'+\overline B_3'\) \\ 
\(2\)&\(1\)&\((0,0,B_3')\) & \(3\overline C+2\overline B_1+\overline B_1' +2\overline B_2+\overline B_2'+2\overline B_3\) \\ \hline 
\(3\)&\(2\)&\((0,0,T_3)\) & \(3\overline C+2\overline B_1+\overline B_1' +2\overline B_2+\overline B_2'\) \\ 
\(4\)&\(2\)&\((0,B_2,B_3)\) & \(3\overline C+2\overline B_1+\overline B_1' +\overline B_2'+\overline B_3'\) \\ 
\(5\)&\(2\)&\((0,B_2,B_3')\) & \(3\overline C+2\overline B_1+\overline B_1' +\overline B_2'+2\overline B_3\) \\ 
\(6\)&\(2\)&\((0,B_2',B_3')\) & \(3\overline C+2\overline B_1+\overline B_1' +2\overline B_2+2\overline B_3\) \\ \hline 
\(7\)&\(3\)&\((0,B_2,T_3)\) & \(3\overline C+2\overline B_1+\overline B_1' +\overline B_2'\) \\ 
\(8\)&\(3\)&\((0,B_2',T_3)\) & \(3\overline C+2\overline B_1+\overline B_1' +2\overline B_2\) \\ 
\(9\)&\(3\)&\((B_1,B_2,B_3)\) & \(3\overline C+\overline B_1'+\overline B_2'+\overline B_3'\) \\ 
\(10\)&\(3\)&\((B_1,B_2,B_3')\) & \(3\overline C+\overline B_1'+\overline B_2'+2\overline B_3\) \\ 
\(11\)&\(3\)&\((B_1,B_2',B_3')\) & \(3\overline C+\overline B_1'+2\overline B_2+2\overline B_3\) \\ 
\(12\)&\(3\)&\((B_1',B_2',B_3')\) & \(3\overline C+2\overline B_1+2\overline B_2+2\overline B_3\) \\ \hline 
\(13\)&\(4\)&\((0,T_2,T_3)\) & \(3\overline C+2\overline B_1+\overline B_1'\) \\ 
\(14\)&\(4\)&\((B_1,B_2,T_3)\) & \(3\overline C+\overline B_1'+\overline B_2'\) \\ 
\(15\)&\(4\)&\((B_1,B_2',T_3)\) & \(3\overline C+\overline B_1'+2\overline B_2\) \\ 
\(16\)&\(4\)&\((B_1',B_2',T_3)\) & \(3\overline C+2\overline B_1+2\overline B_2\) \\ \hline 
\(17\)&\(5\)&\((B_1,T_2,T_3)\) & \(3\overline C+\overline B_1'\) \\ 
\(18\)&\(5\)&\((B_1',T_2,T_3)\) & \(3\overline C+2\overline B_1\) \\ 
\hline \(19\)&\(6\)&\((T_1,T_2,T_3)\) & \(3\overline C\) \\ \hline 
\end{longtable}

The graphs below describe the configuration of
\(\operatorname{Supp}(\overline F)\).
The component \(\overline C\), which appears in every case, is represented by a thin
horizontal line. A solid black line represents an irreducible component
different from \(\overline C\), while a black dot \((\bullet)\) represents an
intersection point of irreducible components.

\newcommand{\FbarGraphC}[3]{%
\begin{tikzpicture}[scale=0.48, baseline=-0.5ex]
  \tikzset{
    Ccomp/.style={line width=.3pt},
    comp/.style={line width=1.2pt},
    intpt/.style={circle, fill, inner sep=1.8pt}
  }

  \draw[Ccomp] (-2,0)--(2,0);

  \ifnum#1>0
    \draw[comp] (-1,0)--(-1,1);
    \node[intpt] at (-1,0) {};
  \fi
  \ifnum#1>1
    \draw[comp] (-1,1)--(-1,2);
    \node[intpt] at (-1,1) {};
  \fi

  \ifnum#2>0
    \draw[comp] (0,0)--(0,1);
    \node[intpt] at (0,0) {};
  \fi
  \ifnum#2>1
    \draw[comp] (0,1)--(0,2);
    \node[intpt] at (0,1) {};
  \fi

  \ifnum#3>0
    \draw[comp] (1,0)--(1,1);
    \node[intpt] at (1,0) {};
  \fi
  \ifnum#3>1
    \draw[comp] (1,1)--(1,2);
    \node[intpt] at (1,1) {};
  \fi
\end{tikzpicture}%
}

\begin{longtable}{c|c}
Corresponding cases
&
Graph of \(\operatorname{Supp}(\overline F)\)
\\
\hline
\endfirsthead

Corresponding cases
&
Graph of \(\operatorname{Supp}(\overline F)\)
\\
\hline
\endhead

No.~1, No.~2
&
\FbarGraphC{2}{2}{1}
\\
\hline

No.~3
&
\FbarGraphC{2}{2}{0}
\\
\hline

No.~4, No.~5, No.~6
&
\FbarGraphC{2}{1}{1}
\\
\hline

No.~7, No.~8
&
\FbarGraphC{2}{1}{0}
\\
\hline

No.~9, No.~10, No.~11, No.~12
&
\FbarGraphC{1}{1}{1}
\\
\hline

No.~13
&
\FbarGraphC{2}{0}{0}
\\
\hline

No.~14, No.~15, No.~16
&
\FbarGraphC{1}{1}{0}
\\
\hline

No.~17, No.~18
&
\FbarGraphC{1}{0}{0}
\\
\hline

No.~19
&
\FbarGraphC{0}{0}{0}
\\
\hline
\end{longtable}

\section{List of fibers of type
\texorpdfstring{\(\mathfrak{IV}^{*,\mathrm{br}}\)}{IV* br}}
\label{app:IVstar-with-branch-list}

In this appendix, we list the fibers of type \[ \mathfrak{IV}^{*,\mathrm{br}}(\Gamma;B^{\mathrm{tail}}). \]
Here $R=\Gamma+B^{\mathrm{tail}}$ is as in Proposition~\ref{prop:ADE-IVstar-with-branch}. 
Let \(\nu\colon X\rightarrow Y\) be the contraction of \(R\). We
write \(\overline D:=\nu(D)\) for each irreducible component
\(D\not\subset \operatorname{Supp}(R)\) of \(F\).
We write 
\[ \Gamma = C+\sum_{i:m_i\geq1}B_i+\sum_{i:m_i=2}B_i', \qquad m_i\in\{0,1,2\}, \]
with \(m_1\leq m_2\leq m_3\) after renumbering, and 
\[ (m_1,m_2,m_3)\neq(2,2,2). \] 
The remaining tail is a reduced subdivisor of
\[
\sum_{i:m_i=0}B_i'.
\]
Equivalently, for each \(i\) with \(m_i=0\), the component \(B_i'\) may or may
not be contained in \(B^{\mathrm{tail}}\).
For the table, we record the contribution on the \(i\)-th branch by 
\[ Q_i= \begin{cases} 0, & m_i=0,\ B_i'\not\subset B^{\mathrm{tail}},\\ B_i', & m_i=0,\ B_i'\subset B^{\mathrm{tail}},\\ B_i, & m_i=1,\\ T_i, & m_i=2, \end{cases} \qquad T_i:=B_i+B_i' \] 
for each $i=1,2,3$. 
Thus 
\[ R=C+Q_1+Q_2+Q_3. \]
Since the three branches are symmetric, it is enough to list the triples \((Q_1,Q_2,Q_3)\) up to permutation. 
The case \((Q_1,Q_2,Q_3)=(T_1,T_2,T_3)\) is excluded.

\begin{longtable}{c|c|c|c} 
No. & \(\#\operatorname{Irr}(R)\) & \((Q_1,Q_2,Q_3)\) & \(\overline F\) \\ \hline \endfirsthead 
No. & \(\#\operatorname{Irr}(R)\) & \((Q_1,Q_2,Q_3)\) & \(\overline F\) \\ \hline \endhead 
\(1\)&\(1\)&\((0,0,0)\) & \(2\overline B_1+\overline B_1' +2\overline B_2+\overline B_2' +2\overline B_3+\overline B_3'\) \\ \hline \(2\)&\(2\)&\((0,0,B_3)\) & \(2\overline B_1+\overline B_1' +2\overline B_2+\overline B_2' +\overline B_3'\) \\ \(3\)&\(2\)&\((0,0,B_3')\) & \(2\overline B_1+\overline B_1' +2\overline B_2+\overline B_2' +2\overline B_3\) \\ \hline \(4\)&\(3\)&\((0,0,T_3)\) & \(2\overline B_1+\overline B_1' +2\overline B_2+\overline B_2'\) \\ \(5\)&\(3\)&\((0,B_2,B_3)\) & \(2\overline B_1+\overline B_1' +\overline B_2'+\overline B_3'\) \\ \(6\)&\(3\)&\((0,B_2,B_3')\) & \(2\overline B_1+\overline B_1' +\overline B_2'+2\overline B_3\) \\ \(7\)&\(3\)&\((0,B_2',B_3')\) & \(2\overline B_1+\overline B_1' +2\overline B_2+2\overline B_3\) \\ \hline \(8\)&\(4\)&\((0,B_2,T_3)\) & \(2\overline B_1+\overline B_1' +\overline B_2'\) \\ 
\(9\)&\(4\)&\((0,B_2',T_3)\) & \(2\overline B_1+\overline B_1' +2\overline B_2\) \\ 
\(10\)&\(4\)&\((B_1,B_2,B_3)\) & \(\overline B_1'+\overline B_2'+\overline B_3'\) \\ \(11\)&\(4\)&\((B_1,B_2,B_3')\) & \(\overline B_1'+\overline B_2'+2\overline B_3\) \\ \(12\)&\(4\)&\((B_1,B_2',B_3')\) & \(\overline B_1'+2\overline B_2+2\overline B_3\) \\ \(13\)&\(4\)&\((B_1',B_2',B_3')\) & \(2\overline B_1+2\overline B_2+2\overline B_3\) \\ \hline \(14\)&\(5\)&\((0,T_2,T_3)\) & \(2\overline B_1+\overline B_1'\) \\ 
\(15\)&\(5\)&\((B_1,B_2,T_3)\) & \(\overline B_1'+\overline B_2'\) \\ 
\(16\)&\(5\)&\((B_1,B_2',T_3)\) & \(\overline B_1'+2\overline B_2\) \\ 
\(17\)&\(5\)&\((B_1',B_2',T_3)\) & \(2\overline B_1+2\overline B_2\) \\ \hline 
\(18\)&\(6\)&\((B_1,T_2,T_3)\) & \(\overline B_1'\) \\ 
\(19\)&\(6\)&\((B_1',T_2,T_3)\) & \(2\overline B_1\) \\ \hline 
\end{longtable} 

The graphs below describe the configuration of \(\operatorname{Supp}(\overline F)\). 
A solid black line represents an irreducible component, while a black dot
\((\bullet)\) represents a point where irreducible components meet.

\newcommand{\FbarGraph}[3]{%
\begin{tikzpicture}[scale=0.42, baseline=-0.5ex] 
\tikzset{ comp/.style={line width=1.1pt}, intpt/.style={circle, fill, inner sep=1.6pt} } 
\ifnum#1>0 \draw[comp] (0,0)--(0,1); 
\fi \ifnum#1>1 \draw[comp] (0,1)--(0,2); 
\node[intpt] at (0,1) {}; 
\fi \ifnum#2>0 
\draw[comp] (0,0)--(-0.866,-0.5); \fi \ifnum#2>1 \draw[comp] (-0.866,-0.5)--(-1.732,-1); \node[intpt] at (-0.866,-0.5) {}; \fi \ifnum#3>0 \draw[comp] (0,0)--(0.866,-0.5); \fi \ifnum#3>1 \draw[comp] (0.866,-0.5)--(1.732,-1); \node[intpt] at (0.866,-0.5) {}; \fi \ifnum#1>0 \ifnum#2>0 \node[intpt] at (0,0) {}; \fi \fi \ifnum#1>0 \ifnum#3>0 \node[intpt] at (0,0) {}; \fi \fi \ifnum#2>0 \ifnum#3>0 \node[intpt] at (0,0) {}; \fi \fi 
\end{tikzpicture}%
} 
\begin{longtable}{c|c} 
Corresponding cases & Graph of \(\operatorname{Supp}(\overline F)\) \\ \hline \endfirsthead 
Corresponding cases & Graph of \(\operatorname{Supp}(\overline F)\) \\ \hline \endhead 
No.~1 & \FbarGraph{2}{2}{2} \\ 
\hline 
No.~2, No.~3 & \FbarGraph{2}{2}{1} \\ \hline No.~4 & \FbarGraph{2}{2}{0} \\ \hline No.~5, No.~6, No.~7 & \FbarGraph{2}{1}{1} \\ \hline No.~8, No.~9 & \FbarGraph{2}{1}{0} \\ \hline No.~10, No.~11, No.~12, No.~13 & \FbarGraph{1}{1}{1} \\ 
\hline 
No.~14, No.~15, No.~16, No.~17 & \FbarGraph{2}{0}{0} \\
\hline 
No.~18, No.~19 & \FbarGraph{1}{0}{0} \\ 
\hline 
\end{longtable}

\section*{Declarations}

\subsection*{Conflict of interest}
The authors declare that they have no conflict of interest.

\subsection*{Funding}
No funding was received for conducting this study.

\subsection*{Data availability}
Data sharing is not applicable to this article as no datasets were generated or analyzed during the current study.

\subsection*{Author contributions}
All authors contributed to the conception and design of the study. 
The first draft of the manuscript was written by Taro Hayashi, and all authors commented on previous versions of the manuscript. 
All authors read and approved the final manuscript.


\begin{thebibliography}{99}
\bibitem{a02}
E. Artal-Bartolo, H. Tokunaga and D.-Q. Zhang, Miranda-Persson's problem on extremal elliptic $K3$ surfaces, Pacific J. Math. {\bf 202} (2002), no.~1, 37--72; MR1883969

\bibitem{a62}
M. Artin, Some numerical criteria for contractability of curves on algebraic surfaces, Amer. J. Math. {\bf 84} (1962), 485--496; MR0146182

\bibitem{a66}
M. Artin, On isolated rational singularities of surfaces, Amer. J. Math. {\bf 88} (1966), 129--136; MR0199191

\bibitem{ba18}
F. Balestrieri et al., Elliptic fibrations on covers of the elliptic modular surface of level 5, in {\it Women in numbers Europe II}, 159--197, Assoc. Women Math. Ser., 11, Springer, Cham, 2018 ; MR3882710

\bibitem{be15}
M.-J. Bertin et al., Classifications of elliptic fibrations of a singular K3 surface, in {\it Women in numbers Europe}, 17--49, Assoc. Women Math. Ser., 2, Springer, Cham, 2015 ; MR3596600

\bibitem{cko03}
F. Catanese, J.~H. Keum and K. Oguiso, Some remarks on the universal cover of an open $K3$ surface, Math. Ann. {\bf 325} (2003), no.~2, 279--286; MR1962049

\bibitem{c25}
P. Comparin et al., On strictly elliptic K3 surfaces and del Pezzo surfaces, Rev. Mat. Iberoam. {\bf 41} (2025), no.~4, 1489--1512; MR4912927


\bibitem{gs19}
A. Garbagnati and C. Salgado, Linear systems on rational elliptic surfaces and elliptic fibrations on K3 surfaces, J. Pure Appl. Algebra {\bf 223} (2019), no.~1, 277--300; MR3833460

\bibitem{gs20}
A. Garbagnati and C. Salgado, Elliptic fibrations on K3 surfaces with a non-symplectic involution fixing rational curves and a curve of positive genus, Rev. Mat. Iberoam. {\bf 36} (2020), no.~4, 1167--1206; MR4130832

\bibitem{h24}
T. Hayashi, Non-symplectic involutions of normal K3 surfaces associated with Hirzebruch surfaces, Geom. Dedicata {\bf 218} (2024), no.~6, Paper No. 112, 13 pp.; MR4813249

\bibitem{h26}
T. Hayashi, Orders of purely non-symplectic automorphisms of normal K3 surfaces, Bull. Malays. Math. Sci. Soc. {\bf 49} (2026), no.~1, Paper No. 25, 13 pp.; MR5006600

\bibitem{k06}
R.~N. Kloosterman, Classification of all Jacobian elliptic fibrations on certain $K3$ surfaces, J. Math. Soc. Japan {\bf 58} (2006), no.~3, 665--680; MR2254405

\bibitem{k63}
K. Kodaira, On compact analytic surfaces. II, Ann. of Math. (2) {\bf 77} (1963), 563--626; MR0184257


\bibitem{ku14}
A. Kumar, Elliptic fibrations on a generic Jacobian Kummer surface, J. Algebraic Geom. {\bf 23} (2014), no.~4, 599--667; MR3263663

\bibitem{ku18}
M. Kuwata and K. Utsumi, Mordell-Weil lattice of Inose's elliptic $K3$ surface arising from the product of 3-isogenous elliptic curves, J. Number Theory {\bf 190} (2018), 333--351; MR3805462


\bibitem{mp89}
R. Miranda and U.~A. Persson, Configurations of ${\rm I}_n$ fibers on elliptic $K3$ surfaces, Math. Z. {\bf 201} (1989), no.~3, 339--361; MR0999732


\bibitem{n96}
K. Nishiyama, The Jacobian fibrations on some $K3$ surfaces and their Mordell-Weil groups, Japan. J. Math. (N.S.) {\bf 22} (1996), no.~2, 293--347; MR1432379

\bibitem{o89}
K. Oguiso, On Jacobian fibrations on the Kummer surfaces of the product of nonisogenous elliptic curves, J. Math. Soc. Japan {\bf 41} (1989), no.~4, 651--680; MR1013073

\bibitem{schs13}
M. Sch\"utt and A. Schweizer, On the uniqueness of elliptic K3 surfaces with maximal singular fibre, Ann. Inst. Fourier (Grenoble) {\bf 63} (2013), no.~2, 689--713; MR3112845

\bibitem{s00}
I. Shimada, On elliptic $K3$ surfaces, Michigan Math. J. {\bf 47} (2000), no.~3, 423--446; MR1813537

\bibitem{sz01}
I. Shimada and D.-Q. Zhang, Classification of extremal elliptic $K3$ surfaces and fundamental groups of open $K3$ surfaces, Nagoya Math. J. {\bf 161} (2001), 23--54; MR1820211

\bibitem{s07}
I. Shimada, On normal $K3$ surfaces, Michigan Math. J. {\bf 55} (2007), no.~2, 395--416; MR2369942


\bibitem{s03}
T. Shioda, The elliptic $K3$ surfaces with a maximal singular fibre, C. R. Math. Acad. Sci. Paris {\bf 337} (2003), no.~7, 461--466; MR2023754

\bibitem{si77}
T. Shioda and H. Inose, On singular $K3$ surfaces, Complex analysis and algebraic geometry, 119--136, Iwanami Shoten Publishers, Tokyo, 1977; MR0441982


\bibitem{sch07}
M. Sch\"utt, The maximal singular fibres of elliptic $K3$ surfaces, Arch. Math. (Basel) {\bf 87} (2006), no.~4, 309--319; MR2263477

\bibitem{schs10}
M. Sch\"{u}tt and T. Shioda, Elliptic surfaces, In Algebraic geometry in East Asia—Seoul 2008, 51–160, Adv. Stud. Pure Math., {\bf 60}, Math. Soc. Japan, Tokyo (2010) ; MR2732092

\bibitem{u10}
K. Utsumi, On the structure of certain $K3$ surfaces, Comment. Math. Univ. St. Pauli {\bf 59} (2010), no.~1, 29--49; MR2730091

\bibitem{u12}
K. Utsumi, Weierstrass equations for Jacobian fibrations on a certain $K3$ surface, Hiroshima Math. J. {\bf 42} (2012), no.~3, 355--383; MR3050126

\bibitem{u16}
K. Utsumi, Jacobian fibrations on the singular $K3$ surface of discriminant $3$, J. Math. Soc. Japan {\bf 68} (2016), no.~3, 1133--1146; MR3523541

\bibitem{u23}
K. Utsumi, The Mordell-Weil lattice of an Inose surface arising from isogenous elliptic curves, Comment. Math. Univ. St. Pauli {\bf 71} (2023), 25--35; MR4762180

\end{thebibliography}
\end{document}